\documentclass[a4paper]{article}

\usepackage[hidelinks]{hyperref}
\usepackage[numbers]{natbib}
\usepackage{mathtools}
\usepackage{amsthm}
\usepackage{amsfonts}
\usepackage{amssymb}
\usepackage{stmaryrd} 
\usepackage{dsfont}
\numberwithin{equation}{section}
\numberwithin{figure}{section}

\usepackage[capitalize]{cleveref}

\usepackage{tikz}
\tikzset{whitenoise/.style={radius=0.07, fill=black}}
\tikzset{contract/.style={dashed}}

\renewcommand{\epsilon}{\varepsilon}

\newcommand{\halfspace}{\mathcal D}
\newcommand{\goodregion}{\mathcal B}
\newcommand{\badregion}{\mathcal R}
\newcommand{\pf}{\mathcal Z} 
\newcommand{\Gaussian}{\rho} 
\newcommand{\negpart}[1]{\llbracket #1 \rrbracket}
\newcommand{\besovinfty}{\mathcal C}

\DeclareMathOperator{\Prob}{\mathbb P}
\DeclareMathOperator{\E}{\mathbb E}
\newcommand{\I}{\mathds 1}
\newcommand{\N}{\mathbb N}
\newcommand{\R}{\mathbb R}
\newcommand{\Z}{\mathbb Z}
\newcommand{\W}{\mathbb W}
\newcommand{\Tor}{\mathbb T}
\newcommand{\Laplace}{\Delta}
\newcommand{\inv}{^{-1}}
\newcommand{\abs}[1]{{\left|#1\right|}}
\newcommand{\smallabs}[1]{|#1|}
\newcommand{\norm}[1]{\left\|#1\right\|}
\newcommand{\smallnorm}[1]{\|#1\|}
\newcommand{\inorm}[1]{{\left\langle{#1}\right\rangle}}
\newcommand{\inormm}[1]{{\left\langle{#1}\right\rangle}_-}
\newcommand{\inormp}[1]{{\left\langle{#1}\right\rangle}_+}
\newcommand{\inormc}[1]{{\left\langle{#1}\right\rangle}_\bullet}

\newcommand{\diff}{\mathop{}\!\mathrm{d}}
\newcommand{\dx}{\diff x}
\newcommand{\ds}{\diff s}
\newcommand{\dt}{\diff t}
\newcommand{\dr}{\diff r}
\DeclareMathOperator{\proj}{P}
\DeclareMathOperator{\arcsinh}{arcsinh}
\DeclareMathOperator{\Law}{Law}
\newcommand{\wick}[1]{\mathinner{:\!#1\!:}}
\newcommand{\wickp}[1]{\mathinner{:\!#1\!:_+}}
\newcommand{\wickm}[1]{\mathinner{:\!#1\!:_-}}
\newcommand{\paral}{\prec}
\newcommand{\parar}{\succ}
\newcommand{\reson}{\odot}
\newcommand{\bigO}{\mathcal O}

\newtheorem{theorem}{Theorem}[section]
\newtheorem{lemma}[theorem]{Lemma}
\newtheorem{corollary}[theorem]{Corollary}
\theoremstyle{definition}
\newtheorem{remark}[theorem]{Remark}
\newtheorem{example}[theorem]{Example}
\newtheorem{definition}[theorem]{Definition}
\newtheorem{assumption}[theorem]{Assumption}
\crefname{assumption}{assumption}{assumptions}

\title{Eyring--Kramers law for the hyperbolic $\phi^4$ model}
\author{Nikolay Barashkov%
\footnote{\url{nikolay.barashkov@mis.mpg.de};
Max Planck Institute for Mathematics in the Sciences, Leipzig, Germany},
Petri Laarne%
\footnote{\url{petri.laarne@helsinki.fi};
Department of Mathematics and Statistics, University of Helsinki, Finland}}
\date{} 

\begin{document}

\maketitle

\begin{abstract}

\noindent
We study the expected transition frequency between the two metastable states
of a stochastic wave equation with double-well potential.
By transition state theory, the frequency factorizes into two components:
one depends only on the invariant measure, given by the $\phi^4_d$ quantum field theory,
and the other takes the dynamics into account.
We compute the first component with the variational approach to stochastic quantization when $d = 2, 3$.
For the two-dimensional equation with random data but no stochastic forcing,
we also compute the transmission coefficient.


\end{abstract}


%
%
%
\section{Introduction}\label{sec:intro}

Consider the \emph{stochastic damped nonlinear wave equation}
\begin{equation}\label{eq:SdNLW}
\partial_{tt} u(x, t) + \gamma \partial_t u(x,t) - (m^2 + \Laplace) u(x, t)
    = - u(x, t)^3 + \sqrt{\frac{2\gamma}{\beta}} \xi(x,t),
\tag{SdNLW}
\end{equation}
where $\xi$ is the space-time white noise,
$\beta > 0$ is the inverse temperature, $\gamma > 0$ is a damping parameter, 
and $m^2 > 0$ is a mass term.
The $-u^3 + m^2 u$ terms mean that there are symmetric potential wells at $\pm m$.

Since there is stochastic forcing,
a solution $u$ will stay most of the time near either constant function $\pm m$,
but will occasionally jump between the two wells.
The question is: how often will such a jump happen?

This question was originally posed in the context of chemical reaction rates
and for the Langevin dynamics,
which are the finite-dimensional analogue of \eqref{eq:SdNLW}.
The \emph{Eyring--Kramers law},
named after \cite{eyring_activated_1935,kramers_brownian_1940},
states that the expected transition time
is proportional to $\exp(\beta h)$,
where $h$ is the height of the potential barrier (see \Cref{fig:intro potential}).
Our concern is to understand this relation in the low-temperature regime for this particular PDE.

As $\gamma$ approaches either $0$ or $\infty$, \eqref{eq:SdNLW} gives different equations in the limit.
In the overdamped limit $\gamma \to \infty$, one obtains after a suitable rescaling of time
the \emph{stochastic heat equation}
\begin{equation} \label{eq:SHE}
    \partial_t u(x, t) - (m^2 + \Laplace) u(x, t) = - u(x, t)^3 + \sqrt{\frac{2}{\beta}} \xi(x, t),
    \tag{SHE}
\end{equation}
which is the parabolic counterpart of \eqref{eq:SdNLW}.
Recently this equation has been of interest in constructive quantum field theory,
as the $\phi^4$ measure can recovered as the invariant measure of \eqref{eq:SHE}.
When the mass term has negative sign as in here,
this equation is also called the \emph{stochastic Allen--Cahn equation}.
Metastable behaviour for overdamped Langevin dynamics has been extensively studied in the literature;
see Section \ref{sec:intro previous} for an overview. 

Conversely taking $\gamma \to 0$ one obtains the \emph{nonlinear wave equation}
\begin{equation}\label{eq:NLW}
  \partial_{tt} u(x, t) - (m^2 + \Laplace) u(x, t) = - u(x, t)^3.
    \tag{NLW}
\end{equation}
This equation lacks both the stochastic forcing and the $\partial_t u$ damping term.
Yet if the initial data is sampled from the invariant measure of \eqref{eq:SdNLW},
the measure is again invariant under the dynamics and
one can still ask the same question about metastability.
In comparison to the overdamped case, the Langevin dynamics corresponding to \eqref{eq:SdNLW}
and its underdamped limit \eqref{eq:NLW}
have received little attention, especially in the PDE setting.  
Our aim in this article is to make a step in that direction.

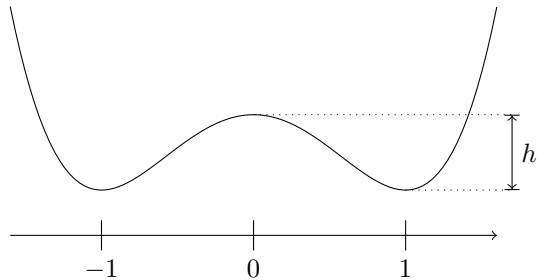
\begin{figure}
\centering
\begin{tikzpicture}[xscale=2, yscale=4]
    \draw[domain=-1.6:1.6, samples=100] plot (\x, {pow(\x,4)/4 - pow(\x,2)/2});
    \draw[->] (-1.6,-0.4) -- (1.6,-0.4);
    \foreach \x in {-1,0,1}
        {\draw (\x,-0.35) -- (\x,-0.45) node[below] {$\x$};}
    \draw[dotted] (0,0) -- (1.7,0);
    \draw[dotted] (1,-0.25) -- (1.7,-0.25);
    \draw[|<->|] (1.7,0) -- (1.7,-0.25) node[pos=0.5,right] {$h$};
\end{tikzpicture}
\caption{The $u^4/4 - u^2/2$ double-well potential in one-dimensional state space.
The height $h$ of the potential barrier is marked.}
\label{fig:intro potential}
\end{figure}

The $d$-dimensional $\phi^4_d$ measure is formally defined as the Gibbs measure
\begin{equation}
\pf\inv \exp\left( -\frac \beta 4 \int_{\Tor^d} \phi(x)^4 \dx \right) \diff\Gaussian(\phi),
\end{equation}
where $\Gaussian$ is a Gaussian measure with (formal) covariance $\beta\inv (-m^2 - \Laplace)\inv$,
and the normalization constant $\pf\inv$ is also called the partition function.

The solutions of \eqref{eq:SdNLW} and \eqref{eq:NLW}
are determined by the pair $(u, \partial_t u)$;
hence their stationary measure is the product of $\phi^4$ and white noise measures.

\begin{remark}
For the rest of the article, we fix $m^2 = 1$.
\end{remark}

\subsection{Previous literature}\label{sec:intro previous}

The Eyring--Kramers law has its roots in the work of van~'t~Hoff and Arrhenius in the 1880s.
Some of the early history of the Eyring--Kramers law in chemistry and physics
is collected e.g.\ in the survey \cite{hanggi_reactionrate_1990}.
Mathematical study of metastability of diffusion processes
goes back to the work of Freidlin and Wentzell;
see e.g.\ \cite{freidlin_random_1984}.

Much of the recent progress in the field has built on
the 2004~article \cite{bovier_metastability_2004} by Bovier, Eckhoff, Gayrard, and Klein.
In it and accompanying articles, the authors developed a potential-theoretic approach
to computing expected transition times of finite-dimensional diffusion processes.

The potential-theoretic method is described in the textbook \cite{bovier_metastability_2015},
and along some other methods in the survey \cite{berglund_kramers_2013}.
The derivation of the result is also further discussed
in e.g.\ the recent article \cite{avelin_geometric_2023} by Avelin, Julin, and Viitasaari.

This approach can be extended to infinite dimensions and thus parabolic PDEs.
The Eyring--Kramers law for \eqref{eq:SHE} in one spatial dimension
and either periodic or Neumann boundary conditions
was derived by Berglund and Gentz in their 2013 article:

\begin{theorem}[{{\cite[Theorem~2.6]{berglund_sharp_2013}}}]
\label{thm:intro she 1d}
Let us consider the flow of \eqref{eq:SHE} on periodic domain $\Tor \coloneqq [0,L]$, where $L < 2\pi$,
and assume small fixed $\delta > 0$ and sufficiently large $\beta > 0$.
We define $\mathcal B_+ \coloneqq \{ u \colon \norm{u - 1}_{L^\infty} \leq \delta \}$
and $\mathcal B_- \coloneqq \{ u \colon \norm{u + 1}_{L^\infty} \leq \delta \}$.
For initial data $u_0 \in \mathcal B_+$,
the expected first hitting time of $\mathcal B_-$ is then
\[
\sqrt 2 \pi \prod_{k \in \Z} \sqrt{\frac{\abs{C_L k^2 - 1}}{C_L k^2 + 2}}
    \exp(-\beta/4) (1 + \bigO(\beta^{-1/2} \abs{\log\beta}^{3/2})).
\]
Here $C_L \coloneqq (2\pi/L)^2$.
\end{theorem}

\begin{figure}
\centering
\begin{tikzpicture}[xscale=2, yscale=4]
    \fill[black!20, domain=-1.2:-0.8, samples=20] plot (\x, {pow(\x,4)/4 - pow(\x,2)/2})
        -- (-0.8, 0.3584) -- (-1.2, 0.3584) -- cycle;
    \fill[black!20, domain=0.8:1.2, samples=20] plot (\x, {pow(\x,4)/4 - pow(\x,2)/2})
        -- (1.2, 0.3584) -- (0.8, 0.3584) -- cycle;
    \node at (-1, 0.42) {$\mathcal B_-$};
    \node at (1, 0.42) {$\mathcal B_+$};
    \draw[domain=-1.6:1.6, samples=100] plot (\x, {pow(\x,4)/4 - pow(\x,2)/2});
    \draw[->] (-1.6,-0.4) -- (1.6,-0.4) node[right] {$\hat u(0)$};
    \foreach \x in {-1,0,1}
        {\draw (\x,-0.35) -- (\x,-0.45) node[below] {$\x$};}
    \foreach \x/\y in {-1.2/-0.2016, -0.8/-0.2176, 0.8/-0.2176, 1.2/-0.2016}
        {\draw[dashed] (\x, 0.3584) -- (\x,\y);}
\end{tikzpicture}
\caption{To prove \Cref{thm:intro she 1d}, one estimates the capacity between the shaded sets.
The measure concentrates in these sets as $\beta \to \infty$.
The potential is shown with respect to constant functions;
oscillatory modes are not pictured.}
\label{fig:intro parabolic}
\end{figure}
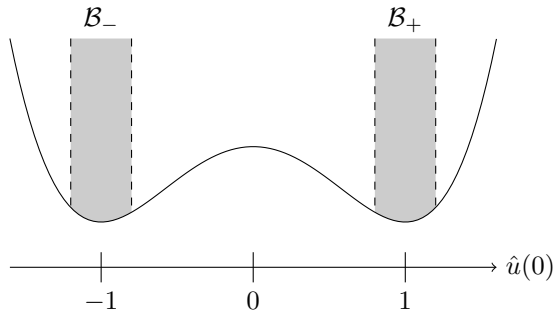

The assumption $L < 2\pi$ simplifies many arguments:
it ensures that the constant function $0$ is the relevant saddle point between the two potential wells.
When $L \geq 2\pi$, the deterministic equation has additional stationary states,
all consisting of kink-antikink pairs (see \cite[Figure~2]{berglund_sharp_2013}).
This changes the metastable behaviour,
and has also been studied in \cite{berglund_sharp_2013}.

\bigskip\noindent
The hyperbolic equation \eqref{eq:NLW}
was studied at around the same time by Newhall and Vanden-Eijnden \cite{newhall_metastability_2017}.
Their work builds on a different definition of metastable transitions.
The so-called \emph{transition state theory} (TST) counts the frequency at which the solution
crosses the $\{ \hat u(0) = 0 \}$ hypersurface (see \Cref{fig:intro hyperbolic}).

TST is closer to the original theory of Eyring \cite{eyring_activated_1935}.
The potential-theoretic approach does not lend itself to hyperbolic equations,
since it assumes both reversibility and strict ellipticity of the associated Markov diffusion process.

Newhall and Vanden-Eijnden consider only \eqref{eq:NLW},
but their argument also applies to \eqref{eq:SdNLW}.
They also state their results for generic boundary conditions.
We summarize the proof of the following theorem in \Cref{sec:tst}.

\begin{figure}
\centering
\begin{tikzpicture}[xscale=2, yscale=4]
    \fill[black!20, domain=-1.6:-0.8, samples=40] plot (\x, {pow(\x,4)/4 - pow(\x,2)/2})
        -- (-0.8, 0.3584) -- (-1.2, 0.3584) -- cycle;
    \fill[black!20, domain=0.8:1.6, samples=40] plot (\x, {pow(\x,4)/4 - pow(\x,2)/2})
        -- (1.2, 0.3584) -- (0.8, 0.3584) -- cycle;
    \node at (-1.15, 0.42) {$\halfspace_-'$};
    \node at (1.15, 0.42) {$\halfspace_+'$};
    \node (A) at (0, 0.42) {$A$};
    \draw[domain=-1.6:1.6, samples=100] plot (\x, {pow(\x,4)/4 - pow(\x,2)/2});
    \draw[->] (-1.6,-0.4) -- (1.6,-0.4) node[right] {$\hat u(0)$};
    \draw[dashed] (0,-0.4) -- (A.south);
    \foreach \x in {-1,0,1}
        {\draw (\x,-0.35) -- (\x,-0.45) node[below] {$\x$};}
    \foreach \x/\y in {-0.8/-0.2176, 0.8/-0.2176}
        {\draw[dashed] (\x, 0.3584) -- (\x,\y);}
\end{tikzpicture}
\caption{In \Cref{thm:eyring-kramers}, the crossing frequency of line $A$ is computed.
The transmission coefficient tells how many of those crossings are
from one shaded set $\halfspace'_\pm$ to the other.
Again, the potential is shown with respect to constant functions only.}
\label{fig:intro hyperbolic}
\end{figure}
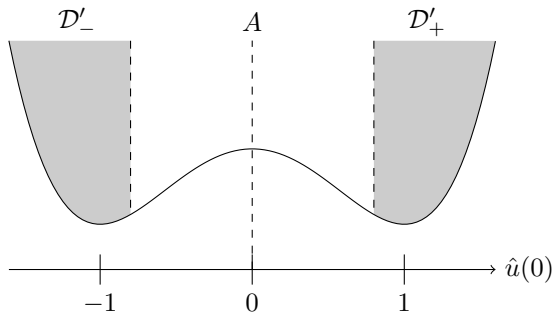

\begin{theorem}\label{thm:eyring-kramers}
Let us consider the flow of \eqref{eq:NLW} or \eqref{eq:SdNLW} on $\Tor$,
with initial data $(u_0, v_0)$ sampled from the ($\phi^4_d$, white noise) product measure.
Let us denote $\halfspace_+ \coloneqq \{ \hat u(0) > 0 \}$
and $\halfspace_- \coloneqq \{ \hat u(0) < 0 \}$.
Then the expected frequency of the flow to cross $\halfspace_+ \leftrightarrow \halfspace_-$ is
equal to
\[
\frac{2}{\sqrt{2\pi \beta L^d}} \frac{\int \delta_0(\hat u(0)) \exp(-\beta H(u)) \diff\Gaussian(u)}
    {\int \exp(-\beta H(u)) \diff\Gaussian(u)}.
\]
This requires that $\Tor$ has period length $L < 2\pi$.
\end{theorem}

\begin{corollary}[{{\cite[Section~4.4]{newhall_metastability_2017}}}]
\label{thm:intro nlw 1d}
In the previous setup of \eqref{eq:NLW} or \eqref{eq:SdNLW},
the expected transition frequency is asymptotically equal to
\[
\frac{1}{2\pi} \prod_{k \in \Z} \sqrt{\frac{C_L k^2 + 2}{\abs{C_L k^2 - 1}}}
    \exp(-\beta L / 4)
\quad\text{as } \beta \to \infty.
\]
\end{corollary}

It is worth noting that the prefactors in \Cref{thm:intro she 1d,thm:intro nlw 1d}
are inverses of each other except for a factor of $\sqrt 2$.

Transition state theory only depends on the invariant measure,
so the expected transition frequencies for \eqref{eq:NLW} and \eqref{eq:SdNLW} are equal.
However, TST overcounts the number of transitions:
it is possible that the solution crosses the dividing surface,
but ``turns back'' into the potential well it came from.
This needs to be accounted for with a \emph{dynamical correction}
\cite{sharia_analytic_2016,tal_transition_2006,vanden-eijnden_transition_2005},
also known as \emph{transition coefficient}:
the proportion of crossings of the dividing surface
that are also crossings from one well to the other. 

The flows of \eqref{eq:NLW} and \eqref{eq:SdNLW} are expected to differ in this respect.
To the knowledge of the present authors,
dynamical corrections for neither model have not been explicitly derived before.

\bigskip\noindent
In two dimensions, the solution theory of our equations gets more challenging.
Due to the irregularity of the space-time white noise,
the solutions are almost surely distribution-valued.
It becomes necessary to renormalize the $u^3$ nonlinearity by Wick ordering,
denoted $\wick{u^3}$.
The system is truncated to wavenumbers of magnitude at most $N$,
which lets us replace the nonlinearity by
\[
\proj_N \wick{u_N^3} = \proj_N (u_N^3) - C_N u_N.
\]
The divergent renormalization constant $C_N$ is given in \Cref{def:wick powers}.
We defer further details to \Cref{sec:wick}.

Berglund, Di Gesù, and Weber were able to compute the expected transition time
for truncated, renormalized \eqref{eq:SHE} in two dimensions:
\begin{theorem}[{{\cite[Theorem~2.3]{berglund_eyringkramers_2017}}}]
\label{thm:intro she 2d}
Given $0 < \delta < 1$ and $\epsilon > 0$, we define
\[
\mathcal B_\pm \coloneqq \Big\{ u \in H^{-\epsilon}(\Tor^2) \colon \abs{\pm 1 - \hat u(0)} < \delta,
    \norm{\proj_\perp u}_{H^{-\epsilon}} \lesssim \beta^{-1/2} \log(\beta) \Big\}.
\]
There exists a sequence $(\mu_N)$ of probability measures concentrated on $\partial \mathcal B_+$
such that the expected hitting time of $\mathcal B_-$
by the \eqref{eq:SHE} flow is
\[
\bigg[ 2\pi \sqrt{\prod_{\abs k \leq N} \frac{\smallabs{C_L \abs k^2 - 1}}{C_L \abs k^2 + 2}}
    \exp(3 C_N L^2 /2) \bigg] \exp(-\beta L^2 / 4) (1 + r(\beta)),
\]
where the error term satisfies $-\bigO(\beta\inv) \leq r(\beta) \leq \bigO(\beta^{-1/2})$.
This estimate holds uniformly in a subsequence $N \to \infty$.
The prefactor in brackets is positive and bounded.
\end{theorem}

\begin{remark}
Although the result is stated for a flow started from $\partial \mathcal B_+$,
the details of the initial distribution are expected not to matter very much.
As $\beta \to \infty$, the time spent within a potential well should be large enough
for mixing to occur.
\end{remark}

\begin{remark}
The product over $\abs k \leq N$ vanishes as $N \to \infty$,
but this is balanced by the divergent exponential renormalization term.
See \Cref{thm:prefactor renormalization}.

The renormalization constant $C_N$ depends on the mass of the field.
As noted in \cite[Remark~2.5]{berglund_eyringkramers_2017},
a finite shift $\theta_N$ in the renormalization constant
multiplies the expected transition time by $\exp(3 \theta_N L^2 / 2)$.
In this article, we state all results with respect to the constant
corresponding to $\beta\inv (-1-\Laplace)\inv$ covariance.
\end{remark}

\bigskip\noindent
In three dimensions, a further renormalization is needed to construct the $\phi^4$ measure.
The formal definition of the measure is
\begin{equation}
\lim_{N \to \infty}
\pf_N\inv \exp\left( -\frac \beta 4 \int_{\Tor^d}
    \wick{\phi_N^4} - \gamma_N \wick{\phi_N^2} - \delta_N \dx \right) \diff\Gaussian_N(\phi_N),
\end{equation}
where $\gamma_N$ and $\delta_N$ are certain diverging constants.
The corresponding version of \eqref{eq:NLW} is then
\begin{equation}
\partial_{tt} u(x, t) - (m^2 + \Laplace) u(x, t) = -\wick{u(x, t)^3} + 2\gamma_N u(x,t).
\end{equation}
Here the mass term is effectively $m^2 + C_N + 2 \gamma_N$,
which makes the potential wells deeper and further apart as $N$ grows.
To the knowledge of present authors,
no three-dimensional counterpart of \Cref{thm:intro she 2d} has been proved yet.

The programme of \emph{stochastic quantization} is about constructing the limiting $\phi^4_d$ measure
(where the domain is usually $\Tor^d$, or $\R^d$ equipped with a weight)
as an invariant measure to a stochastic PDE, usually \eqref{eq:SHE}.
In two dimensions, this was achieved by Da~Prato and Debussche
at the turn of the millennium \cite{da_prato_strong_2003}.

The three-dimensional construction came in three approaches in mid-2010s:
Hairer's theory of regularity structures \cite{hairer_theory_2014},
Kupiainen's renormalization group proof \cite{kupiainen_renormalization_2016}, and
Gubinelli, Imkeller, and Perkowski's paracontrolled distributions \cite{gubinelli_paracontrolled_2015}.
In this article, we use Barashkov and Gubinelli's variational approach \cite{barashkov_variational_2020}
that also builds on the latter.
In four and more dimensions the limiting measure is trivial \cite{aizenman_marginal_2021}.

The study of the double-well $\phi^4$ model also has a long tradition.
The model arises as the limiting dynamics of the Ising--Kac model
\cite{grazieschi_dynamical_2023, mourrat_convergence_2017}.
The double-well $\phi^4_2$ measure is further analyzed in \cite{hairer_tightness_2018},
and the $\phi^4_3$ measure in \cite{chandra_phase_2022}.

Studying \eqref{eq:NLW} with $\phi^4$ initial data goes back to works of
Friedlander \cite{friedlander_1985} and McKean and Vaninsky \cite{mckean_1999}, 
and is closely related in spirit to the study of the Schrödinger equation with $\phi^4$ initial data,
undertaken by Lebowitz, Rose and Speer \cite{lebowitz_1988} and Bourgain \cite{bourgain_1994,bourgain_1996}. 
The more difficult two-dimensional case has been studied by Oh and Thomann \cite{oh_invariant_2020}.
Recently the $\phi^4_3$ case was understood in the breakthrough work of
Bringmann, Deng, Nahmod and Yue \cite{bringmann_invariant_2022}.  

\eqref{eq:SdNLW} has also been studied in one dimension by Walsh \cite{walsh1986introduction}
and in two dimensions by Gubinelli, Koch, Oh and Tolomeo \cite{gubinelli2018renormalization, gubinelli_global_2022}.
In three dimensions well-posedness for \eqref{eq:SdNLW} with a quadratic nonlinearity
has been investigated by Gubinelli, Koch, and Oh \cite{gubinelli2023paracontrolled}
and Oh, Okamoto and Tolomeo \cite{oh2021stochastic}.
Well-posedness of the untruncated cubic equation is still to be proven.

Metastability as a study of rare events is closely related to large deviations.
Large deviations for the $\phi^4$ measure have also received some attention in the literature,
see \cite{barashkov2022variational, barashkov_2023, klose_2024}.
A refinement of this is the recent work on low-temperature expansions in the double-well case
on $\Tor^2$ \cite{gess2024low}.
We remark that some of our computations from Section~\ref{sec:3d tst} might be
of independent interest to extend the results of \cite{gess2024low} to three dimensions.

\subsection{Outline of results}\label{sec:intro results}

In this article, we develop a hyperbolic analogue of \Cref{thm:intro she 2d}.
We are also able to produce some initial results on $\Tor^3$.
We extend \Cref{thm:intro nlw 1d} into two and three dimensions as follows:
\begin{theorem}\label{thm:main tst}
The expected hitting frequency of $\{ \hat u(0) = 0 \}$ by the
truncated, renormalized \eqref{eq:NLW} or \eqref{eq:SdNLW} flow
on $\Tor^d$, $d=2,3$, is
\[
\bigg[ \frac{1}{2\pi}
\sqrt{\prod_{\abs k \leq N} \frac{C_L \abs k^2 + 2}{\smallabs{C_L \abs k^2 - 1}}}
\exp(-3 C_N L^d / 2) \bigg] \exp(-\beta L^d / 4) (1 + \bigO(\beta\inv)).
\]
This estimate is uniform in the truncation $N$,
and the prefactor in brackets is positive and bounded.
\end{theorem}
The proof of this result is presented in \Cref{sec:2d tst,sec:3d tst}.
It consists of computing the two integrals appearing in \Cref{thm:eyring-kramers}.
We adapt some ideas from previous literature (e.g.\ \cite{berglund_eyringkramers_2017})
to work with the $\phi^4_3$ computations developed in \cite{barashkov_variational_2020,chandra_phase_2022}.

The main ideas of the argument are laid out in \Cref{sec:2d tst},
which shows the two-dimensional case. 
The three-dimensional argument in \Cref{sec:3d tst}
requires more work due to the additional renormalisation constants;
see especially \Cref{sec:3d wick}.

\bigskip\noindent
Our second result is, to the present authors' knowledge,
the first explicit computation of the transmission coefficient of \eqref{eq:NLW}.
It is proved in \Cref{sec:corr}.
Putting it together with Theorem \ref*{thm:main tst} yields the following result.
\begin{theorem}\label{thm:main corr}
For $\epsilon > 0$ set 
\[
    \halfspace'_\pm \coloneqq \Big\{ u \in H^{-\epsilon}(\Tor^2) \colon \abs{\pm 1 - \hat u(0)} < \delta,
        \norm{\proj_\perp u}_{H^{-\epsilon}} \lesssim \beta^{-1/2+\epsilon} \Big\}.
\]
Let $u$ be the solution to \eqref{eq:NLW} on $\mathbb{T}^2$,
with initial data sampled from the invariant measure.
Let $\mathcal N(T)$ be the number of transitions
between $\halfspace'_-$ and $\halfspace'_+$ in the time interval $[0,T]$.
Then $\E \mathcal N(T) / T$ is equal to
\begin{equation*}
\bigg[ \frac{1}{2\pi}
\sqrt{\prod_{\abs k \leq N} \frac{C_L \abs k^2 + 2}{\smallabs{C_L \abs k^2 - 1}}}
\exp(-3 C_N L^d / 2) \bigg] \exp(-\beta L^d / 4) (1 + \bigO(\beta^{-1/4+\epsilon})).
\end{equation*}
\end{theorem}

We do not aim for sharpest possible exponent in the error term, and we expect that the bound could be further improved.

\subsection{Further directions}
Our results leave two interesting questions: 
\begin{itemize}
\item What is the transmission coefficient of \eqref{eq:NLW} in three dimensions?
    The $\phi^4_3$ measure is not absolutely continuous with respect to a Gaussian measure,
    whereas $\phi^4_2$ is.
    The additional renormalization also changes the dynamics.
    We expect that adapting the techniques developed in \cite{bringmann_invariant_2022}
    one should be able to show that the dynamical correction is again close to $1$,
    but this is certainly technically difficult. 

\item What is the transmission coefficient of \eqref{eq:SdNLW}, in any dimension?
    The stochastic forcing should increase the likelihood of the flow ``turning back'' near the saddle,
    and we would expect the transition coefficient to approach a value strictly between $0$ and $1$.
\end{itemize}
The last question is especially interesting,
as \eqref{eq:SdNLW} can be seen to ``interpolate'' between \eqref{eq:NLW} and \eqref{eq:SHE} dynamics
as discussed in the beginning of this section.
Better understanding of the stochastic case could thus lead to a more unified metastability
theory of the parabolic and hyperbolic equations.
A heuristic result is discussed in \cite[Remark~5.2]{bouchet_generalisation_2016}.

Finally, let us remark on convergence as $N \to \infty$.
Our estimates for the transition frequency and transmission coefficient are uniform in $N$,
and \Cref{sec:tst} holds also for the limiting flow.
However, we refrain here from stating a precise convergence result for the random variable $\mathcal N(T)$.
This topic is discussed in \cite[Remark~2.7]{berglund_eyringkramers_2017},
where convergence in probability is sketched.

The one-dimensional \Cref{thm:intro she 1d} does hold for the limiting flow,
thanks to higher moment bounds for hitting times and thus uniform integrability
\cite[Section~3.3]{berglund_sharp_2013}.
Such a moment bound depends on the particular PDE dynamics.
The proof in \cite{berglund_sharp_2013} depends heavily on the Markov property of the flow.

\subsection{Notation}\label{sec:notation}

We will fix the following notational conventions:
\begin{itemize}
\item $\Tor^d$ is the periodic domain ${[{0}, {L}]}^d$ with $L \in ({0}, {2\pi})$ fixed.
\item $\epsilon$ is a small positive constant that may vary between sections.
\item $C$ and $c$ are positive constants that may vary from line to line.
    As an exception, we fix
    $C_N$ as in \Cref{def:wick powers} and $C_L \coloneqq (2\pi/L)^2$.
\item We denote $A \lesssim B$ if $A \leq CB$, and $A \simeq B$ if $cB \leq A \leq CB$.
\item $\inorm{\,\cdot\,} \coloneqq (1 + \abs{\,\cdot\,}^2)^{1/2}$ is the usual
    inhomogenous norm.
    We also denote $\inormp{\,\cdot\,} \coloneqq (2 + \abs{\,\cdot\,}^2)^{1/2}$
    and $\inormm{\,\cdot\,} \coloneqq (-1 + \abs{\,\cdot\,}^2)^{1/2}$
    for versions adapted to the mass of the field.
\item $\E$ is expectation with respect to the ambient probability space $\Prob$.
\item $\negpart{x}$ is the negative part of $x$, equal to $\min(-x, 0)$.
\item $\proj_N$ denotes Fourier-space truncation to wavenumbers of size at most $N$,
    and $\proj_\perp$ the projection to oscillatory part (which may be also truncated).
\end{itemize}

\subsection{Acknowledgements}

The authors would like to thank Giacomo Di~Gesù and Antti Kupiainen for useful discussions.

NB and PL were funded by the ERC Advanced Grant 741487 ``Quantum Fields and Probability''.
PL was additionally funded by the Academy of Finland project 339982
and the Finnish Centre of Excellence in Randomness and Structures.
PL would like to thank the Max Planck Institute for Mathematics in the Sciences
for hospitality during a research visit.

\section{Preliminaries}

\subsection{Expected transition frequency}\label{sec:tst}

The prefactor in the Eyring--Kramers law is usually written as an expression involving
eigenvalues of $-\Laplace + U''$ evaluated at the local maxima and minima of the potential $U$
(see e.g.\ the introductory sections of \cite{berglund_sharp_2013,newhall_metastability_2017}).
In this article, we simplify the notation by always writing these eigenvalues explicitly,
under the assumption $m^2 = 1$.
The derivation of the prefactor as a Gaussian normalization constant is done in \Cref{sec:2d tst}.

In this subsection, we discuss the abstract \Cref{thm:eyring-kramers}
and the definition of the transmission coefficient.
The following proof is adapted from \cite[Section~4.1]{newhall_metastability_2017}.
For illustrative purposes, we work here directly with Dirac deltas $\delta_0$
and not any regularized objects.

\begin{proof}[Proof of \Cref{thm:eyring-kramers}]
The majority of proof holds for an arbitrary hyperbolic PDE that has an invariant measure,
and state space partitioned into two generic open sets and their boundary.
The sets are parametrized by a continuous function $q$ so that $\halfspace_+ \coloneqq \{ q > 0 \}$ and
$\halfspace_- \coloneqq \{ q < 0 \}$.

Let $u(t)$ evolve according to the flow of the PDE.
The number of $\halfspace_- \to \halfspace_+$ crossings in the time interval ${[{0}, {T}]}$ is
\begin{equation}\label{eq:tst starting point}
\int_0^T \max\{ 0, \partial_t \Gamma(q(u(t))) \} \dt,
\end{equation}
where $\Gamma$ is the Heaviside step function.
The maximum function is used to count only the positive-direction crossings.
We can expand the derivative with the chain rule:
\begin{equation}
\int_0^T \max\{ 0, \nabla q(u(t)) \cdot \partial_t u(t) \} \delta_0(q(u(t))) \dt.
\end{equation}
Now if we assume the initial data to be sampled from the relevant invariant measure $\mu$,
the distribution of $(u, v) = (u, \partial_t u)$ is also $\mu$.
Integrating over $\mu$ and dividing by $T$, we find that the expected crossing frequency is
\begin{equation}
\int \max\{ 0, \nabla q(u) \cdot v \} \delta_0(q(u)) \diff\mu(u, v).
\end{equation}
We now put $q(u) \coloneqq \hat u(0)$ into this formula
and expand the definition of the product measure to get
\begin{equation}
\frac{\E_\beta \max\{ 0, \hat v(0) \} \int \delta_0(\hat u(0)) \exp(-\beta H(u)) \diff\Gaussian(u)}{\pf(\beta)}.
\end{equation}
By symmetry of the potential,
we multiply by two in order to account for $\halfspace_+ \to \halfspace_-$ crossings as well.

For both \eqref{eq:NLW} and \eqref{eq:SdNLW},
the distribution of $v = \partial_t u$ is white noise by the product measure.
The white-noise part has the one-dimensional Gaussian expectation
\[
\E_\beta \max\{ 0, \hat v(0) \}
= \sqrt{\frac{1}{2\pi \beta L^d}},
\]
and this yields the stated formula.
\end{proof}

If we identify a Schwartz distribution on $\Tor^d$ with its Fourier coefficients,
we see that the set $\{ u \colon \hat u(0) = 0 \}$
is a hypersurface of codimension $1$.
In the following, we will call the integral over this surface the \emph{saddle integral},
as it corresponds to the saddle point of the potential.

\begin{remark}
The parabolic results are about the expected hitting time of $\halfspace_+$,
whereas we derived the expectation of the inverse quantity.
The inverse of the expected frequency only provides a lower bound
for the expected transition time by Jensen's inequality.

However, the difference is expected to be vanishing in the $\beta \to \infty$ limit,
since the flow spends most of its time inside the potential wells.
This makes it approach a two-state continuous-time Markov process.
\end{remark}

This derivation depended in no way on the specifics of the flow,
only the invariant measure.
This implies that the number of $\{ \hat u(0) = 0 \}$ crossings
is the same for both \eqref{eq:NLW} and \eqref{eq:SdNLW}.
Their difference is in the crossing frequency between the smaller subsets
$\halfspace_+'$ and $\halfspace_-'$ defined in \Cref{thm:main corr}.

Since \eqref{eq:NLW} is a Hamiltonian equation with no external forcing,
its past and future trajectory is completely determined by the current state space position.
If we consider a $\halfspace_- \to \halfspace_+$ crossing,
the hypersurface $\{ \hat u(0) = 0 \}$ is partitioned into two:
those trajectories that are also part of a $\halfspace_-' \to \halfspace_+'$ crossing,
and those that return to $\halfspace_-'$ without entering $\halfspace_+'$.

This idea can be formalized as follows.
If $\hat u(0, t) = 0$, then we define the four hitting times
\begin{equation}
\begin{split}
\tau_t^+ \coloneqq \inf_{s > t} \{ \hat u(0, s) \geq 1 - \delta \}, \quad&\quad
\tau_t^- \coloneqq \sup_{s < t} \{ \hat u(0, s) \leq -1 + \delta \},\\
\sigma_t^+ \coloneqq \inf_{s > t} \{ \hat u(0, s) \leq 0 \}, \quad&\quad
\sigma_t^- \coloneqq \sup_{s < t} \{ \hat u(0, s) \geq 0 \}.
\end{split}
\end{equation}
Thus $\tau_t^+ < \sigma_t^+$ corresponds to the flow going to $\halfspace_+'$
without turning back to the saddle,
and $\tau_t^- > \sigma_t^-$ to the flow having not already hit the saddle after departing from $\halfspace_-'$.
Therefore a $\halfspace_-' \to \halfspace_+'$ crossing at time $t$ can be written as the event
\begin{equation}
E_t \coloneqq
\{ \hat u(0, t) = 0 \}
\cap \{ \tau_t^+ < \sigma_t^+ \}
\cap \{ \tau_t^- > \sigma_t^- \}.
\end{equation}
We can then modify \eqref{eq:tst starting point} as
\begin{equation}
\int_0^T \max\{ 0, \partial_t \Gamma(q(u(t))) \} \I(E_t) \dt.\label{eq:dynamical corrections}
\end{equation}
The rest of the derivation proceeds as before,
meaning that the $\halfspace_-' \to \halfspace_+'$ transition frequency is
\begin{equation}
\frac{\E_\beta \max\{ 0, \hat v(0) \}
    \int \I(E_t) \delta_0(\hat u(0)) \exp(-\beta H(u)) \diff\Gaussian(u)}{\pf(\beta)}.
\end{equation}
Observe that we may use conditional probability to write
\begin{equation}
\begin{split}
&\mathrel{\phantom{=}} \int \I(E_t) \delta_0(\hat u(0)) \exp(-\beta H(u)) \diff\Gaussian(u).\\
&= \Prob_{\perp}(E_t) \int\delta_0(\hat u(0)) \exp(-\beta H(u)) \diff\Gaussian(u).
\end{split}
\end{equation}
where the probability $\Prob_{\perp}$ is with respect to the measure given by 
\[
\frac{ \delta_0(\hat u(0)) \exp(-\beta H(u)) \diff\Gaussian(u)}
    {\int \delta_0(\hat u(0)) \exp(-\beta H(u)) \diff\Gaussian(u)},
\]
that is, the $\phi^4$ field conditioned to be at the saddle surface. 
Finally, the reversibility of \eqref{eq:NLW} implies that it suffices to look in one direction of time:
\begin{equation}
\begin{split}
\Prob_{\perp}(E_t)
&= \Prob_{\perp}\!\big(\{ \tau_t^+ < \sigma_t^+ \} \cap \{ \tau_t^- >\sigma_t^- \}\big)\\
&= 1 - \Prob_{\perp}\!\big(\{ \tau_t^+ \geq \sigma_t^+ \} \cup \{ \tau_t^- \leq \sigma_t^- \}\big)\\
&\geq 2\Prob_{\perp}(\tau_t^+ < \sigma_t^+) - 1.
\end{split}
\end{equation}
By the stationarity of the flow $\Prob_{\perp}(E_{t})$ does not depend on $t$.
This leads us to the following conclusion:

\begin{lemma}\label{thm:transmission coefficient}
Let $\mathcal N(T)$ be defined as in \Cref{thm:main corr}.
Then 
\[
\E \frac{\mathcal N(T)}{T}
= \frac{\E_\beta \max\{ 0, \hat v(0) \}
        \int \delta_0(\hat u(0)) \exp(-\beta H(u)) \diff\Gaussian(u)}
    {\pf(\beta)}\times \Prob_{\perp}(E_t),
\]
where
\[
2\Prob_{\perp}(\tau_t^+ < \sigma_t^+) - 1 \leq \Prob_{\perp}(E_t) \leq 1.
\]
Here $\Prob_\perp$ is probability with respect to initial data $(u_0, v_0)$ so that
\begin{itemize}
    \item $u_0$ is sampled from the $\phi^4$ measure conditioned on $\hat u_0(0) = 0$, and
    \item $v_0$ is sampled from the white noise measure conditioned on $\hat v_0(0) > 0$.
\end{itemize}
\end{lemma}
We will call  $\Prob_{\perp}(\tau^{+}_{t}<\sigma_{t}^{+})$ the dynamical correction.
It is estimated in \Cref{sec:corr}.
The stopping time $\tau^+_t$ does not take the size of the oscillatory modes into account,
but as a byproduct of the proof we find that
\begin{equation}
\norm{\proj_\perp u(\tau^+_t)}_{H^{-\epsilon}} \lesssim \beta^{-1/2+\epsilon}
\end{equation}
with high probability,
as required by the definition of $\halfspace'_+$.

The same principle of hitting is expected to apply to \eqref{eq:SdNLW},
however in \eqref{eq:dynamical corrections} one sees a probability instead of a characteristic function,
since the dynamics is no longer completely determined by the initial data.
Computing this probability is expected to give rise to a different dynamical correction
leading to a modification of the Eyring--Kramers formula.
We consider this to be an interesting open problem.

\subsection{Besov spaces}\label{sec:besov}

Throughout this article, we consider functions and distributions as elements of Besov spaces.
These spaces somewhat generalize Sobolev spaces.
Below we recall the definition and basic properties.
For more information, see e.g.\ the textbook \cite{bahouri_fourier_2011},
\cite[Appendix~A]{gubinelli_paracontrolled_2015},
or \cite[Section~3]{mourrat_global_2017}.

\begin{definition}
Let $\chi$ be a smooth function that equals $1$ in $B(0, 2) \setminus B(0, 1) \subset \R$,
and is supported in $B(0, 8/3) \setminus B(0, 3/4)$.
The \emph{Littlewood--Paley blocks} $\Delta_k$, $k \geq -1$, are defined as Fourier multipliers:
\begin{itemize}
\item for $k \geq 0$, they are defined by symbol $\xi \mapsto \chi(2^{-k} \abs\xi)$,
    that is, they restrict frequencies to dyadic annuli; and
\item $\Delta_{-1}$ has symbol $\xi \mapsto 1 - \sum_{k \geq 0} \chi(2^{-k} \abs\xi)$,
    making it a smooth indicator of the unit ball in frequency space.
\end{itemize}
Here $\chi$ is chosen so that the symbols form a partition of unity.
\end{definition}

\begin{definition}
The (nonhomogeneous) \emph{Besov space} $B^s_{p,r}(\Tor^d)$
is the completion of $C^\infty(\Tor^d)$ with respect to the norm
\[
\norm{f}_{B^s_{p,r}(\Tor^d)} \coloneqq \norm{ 2^{ks} \norm{\Delta_k f}_{L^p(\Tor^d)} }_{\ell^r(k \geq -1)},
\quad 1 \leq p,r \leq \infty,\; s \in \R.
\]
We abbreviate $H^s \coloneqq B^s_{2,2}$ and $\besovinfty^s \coloneqq B^s_{\infty,\infty}$.
\end{definition}

\begin{remark}
As a particular case we get from Plancherel's identity
\[
\norm{f}_{H^s}^2 \simeq \sum_{k \in \Z^2} \inorm{k}^{2s} \smallabs{\hat f(k)}^2.
\]
\end{remark}

The decomposition of functions/distributions into dyadic frequency blocks
permits a useful decomposition of products.
Paraproducts capture low--high-frequency interactions,
whereas the resonant product captures interactions between similar frequency ranges.
Paraproducts are heavily used in the analysis of the three-dimensional $\phi^4$ model.

\begin{definition}
We can decompose
\[
fg = f \paral g + f \reson g + f \parar g,
\]
where the \emph{paraproduct} is defined as
\[
f \paral g \coloneqq \sum_{k < \ell - 1} (\Delta_k f)(\Delta_\ell g),
\]
and the \emph{resonant product} as
\[
f \reson g \coloneqq \sum_{\abs{k - \ell} \leq 1} (\Delta_k f)(\Delta_\ell g).
\]
\end{definition}

An important consequence of the definitions is that
paraproducts are Fourier-supported away from the origin.
The mean of the product $fg$ is completely contained in the resonant term $f \reson g$.
Moreover, paraproducts and resonant products
are ``almost adjoint'' operators.
This is stated precisely in \Cref{thm:paraproduct almost adjoint},
closer to where the result is used.

The usefulness of Besov spaces lies in the following multiplication inequalities.
A distribution and a function can be multiplied together if their differentiabilities
(parameter $s$ in the norm) sum to a positive value.

\begin{theorem}\label{thm:besov multiplication}
Let $1/p = 1/p_1 + 1/p_2$.
Then we have the paraproduct estimate
\[
\norm{f \paral g}_{B^{min(0, \alpha) + \beta}_{p, r}}
\lesssim \norm{f}_{B^\alpha_{p_1, \infty}} \norm{g}_{B^\beta_{p_2, r}}.
\]
Note that the low-frequency term can only decrease the smoothness of the product.
On the other hand, for $\alpha + \beta > 0$ we have the resonant bound
\[
\norm{f \reson g}_{B^{\alpha + \beta}_{p, r}}
\lesssim \norm{f}_{B^\alpha_{p_1, \infty}} \norm{g}_{B^\beta_{p_2, r}}.
\]
Together these imply that for $\alpha + \beta > 0$ we have
\[
\norm{fg}_{B^{\min(\alpha, \beta)}_{p, r}}
\lesssim \norm{f}_{B^\alpha_{p_1, r}} \norm{g}_{B^\beta_{p_2, r}}.
\]
There is also a special case for duality, where $s \in \R$:
\[
\norm{fg}_{L^1} \lesssim \norm{f}_{B^{-s}_{p,r}} \norm{g}_{B^s_{p', r'}},
\quad 1 = \frac 1 p + \frac 1 {p'} = \frac 1 r + \frac 1 {r'}.
\]
\end{theorem}

An especially useful variant of the multiplication estimate
is proved in \cite[Lemma~A.7]{gubinelli_pde_2021}:

\begin{theorem}\label{thm:besov lp product}
For $s, \epsilon > 0$ we have
\[
\smallnorm{f^2}_{B^{s}_{1,1}}
\lesssim \norm{f}_{L^2} \norm{f}_{B^{s+\epsilon}_{2,2}},
\quad\text{and}\quad
\smallnorm{f^3}_{B^s_{1,1}}
\lesssim \norm{f}_{L^4}^2 \norm{f}_{B^{s+\epsilon}_{2,2}}.
\]
\end{theorem}

Besov spaces can be defined via interpolation of Sobolev spaces,
and there are several interpolation results between Besov spaces.
We use the following one to estimate some Besov norms in terms of $L^4$ and $H^1$ norms.

\begin{theorem}\label{thm:besov interpolation}
Fix $\theta \in (0, 1)$ such that $s = \theta \alpha + (1-\theta) \beta$ and
\[
\frac 1 p = \frac{\theta}{p_1} + \frac{1-\theta}{p_2}, \quad
\frac 1 r = \frac{\theta}{r_1} + \frac{1-\theta}{r_2}.
\]
Then we can interpolate
\[
\norm{f}_{B^s_{p,r}} \lesssim \norm{f}_{B^\alpha_{p_1, r_1}}^\theta \norm{f}_{B^\beta_{p_2, r_2}}^{1-\theta}.
\]
\end{theorem}

Another important property is the ability to trade smoothness for integrability.
This leads to the following embedding results.
Note that the ``exchange rate'' between the $p$ and $s$ parameters
depends on the dimension of the underlying space.

\begin{theorem}\label{thm:besov embeddings}
Let $1 \leq q < p \leq \infty$ and $s \in \R$.
Then
\[
\norm{f}_{B^s_{p,r}} \lesssim \norm{f}_{B^{s'}_{q,r}}, \quad
\text{where } s' \geq s + d\left( \frac 1 q - \frac 1 p \right).
\]
There are also the following relations between Besov and usual Sobolev norms
(where $k \in \N$, including $k = 0$ corresponding to $L^p$ space):
\[
\norm{f}_{B^k_{p,\infty}} \lesssim \norm{f}_{W^{k,p}} \lesssim \norm{f}_{B^k_{p,1}}.
\]
\end{theorem}

\subsection{Variational formulation}\label{sec:boue-dupuis}

We will be computing expectations of form $\E e^{-V}$.
It is useful to rephrase them as stochastic control problems over $V$
without the exponential.

This variational approach is due to Üstünel \cite{ustunel_variational_2014},
who in turn extended previous work of Boué and Dupuis \cite{boue_variational_1998}.
Barashkov and Gubinelli used this technique to construct the $\phi^4$ measure
\cite{barashkov_variational_2020}.
(See also e.g.\ \cite{bringmann_invariant_2022b,chandra_phase_2022}
for other applications of the same underlying argument.)

We begin by defining our stochastic objects precisely.
As we discuss in \Cref{sec:wick},
it is necessary to consider regularized versions of these objects.
In the variational formalism the regularization is controlled by a continuous parameter $T \in \R_+$.%
\footnote{This parameter is sometimes referred to as ``time'',
but is not to be confused with the dynamics of the wave equation.}

\begin{definition}
Let $(B^k_t)_{k \in \Z^d}$ be standard complex Brownian motions,
independent except for $B^{-k}_t = \overline{B^k_t}$.
For $k=0$ we take $B^0_t$ to be a real-valued Brownian motion.
We then define a Gaussian process
\[
B_t(x) \coloneqq \sum_{k \in \Z^d} \exp\!\left(\tfrac{2\pi}{L} i k \cdot x\right) B^k_t.
\]
Let $\mathcal F_t$ be the filtration corresponding to this process.
We denote by $\mathbb H_a$ the space of $\mathcal F_t$-progressively measurable processes
that are $\Prob$-almost surely in $L^2(\R_+ \times \Tor^d; \; \R)$.
Elements of $\mathbb H_a$ are called \emph{drifts}.
\end{definition}

We are going to define the Gaussian free field as the limit of regularized processes
where $B_t$ is Fourier-multiplied with the covariance $(\pm m^2 - \Laplace)\inv$.
We first define a family of auxiliary functions
that are used for regularization.

\begin{definition}
Let us fix a non-increasing function $\rho \in C_c^\infty(\R_+;\; \R_+)$
with $\rho(s) = 1$ for $s \leq 1$ and $\rho(s) = 0$ for $s \geq 2$.
We then define
\[
\rho_t(k) \coloneqq \rho\left( \frac{2 \inorm k}{\inorm t} \right)
\quad\text{and}\quad
\sigma_t(k) \coloneqq \left( \frac{\diff}{\dt} \rho_t(k)^2 \right)^{1/2}.
\]
\end{definition}

The definition of $\sigma_t$ is chosen so that
\begin{equation}\label{eq:boue-dupuis sigma squared}
    \int_0^T \sigma_t(k)^2 \dt = \rho_T(k)^2 \xrightarrow[T \to \infty]{} 1.
\end{equation}
This means that the following definition would yield a Gaussian process that
has $(m^2 - \Laplace)\inv$ as its limiting covariance when $T \to \infty$.
Up to any finite $T$, the process is continuous in $t$ and smooth in space.

\begin{remark}
Here we deviate slightly from the argument of \cite{barashkov_variational_2020}.
The invariance of the $\phi^4$ measure follows immediately from Liouville's theorem
if the measure and dynamics are truncated to finitely many Fourier modes.
This is incompatible with the continuous parametrization.

This issue prompted Bringmann to use a slightly different definition in \cite{bringmann_invariant_2022b}.
We instead truncate the dynamics to wavenumbers of magnitude at most $N$,
and stop the stochastic control evolution once $T \geq N$.
This is possible since we never pass $N \to \infty$,
as long as our estimates are uniform in $N$.
\end{remark}

\begin{definition}\label{def:variational gff}
Fix $N \in \N$.
Let $\sigma_t$ be as above,
and let the mass associated with the Gaussian free field be~$\pm m^2$.
We define the operator $J_t \colon L^2(\Tor^d) \to L^2(\Tor^d)$ by its Fourier transform
\[
\widehat{J_t[g]}(k) \coloneqq \frac{\sigma_t(k)
    \I_{\abs k \leq N} \I_{\pm m^2 + C_L \abs k^2 > 0}}{\sqrt{\pm m^2 + C_L \abs k^2}} \hat g(k),
\]
where $C_L = (2\pi/L)^d$,
and then the operator $I_T \colon \mathbb H_a \to L^2(\Tor^d)$ by
\[
I_T[f] \coloneqq \int_0^T J_t f_t \dt.
\]
In particular, we define the \emph{$N$-truncated Gaussian free field} as $\phi_N \coloneqq I_N[B_t]$,
which equals $I_T[B_t]$ for all $T \geq N$.
\end{definition}

Let us note that $t \mapsto J_t[B_t]$ and $t \mapsto I_t[B_t]$ are both martingales.
We use this property together with the Itô formula.
Additionally, the following lemma formalizes the smoothing behaviour of $J_t$.

\begin{lemma}\label{thm:boue-dupuis j regularity}
For any $s,r \in \R$ and $g \in H^r$ we have
\[
\norm{J_t[g]}_{H^r}^2 \lesssim \frac{1}{\inorm t^{1+2s}} \norm{g}_{H^{r-1+s}}^2.
\]
\end{lemma}
\begin{proof}
Let us write
\begin{equation}
\norm{J_t[g]}_{H^r}^2
\lesssim \sum_{k \in \Z^d} \frac{(1+\abs k^2)^r}{\pm m^2 + C_L \abs k^2}
    \abs{ \rho_t(k) \rho'(2 \inorm k / \inorm t) \frac{\inorm k}{\inorm t^2} } \hat g(k)^2.
\end{equation}
(If the mass term is $-1$, we sum over non-zero $k$ only.)
Both $\rho$ and $\rho'$ are bounded by smoothness.
As $\sigma_t$ is supported in an annulus proportional to $t$,
we also have $m^2 + C\abs k^2 \simeq \inorm k^2 \simeq \inorm t^2$.
Thus we can move out $\inorm{k} \inorm{t}^{-2}$
and introduce $\inorm k^{2s} / \inorm t^{2s} \simeq 1$ to the above:
\begin{equation}
\frac{1}{\inorm t^{1+2s}} \sum_{k \in \Z^d} \inorm k^{2r-2+2s} \hat g(k)^2,
\end{equation}
which is exactly as required.
\end{proof}

We are now ready to present the \emph{Boué--Dupuis formula}.
The fourth-order polynomial potential we study satisfies the tameness condition
by boundedness from below and hypercontractivity of the Gaussian measure
(see e.g.\ \cite[Proposition~3.3]{mourrat_construction_2017}).

\begin{definition}
We call a measurable functional $V \colon  C_x^\infty(\Tor^d) \to \R$ \emph{tame} if
\[
\E \abs{V(\phi_T)}^p + \E \exp(-q V(\phi_T)) < \infty
\quad\text{for some } \frac 1 p + \frac 1 q = 1.
\]
\end{definition}

\begin{theorem}[{{\cite[Theorem~7]{ustunel_variational_2014}}}]\label{thm:boue-dupuis}
Fix $0 < T < \infty$
and let $V$ be as above.
Then
\[
-\log \E e^{-V(\phi_T)}
= \inf_{v \in \mathbb H_a} \E \left[ V(\phi_T + I_T[v]) + \frac 1 2 \int_0^T \norm{v_t}_2^2 \dt \right].
\]
\end{theorem}

\subsection{Wick renormalization}\label{sec:wick}

Let us now turn to the necessity of regularization and renormalization.
As was mentioned in the introduction,
the Gaussian free field $\phi$ is not well-defined in two or more dimensions.
This is due to the divergence of
\begin{equation}
\E \phi^2
= \operatorname{Tr} [\beta\inv (\pm m^2 - \Laplace)\inv]
= \frac{1}{L^d} \sum_{k \in \Z^d} \frac{1}{\beta(\pm m^2 + C_L \abs k^2)}
\end{equation}
when $d \geq 2$.
However, it turns out that by subtracting this divergence,
we can get a well-defined stochastic object.
(The case of negative mass is only formal at this point,
but makes sense once coupled with the quartic potential.)

\begin{definition}\label{def:wick powers}
Let us define
\[
C_{T, \beta, +} \coloneqq \frac{1}{L^d} \sum_{\substack{k \in \Z^d\\ \abs k \leq N}}
    \frac{\rho_T(k)^2}{\beta(2 + C_L \abs k^2)},
\quad
C_{T, \beta, -} \coloneqq \frac{1}{L^d} \sum_{\substack{k \in \Z^d \setminus \{0\}\\ \abs k \leq N}}
    \frac{\rho_T(k)^2}{\beta(-1 + C_L \abs k^2)}.
\]
Note that the other truncation parameter $N$ is implicit in this definition.
We will rescale the Gaussian field to $\beta = 1$ below, and will omit $\beta$ from the notation.
Moreover, we write $C_T$ when the mass is clear from the context.
The Wick powers of $\phi_N$ are then defined via Hermite polynomials;
the first four are
\begin{align*}
\wick{\phi_T} &= \phi_T,\\
\wick{\phi_T^2} &= \phi_T^2 - C_T,\\
\wick{\phi_T^3} &= \phi_T^3 - 3C_T \phi_T,\\
\wick{\phi_T^4} &= \phi_T^4 - 6C_T \phi_T^2 + C_T^2.
\end{align*}
\end{definition}

\begin{remark}
Let us emphasize again that when $T \geq N$, then
\[
C_T = C_N = \sum_{\substack{k \in \Z^d\\ \abs k \leq N}}
    \frac{1}{\pm m^2 + C_L \abs k^2}.
\]
This is the renormalization constant for GFF truncated sharply to wavenumbers at most size $N$.
For $T < N$ the constant $C_T$ is adapted to the continuous variational formulation.
\end{remark}

Let us then collect some important properties of Wick powers.
See e.g.\ \cite[Appendix~A]{berglund_eyringkramers_2017}
and \cite[Section~3]{da_prato_strong_2003} for more details.

\begin{theorem}\label{thm:wick expectation}
For any $j = 1, 2, \ldots$ and $T \geq 0$, the Wick power $\wick{\phi_T^j}$ has zero expectation
under the Gaussian measure of the corresponding mass.
\end{theorem}

\begin{theorem}\label{thm:wick change}
Let $\wick{\phi^j}'$ be the Wick powers with respect to a different renormalization constant $C'_T$.
Then
\begin{align*}
\wick{\phi_T}' &= \wick{\phi_T},\\
\wick{\phi_T^2}' &= \wick{\phi_T^2} - (C'_T - C_T),\\
\wick{\phi_T^3}' &= \wick{\phi_T^3} - 3(C'_T - C_T)\phi_T,\\
\wick{\phi_T^4}' &= \wick{\phi_T^4} - 6(C'_T - C_T)\wick{\phi_T^2} + 3(C'_T - C_T)^2.
\end{align*}
\end{theorem}

\begin{theorem}\label{thm:wick binomial}
Let $x \in \R$. Then we have the usual binomial formula
\[
\wick{(\phi+x)^j} = \sum_{k = 0}^j \binom j k \wick{\phi^k} \, x^{j-k}.
\]
\end{theorem}

We then have the following moment bounds for the Gaussian free field in suitable Besov spaces.
The proof of the two-dimensional case can be found in
\cite[Lemma~3.2]{da_prato_strong_2003} or \cite[Theorem~5.1]{mourrat_global_2017},
and the three-dimensional version is discussed e.g.\ in \cite[Lemma~4]{barashkov_variational_2020}
and \cite{mourrat_construction_2017}.

\begin{theorem}\label{thm:wick gff moments}
The Wick powers of Gaussian free field on $\Tor^2$ satisfy
\[
\E \smallnorm{\wick{\phi_T^j}}_{\besovinfty^{-\epsilon}(\Tor^2)}^p < \infty,
\quad 1 \leq p < \infty,\; j = 1,2,\ldots.
\]
These estimates hold uniformly in $T \to \infty$ and as the truncation $N \to \infty$.
On $\Tor^3$, we have
\[
\limsup_{N,T \to \infty} \E \norm{\phi_T}_{\besovinfty^{-1/2-\epsilon}(\Tor^3)}^p < \infty,
\quad
\limsup_{N,T \to \infty} \E \smallnorm{\wick{\phi_T^2}}_{\besovinfty^{-1-\epsilon}(\Tor^3)}^p < \infty,
\]
but the limit $\wick{\phi^3}$ only exists as a space-time distribution.
\end{theorem}

The definition of Wick powers stems from Gaussian chaos decomposition,
where the Hermite polynomials form a complete orthonormal basis of $L^2(\Prob)$.
In this setting the Fourier coefficients of Wick powers are given by
\begin{equation}
\widehat{\wick{\phi_T^j}}(k) = j! \hspace{-1em} \sum_{\substack{n_1 + \cdots + n_j = k\\ \abs{n_i} \leq N}}
    \int_0^T \! \int_0^{t_1} \!\cdots\! \int_0^{t_{j-1}}
    \sigma_{t_1}(n_1) \cdots \sigma_{t_j}(n_j)
    \diff B_{t_j}^{n_j} \cdots \diff B_{t_1}^{n_1}.
\end{equation}
The summation indices are restricted to $\abs{n_i} \leq N$ due to the implicit sharp truncation.
When the mass is negative, they are further restricted to be non-zero.

\begin{example}\label{thm:chaos w2 w3}
For $\W^2 = 12 \wick{\phi^2}$ and $\W^3 = 4 \wick{\phi^3}$
as used in \Cref{sec:3d tst}, the chaos decompositions are
\begin{gather}
\widehat{\W^2_T}(k) = 24 \sum_{n_1 + n_2 = k}
    \int_0^T \! \int_0^{t_1} \frac{\sigma_{t_1}(n_1) \sigma_{t_2}(n_2)}{\inormc{n_1} \inormc{n_2}}
        \diff B^{n_2}_{t_2} \diff B^{n_1}_{t_1}, \label{eq:chaos w2}\\
\widehat{\W^3_T}(k) = 24 \sum_{\substack{n_1 + n_2\\ + n_3 = k}}
    \int_0^T \! \int_0^{t_1} \!\! \int_0^{t_2}
    \frac{\sigma_{t_1}(n_1) \sigma_{t_2}(n_2) \sigma_{t_3}(n_3)}
        {\inormc{n_1} \inormc{n_2} \inormc{n_3}}
        \diff B^{n_3}_{t_3} \diff B^{n_2}_{t_2} \diff B^{n_1}_{t_1}. \label{eq:chaos w3}
\end{gather}
Again, the wavenumbers $n_1$, $n_2$, $n_3$ are also restricted by the sharp truncation.
\end{example}

There is a diagrammatic approach to understanding the chaos decompositions
of different stochastic objects; we refer to \cite{mourrat_construction_2017} for a more complete overview.
In these diagrams, white noises are represented by dots and integrals by lines.
For example, $\W^2$ and $\W^3$ are respectively represented as

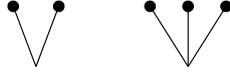
\begin{center}
\begin{tikzpicture}
    \begin{scope}
        \filldraw (0,0) -- (-0.3, 0.8) circle[whitenoise];
        \filldraw (0,0) -- (0.3, 0.8) circle[whitenoise];
    \end{scope}
    \begin{scope}[xshift=2cm]
        \filldraw (0,0) -- (-0.5, 0.8) circle[whitenoise];
        \filldraw (0,0) -- (0, 0.8) circle[whitenoise];
        \filldraw (0,0) -- (0.5, 0.8) circle[whitenoise];
    \end{scope}
\end{tikzpicture}
\end{center}

A result we use repeatedly is \emph{Wick's theorem}.
In diagrammatic terms, it states that the expectation of a stochastic object
is given by summing over all possible \emph{contracted} diagrams:
ones where the dots are paired, with no pairings allowed within a tree.
By the diagram above, $\E [\W^2 \W^3] = 0$ since no such contractions exist.
As a different example, we can contract $\E [\W \W^2 \W^3]$ as

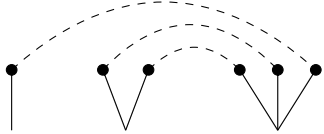
\begin{center}
\begin{tikzpicture}
    \begin{scope}[xshift=-1.5cm]
        \coordinate (A1) at (0,0.8);
        \filldraw (0,0) -- (A1) circle[whitenoise];
    \end{scope}
    \begin{scope}
        \coordinate (B1) at (-0.3, 0.8);
        \coordinate (B2) at (0.3, 0.8);
        \filldraw (0,0) -- (B1) circle[whitenoise];
        \filldraw (0,0) -- (B2) circle[whitenoise];
    \end{scope}
    \begin{scope}[xshift=2cm]
        \coordinate (C1) at (-0.5, 0.8);
        \coordinate (C2) at (0, 0.8);
        \coordinate (C3) at (0.5, 0.8);
        \filldraw (0,0) -- (C1) circle[whitenoise];
        \filldraw (0,0) -- (C2) circle[whitenoise];
        \filldraw (0,0) -- (C3) circle[whitenoise];
    \end{scope}

    \draw[contract] (A1) parabola[bend pos=0.5] bend +(0,0.9) (C3);
    \draw[contract] (B1) parabola[bend pos=0.5] bend +(0,0.6) (C2);
    \draw[contract] (B2) parabola[bend pos=0.5] bend +(0,0.3) (C1);
\end{tikzpicture}
\end{center}

The rule $\diff B^m_t \diff B^n_t = \delta_{m,n} \dt$
introduces linear dependencies between paired summation indices.
This greatly simplifies the chaos decomposition.
To bound some of the resulting sums,
we use the following discrete convolution estimate from \cite{mourrat_construction_2017}.

\begin{lemma}\label{thm:discrete convolution}
Let $\alpha, \beta < d$ and $\alpha + \beta > d$.
Then
\[
\sum_{\substack{n_1, n_2 \in \Z^d\\ n_1 + n_2 = n}} \inorm{n_1}^{-\alpha} \inorm{n_2}^{-\beta}
\lesssim \inorm{n}^{d - \alpha - \beta}.
\]
\end{lemma}
\begin{proof}
\cite[Lemma~4.1]{mourrat_construction_2017}.
\end{proof}

\section{TST in two dimensions}\label{sec:2d tst}

We now move on to the proof of \Cref{thm:main tst}.
In this section we consider the two-dimensional case where the technical details are easier.

To recall \Cref{thm:eyring-kramers}, the expected transition frequency
between the two potential wells can be computed from
\begin{equation}\label{eq:2d eyring-kramers}
\frac{2 \int \delta_0(\hat u(0)) \exp(-\beta H(u)) \diff\Gaussian(u)}
    {\sqrt{2\pi \beta L^d} \pf(\beta)}.
\end{equation}
We must compute the partition function $\pf(\beta)$
and the ``saddle integral'' in the numerator.
We first derive the exact prefactor of the Eyring--Kramers law in this subsection,
and then estimate the normalized integrals with the variational method.

\bigskip\noindent
The Gibbs measure of the truncated $\phi^4_2$ model can be written as
\begin{equation}\label{eq:gibbs 2d}
    \frac{1}{\pf_N} \exp\!\left[ -\frac \beta 4 \int_{\Tor^2} \!\! \proj_N \wick{u(x)^4}_{\beta,-} \dx
        - \frac \beta 2 \int_{\Tor^2} - u(x)^2 + \abs{\nabla u(x)}^2 \dx \right]
    \prod_{\abs k \leq N} \!\diff \hat u(k).
\end{equation}
In the limit $N \to \infty$, the product Lebesgue measure ceases to make sense.
However, the second-order term within the exponential defines a Gaussian density
with covariance $\beta\inv (-1-\Laplace)\inv$.
It is well-known in the theory of stochastic quantization
that limits of measures like \eqref{eq:gibbs 2d} are well-defined
and absolutely continuous with respect to the Gaussian measure
(\cite{da_prato_strong_2003}, see also \cite[Section~6]{mourrat_global_2017}).

However, the above explanation is not strictly true.
A Gaussian free field cannot have covariance $\beta\inv (-1-\Laplace)\inv$
since the zero Fourier mode would have negative variance.
The other Fourier modes are well-defined if the domain size $L$ is strictly smaller than $2\pi$.

The dominant fourth-order term fixes this issue.
As done in \cite{berglund_eyringkramers_2017} and other literature,
we can re-centre the field around a potential minimum.
By symmetry we can take the constant function $+1$ as the minimum
and compute $\pf_N / 2$ by integrating over the half-space where $\hat u(0) > 0$.

\begin{lemma}\label{thm:2d partition function form}
The partition function $\pf_N$ can be written as
\begin{equation*}
\frac{\pf_N}{2} =
\sqrt{\prod_{\abs k \leq N} \frac{2\pi}{\beta L^2 (2 + C_L \abs k^2)}}
\exp\left(\! \frac{3 C_N L^2}{2} + \frac{\beta L^2}{4} \right)
\int_{\halfspace_+} \!\!
    \exp\left( -V_\beta(\varphi) \right)
    \diff\Gaussian_N(\varphi),
\end{equation*}
where
\begin{itemize}
\item $\displaystyle{
V_\beta(\varphi) = \int_{\Tor^2} \frac{1}{4\beta} \wickm{\varphi^4} + \frac{1}{\sqrt\beta} \wickm{\varphi^3} \dx }$,
\item $\Gaussian_N(\diff\varphi)$ is a Gaussian measure with covariance $(2 - \Laplace)\inv$,
truncated to wavenumbers of magnitude at most $N$,
\item $C_N$ is the Wick renormalization constant with negative mass, and
\item $\halfspace_+ \coloneqq \{ \hat\varphi(0) > -\sqrt\beta \}$.
\end{itemize}
\end{lemma}
\begin{proof}
If we denote $u = \varphi + 1$, then the fourth-order term becomes
\begin{equation}\label{eq:2d derivation fourth-order}
\frac 1 4 \wick{(\varphi + 1)^4}
= \frac 1 4 \wick{\varphi^4} + \wick{\varphi^3} + \frac 3 2 \wick{\varphi^2} + \varphi + \frac 1 4.
\end{equation}
The second-order term is changed to
\begin{equation}\label{eq:2d derivation gaussian}
\frac 1 2 \left[ - u^2 + \abs{\nabla u}^2 \right]
= -\frac 1 2 \varphi^2 - \varphi - \frac 1 2 + \abs{\nabla \varphi}^2.
\end{equation}
The newly introduced first-order terms cancel out,
and we may move $3 \varphi^2 / 2$
from \eqref{eq:2d derivation fourth-order} to \eqref{eq:2d derivation gaussian}.
This changes the covariance to $\beta\inv (2 - \Laplace)\inv$.

There are also two constant terms, $\beta L^2 / 4$ and
the $3 \beta L^2 C_{N,\beta} / 2$ coming from renormalization.
The latter simplifies as $3 L^2 C_N / 2$.
We move these out into the exponential prefactor.

We then normalize the Gaussian measure,
and denote the normalized measure by $\diff\Gaussian_{N, \beta}(\varphi)$.
The partition function is therefore equal to
\begin{multline}
\pf_N = 2
\sqrt{\prod_{\abs k \leq N} \frac{2\pi}{\beta L^d (2 + C_L \abs k^2)}}
\exp\left( \frac{3 C_N L^2}{2} + \frac{\beta L^2}{4} \right)\\
\int_{\{ \hat\varphi(0) > -1 \}}
    \exp\left( -\frac \beta 4 \int_{\Tor^2} \proj_N [\wickm{\varphi^4} + 4 \wickm{\varphi^3}] \dx \right)
    \diff\Gaussian_{N, \beta}(\varphi).
\end{multline}
We can drop the projection operator $\proj_N$,
since it does not affect the value of the integral over whole $\Tor^2$.

We finally simplify the notation by rescaling $\varphi$ by $\sqrt\beta$;
this cancels the $\beta\inv$ from the covariance and leaves us with the claim.
From now on we drop $\beta = 1$ from the notation for the Gaussian measure $\Gaussian_N$.
\end{proof}

There is now a mismatch between the random field and the Wick ordering:
we need to change the renormalization
from $(-1 - \Laplace)\inv$ covariance to $(2 - \Laplace)\inv$ covariance.
This is done with \Cref{thm:wick change},
and it introduces some second-order correction terms.
These terms are easily estimated as part of the stochastic quantization argument.

\begin{theorem}\label{thm:2d partition overall}
For $\beta$ large enough, we have
\[
\pf_N =
2 \sqrt{\prod_{\abs k \leq N} \frac{2\pi}{\beta L^d (2 + C_L \abs k^2)}}
    \exp\left( \frac{3 C_{N} L^2}{2} + \frac{\beta L^2}{4} \right)
    (1 + \bigO(\beta\inv)).
\]
\end{theorem}
\begin{proof}
We only need to show that the integral in \Cref{thm:2d partition function form} approaches $1$ sufficiently fast.
First, we split the integration region $\halfspace_+$ into two parts:
a slightly smaller ``good'' set
\begin{equation}
\goodregion_+ \coloneqq \left\{ \hat\varphi(0) > (-1 + \delta')\sqrt\beta \right\},
\end{equation}
and the remaining region
\begin{equation}
\badregion_+ \coloneqq \left\{ -\sqrt\beta < \hat\varphi(0) < (-1 + \delta')\sqrt\beta  \right\}.
\end{equation}
Here $\delta' > 0$ is a parameter to be fixed in \eqref{eq:2d delta fixed};
it will be close to but strictly smaller than $1$.
The integral over $\badregion_+$ corresponds to a (half-)neighbourhood of the saddle.
We show it to be negligible in \Cref{sec:2d concentration}.

The variational formula in \Cref{thm:boue-dupuis} is defined for integrals over the whole space.
To go back to the full space, we can write
\begin{equation}
\int_z^\infty F(x) \dx
= \lim_{\kappa \to \infty} \int_{-\infty}^\infty \exp(-\kappa \negpart{x-z}^2) F(x) \dx,
\end{equation}
where $\negpart{x-z} \coloneqq \min(-x, 0)$ denotes the negative part of $x-z$.
If we fix $\kappa$ to a finite value, as will be done in \eqref{eq:2d kappa fixed},
the estimation error will be
\begin{equation}
\int_{-\infty}^z \exp(-\kappa \negpart{x-z}^2) F(x) \dx.
\end{equation}
In our concrete case, the approximating integral is thus
\begin{equation}\label{eq:2d partition approx}
\int \exp( -\kappa \negpart{\hat\varphi(0) + (1 - \delta')\sqrt\beta }^2 )
    \exp\left( -V_\beta(\varphi) \right)
    \diff\Gaussian(\varphi),
\end{equation}
which is shown to be finite in \Cref{sec:2d partition}.
The error is
\begin{equation}
\int_{\{ \hat\varphi(0) < (-1+\delta')\sqrt\beta \}}
    \exp( -\kappa \negpart{\hat\varphi(0) + (1 - \delta')\sqrt\beta }^2 )
    \exp\left( -V_\beta(\varphi) \right)
    \diff\Gaussian(\varphi).
\end{equation}
We can use Hölder to upper-bound the error by
\begin{multline}\label{eq:2d error bound}
\Prob\!\left( \hat\varphi(0) < (-1+\delta')\sqrt\beta \right)^{1-1/q}\\
\bigg[ \int \exp( -q\kappa \negpart{\hat\varphi(0) + (1 - \delta')\sqrt\beta }^2 )
    \exp\left( -q V_\beta(\varphi) \right)
    \diff\Gaussian(\varphi) \bigg]^{1/q}.
\end{multline}
The term in brackets is \eqref{eq:2d partition approx} except for the additional $q > 1$ factor.
For $q$ small enough, the proof in \Cref{sec:2d partition} still works.
The prefactor is a one-dimensional Gaussian probability that has exponential Markov bounds;
therefore it is much smaller than the required order $1/\beta$.
\end{proof}

We also need to compute the integral over the hypersurface where $\hat u(0) = 0$.
Since the zero mode was the only problematic part of the covariance (as long as $L < 2\pi$),
we will not need to repeat the re-centering trick.
Using relevant parts from \Cref{sec:2d partition},
we find the same convergence rate as for $\pf_N$ and a similar Gaussian prefactor:

\begin{theorem}\label{thm:2d saddle}
The saddle integral
\[
\mathcal I_N = \int_{\{ \hat u(0) = 0 \}} \hspace{-1em}
    \exp\!\left[ -\int_{\Tor^2} \frac{1}{4\beta} \wick{u(x)^4}_-
        + \frac 1 2 u(x)^2 - \frac 1 2\abs{\nabla u(x)}^2 \dx \right]
    \prod_{\abs k \leq N} \!\diff \hat u(k)
\]
satisfies
\[
\mathcal I_N
= \sqrt{\prod_{0 < \abs k \leq N} \frac{2\pi}{\beta L^d (C_L \abs k^2 - 1)}} (1 + \bigO(\beta\inv)).
\]
\end{theorem}
\begin{proof}
Since there is no re-centring, the prefactor comes from
the Gaussian covariance $(-1 - \Laplace) \inv$ restricted to non-zero wavenumbers.

As in \Cref{sec:2d partition}, we can write the variational formulation of the saddle integral as
\begin{equation}
-\log \mathcal I_N
= \inf_v \E_u \left[ \frac{1}{4\beta} \int_{\Tor^2} \wickm{(u + I_N[v])^4} \dx
    + \frac 1 2 \int_0^N \norm{v_t}_2^2 \dt \right].
\end{equation}
The lower bound follows by the quartic terms in \Cref{sec:2d partition},
and the upper bound follows by substituting $v = 0$ as in there.

We note that the operator $I_N$ depends on the mass
and thus the stochastic estimates are slightly different to the $(2 - \Laplace)\inv$ case.
However, the two covariances are essentially same for large wavenumbers,
so the difference is irrelevant.
\end{proof}

Let us then combine \Cref{thm:2d partition overall,thm:2d saddle}
to find the final estimate stated in \Cref{thm:main tst}.
First, the Gaussian prefactors and the white-noise part
from \Cref{thm:eyring-kramers} yield the prefactor
\begin{equation}
\frac{2}{\sqrt{2\pi \beta L^d}}
\cdot \sqrt{\prod_{0 < \abs k \leq N} \frac{2\pi}{\beta L^d (C_L \abs k^2 - 1)}}
\cdot \frac 1 2 \sqrt{\prod_{\abs k \leq N} \frac{\beta L^d (2 + C_L \abs k^2)}{2\pi}},
\end{equation}
which simplifies as
\begin{equation}
\frac{\sqrt{\smash[b]{\smallabs{C_L \abs 0^2 - 1}}}}{2\pi}
\sqrt{\prod_{\abs k \leq N} \frac{C_L \abs k^2 + 2}{\smallabs{C_L \abs k^2 - 1}}}.
\end{equation}
\Cref{thm:2d partition overall} also contributes two exponential terms:
$\exp(-\beta L^2 / 4)$ gives the exponential dependency on temperature,
and $\exp(-3 C_N L^2 / 2)$ renormalizes the prefactor as shown below.
Finally, the quotient of the $(1 + \bigO(\beta\inv))$ error terms is still $(1 + \bigO(\beta\inv))$.

\begin{lemma}\label{thm:prefactor renormalization}
We have the uniform-in-$N$ bounds
\begin{equation}\label{eq:prefactor normalization}
\sqrt{\prod_{0 < \abs k \leq N} \frac{C_L \abs k^2 + 2}{C_L \abs k^2 - 1} }
    \exp(-3 C_N L^d / 2)
\simeq 1.
\end{equation}
This estimate holds both in $d=2$ and $d=3$.
\end{lemma}
\begin{proof}
Let us first expand the definition of $C_N$ in
\begin{equation}
\exp(-3 C_N L^d / 2)
= \exp\!\bigg(\! - \sum_{\abs k \leq N} \frac{3}{2 (C_L \abs k^2 - 1)} \bigg).
\end{equation}
We can thus write the square of \eqref{eq:prefactor normalization} as
\begin{equation}
e^{3} \prod_{0 < \abs k \leq N} \frac{1 + 3/y}{1} \exp(-3/y),
\quad\text{where}\quad y = C_L \abs k^2 - 1.
\end{equation}
One can verify the logarithm estimate
\begin{equation}
0 < \sum_{0 < \abs k \leq N} \log \left( \frac{\exp(3/y)}{1 + 3/y} \right)
< \sum_{0 < \abs k \leq N} 5 y^{-2}.
\end{equation}
The upper estimate is a sum over $\abs k^{-4}$,
so the product is both positive and finite.
This holds for both two- and three-dimensional $k$.
\end{proof}

\subsection{Partition function}\label{sec:2d partition}

In the variational formulation of \Cref{thm:boue-dupuis},
the approximating integral \eqref{eq:2d partition approx} can be written as
\begin{equation}\label{eq:2d variational}
\begin{split}
&\int \exp\left( -q\kappa \negpart{\hat\varphi_N(0) + (1 - \delta')\sqrt\beta }^2
        - qV_\beta(\varphi_N) \right)
    \diff\Gaussian(\varphi_N)\\
=\; &\exp\bigg( {-\inf_v} \E \bigg[ q \kappa \negpart{\hat\phi_N(0) + \widehat{I_N[v]}(0) + (1 - \delta')\sqrt\beta }^2\\
    &\qquad + q V_\beta(\phi_N + I_N[v])
    + \frac 1 2 \int_0^N \norm{v_t}_2^2 \dt \bigg] \bigg).
\end{split}
\end{equation}
Here $\phi_N$ is sampled from the Gaussian free field of covariance $(2-\Laplace)\inv$;
we differentiate it from $\varphi_N$ by typographical convention.
For \eqref{eq:2d partition approx} we have $q = 1$;
for the error bound \eqref{eq:2d error bound} we set $q > 1$ instead.

Our goal is to estimate \eqref{eq:2d variational} uniformly in the truncation parameter $N$.
Let us recall from \Cref{sec:boue-dupuis}
that this parameter also appears implicitly in the definition of the operator $I_t$.

For now, we assume that $V_\beta$ is already defined
with the $(2 - \Laplace)\inv$ Wick ordering that matches the random field.
We defer the analysis of the correction terms to the end of this section.

\bigskip
Let us first prove the much easier upper bound on the infimum,
which is shown by choosing a suitable drift $v$.
In two dimensions it is sufficient to put $v \equiv 0$.
In three dimensions the choice is more complicated.
This operation corresponds essentially to Jensen's inequality;
compare with \cite[Section~4.2]{berglund_eyringkramers_2017}.

\begin{theorem}\label{thm:2d partition lower}
The infimum in \eqref{eq:2d variational} can be bounded from above by $C/\beta$,
where $C$ depends on $\kappa$, $\delta'$ and $q$.
\end{theorem}
\begin{proof}
As we put $v \equiv 0$ into the infimum, we are left with
\begin{equation}
    \E \left[ q\kappa \negpart{\hat\phi_N(0) + (1 - \delta')\sqrt\beta }^2
    + \int_{\Tor^2} \frac{q}{4\beta} \wickp{\phi_N^4} + \frac{q}{\sqrt\beta} \wickp{\phi_N^3} \dx \right].
\end{equation}
The Wick powers have zero expectation by \Cref{thm:wick expectation}.
Hölder and the rough bound
$\negpart{\hat\phi_N(0) + (1 - \delta')\sqrt\beta } \leq \smallabs{\hat\phi_N(0)}$
give us
\begin{equation}
\E\, \negpart{\hat\phi_N(0) + (1 - \delta')\sqrt\beta }^2
\leq \Prob\left(\hat\phi_N(0) < -(1 + \delta')\sqrt\beta \right)^{1/2}
    \left[ \E \smallabs{\hat\phi_N(0)}^4 \right]^{1/2}.
\end{equation}
A fourth-order Markov estimate then gives us
\begin{equation}
\E q\kappa\, \negpart{\hat\phi_N(0) + (1 - \delta')\sqrt\beta }^2
\leq q\kappa \frac{\E \smallabs{\hat\phi_N(0)}^4}{(1-\delta')^2 \beta},
\end{equation}
and since $\hat\phi_N(0)$ is a Gaussian random variable, the claim follows.
\end{proof}

To get an upper bound for \eqref{eq:2d variational}, we need to lower-bound the infimum.
The bounds we get also hold with a small factor $q > 1$ in front
of the two potential terms, as required by \eqref{eq:2d error bound}.

The argument adapts \cite[Section~3]{barashkov_variational_2020}.
We are going to control the infimum with the two purely $v$-dependent positive terms
\begin{equation}\label{eq:2d variational controls}
    \frac{1}{4\beta} \int_{\Tor^2} I_N[v]^4 \dx
    \quad\text{and}\quad
    \frac 1 2 \int_0^N \norm{v_t}_2^2 \dt.
\end{equation}
Some care is necessary to preserve the factor of $\beta\inv$,
especially when working with the cubic terms.

\begin{remark}
The estimates in \cite{barashkov_variational_2020} are mostly independent of the domain size.
Since we have fixed $L < 2\pi$, we are free to use some simpler estimates
(especially in \Cref{sec:3d tst}).

Additionally, the Lebesgue measure to is normalized to unit mass in \cite{barashkov_variational_2020}.
We do not do so,
but this does not change the argument in any meaningful way.
\end{remark}

The following lemma makes the latter drift-dependent term slightly easier to use:

\begin{lemma}\label{thm:2d drift h1}
We have
\[
\norm{I_N[v]}_{H^1}^2 \leq \int_0^N \norm{v_t}_2^2 \dt.
\]
\end{lemma}
\begin{proof}
By moving to the Fourier side, we can estimate the $H^1$ norm with
\begin{align*}
&\sum_{\abs k \leq N} \left[ \sqrt{\frac{1 + C_L \abs k^2}{2 + C_L \abs k^2}}
    \int_0^N \sigma_t(k) \hat v_t(k) \dt \right]^2\\
\leq\; &\sum_{k \in \Z^d} \left[ \int_0^N \sigma_t(k)^2 \dt \right]
    \left[ \int_0^N \hat v_t(k)^2 \dt \right].
\end{align*}
Now we can use the property~\eqref{eq:boue-dupuis sigma squared} that $\sigma_t^2$ integrates to $1$.
\end{proof}

The result that we prove in the rest of this subsection can be stated as:

\begin{theorem}\label{thm:2d variational upper}
For $\beta$ large enough and $q \geq 1$ small,
we can fix $\kappa$, $\delta'$, and $\rho \in {({0}, {1})}$ such that
\begin{align*}
&\mathrel{\phantom{\leq}}
\E \bigg| q\kappa \negpart{\hat\phi_N(0) + \widehat{I_N[v]}(0) + (1 - \delta')\sqrt\beta }^2\\
    &\quad\qquad + q\int_{\Tor^2} \frac{1}{4\beta} \wickp{(\phi_N + I_N[v])^4}
        - \frac{1}{4\beta} I_N[v]^4
    + \frac{1}{\sqrt\beta} \wickp{(\phi_N + I_N[v])^3} \dx \bigg|\\
&\leq \frac C \beta + \rho \left[ \frac{1}{4 \beta} \int_{\Tor^2} I_N[v]^4 \dx
    + \frac 1 2 \norm{I_N[v]}_{H^1}^2 \right].
\end{align*}
This implies that the infimum in \eqref{eq:2d variational} is bounded from below by
\[
-\frac C \beta + (1-\rho) \E \left[ \frac{1}{4\beta} \int_{\Tor^2} I_N[v]^4 \dx
    + \frac 1 2 \int_0^N \norm{v_t}_2^2 \dt \right].
\]
As the term in brackets is non-negative, this yields a lower bound $-C/\beta$.
\end{theorem}

\subsubsection*{Fourth-order terms}

We use \Cref{thm:wick binomial} to expand $\wick{(\phi_N + I_N[v])^4}$ as
\begin{equation}
\wick{\phi_N^4}
+ 4\wick{\phi_N^3} I_N[v]
+ 6\wick{\phi_N^2} I_N[v]^2
+ 4\phi_N I_N[v]^3
+ I_N[v]^4.
\end{equation}
The $\E \wickp{\phi_N^4}$ term vanishes by \Cref{thm:wick expectation}.
To control the mixed terms,
we will apply the multiplicative inequality \Cref{thm:besov multiplication}.

The first mixed term is
\begin{align}\label{eq:2d phi3 v}
\frac{1}{\beta} \int_{\Tor^2} \wickp{\phi_N^3} I_N[v] \dx
&\lesssim \frac 1 \beta \smallnorm{\wickp{\phi_N^3}}_{H^{-\epsilon}} \norm{I_N[v]}_{H^{\epsilon}} \notag\\
&\lesssim \frac{C_\rho}{\beta} \smallnorm{\wickp{\phi_N^3}}_{H^{-\epsilon}}^2
    + \frac \rho \beta \norm{I_N[v]}_{H^1}^2.
\end{align}
The first term has finite expectation by \Cref{thm:wick gff moments}.
The introduction of $\rho$ is even not necessary here, as it would suffice to take $\beta$ large.

In the second mixed term, we apply \Cref{thm:besov lp product}
to use both the $H^1$ and $L^4$ controlling terms:
\begin{align}
\frac{1}{\beta} \int_{\Tor^2} \wickp{\phi_N^2} I_N[v]^2 \dx
&\lesssim \frac{1}{\beta} \smallnorm{\wickp{\phi_N^2}}_{\besovinfty^{-\epsilon}}
    \smallnorm{I_N[v]^2}_{B^{\epsilon}_{1,1}} \notag\\
&\lesssim \frac{1}{\beta} \smallnorm{\wickp{\phi_N^2}}_{\besovinfty^{-\epsilon}}
    \norm{I_N[v]}_{L^2} \norm{I_N[v]}_{H^{2\epsilon}}\\
&\lesssim \frac{C_\rho}{\beta} \smallnorm{\wickp{\phi_N^2}}_{\besovinfty^{-\epsilon}}^4
    + \frac{\rho}{\beta} \norm{I_N[v]}_{L^4}^4
    + \frac{\rho}{\beta} \norm{I_N[v]}_{H^1}^2. \notag
\end{align}
Again, the first term has finite expectation.
Finally, we use \Cref{thm:besov lp product} to estimate the third term:
\begin{equation}\label{eq:2d partition pi3}
\begin{split}
\frac{1}{\beta} \int_{\Tor^2} \phi_N I_N[v]^3 \dx
&\lesssim \frac 1 \beta \norm{\phi_N}_{\besovinfty^{-\epsilon}}
    \smallnorm{I_N[v]^3}_{B^{\epsilon}_{1,1}}\\
&\lesssim \frac 1 \beta \norm{\phi_N}_{\besovinfty^{-\epsilon}}
    \norm{I_N[v]}_{L^4}^2 \norm{I_N[v]}_{H^{2\epsilon}}.
\end{split}
\end{equation}
Now we use a crude application of \Cref{thm:besov interpolation} to estimate
\begin{equation}
\norm{I_N[v]}_{H^{2\epsilon}}
\lesssim \norm{I_N[v]}_{B^{1/2}_{8/3,\infty}}
\lesssim \norm{I_N[v]}_{B^0_{4,\infty}}^{1/2} \norm{I_N[v]}_{B^1_{2,\infty}}^{1/2},
\end{equation}
which lets us bound
\begin{equation}
\begin{split}
\eqref{eq:2d partition pi3}
&\lesssim \frac 1 \beta \norm{\phi_N}_{\besovinfty^{-\epsilon}}
    \norm{I_N[v]}_{L^4}^{5/2} \norm{I_N[v]}_{H^1}^{1/2}\\
&\lesssim \frac{C_\rho}{\beta} \norm{\phi_N}_{\besovinfty^{-\epsilon}}^8
    + \frac \rho \beta \norm{I_N[v]}_{L^4}^4
    + \frac \rho \beta \norm{I_N[v]}_{H^1}^2.
\end{split}
\end{equation}
As the stochastic term has finite expectation,
this completes the estimation of fourth-order terms.

\subsubsection*{Third-order terms}
With the fourth-order terms we could have $\beta^{-1}$ in front of all controlling terms,
including $\norm{I_N[v]}_{H^1}^2$.
Here we only have $\beta^{-1/2}$ at our disposal.
We can modify the argument of \eqref{eq:2d phi3 v} to estimate
\begin{equation}
\frac{1}{\sqrt\beta} \int_{\Tor^2} \wickp{\phi_N^2} I_N[v] \dx
\lesssim \frac{C_\rho}{\beta} \norm{\phi_N^2}_{H^{-\epsilon}}^2 + \rho \norm{I_N[v]}_{H^1}^2.
\end{equation}
We can similarly distribute the $\beta$ in
\begin{align}
\frac{1}{\sqrt\beta} \int_{\Tor^2} \phi_N I_N[v]^2 \dx
&\lesssim \frac{1}{\sqrt\beta} \norm{\phi_N}_{B^{-\epsilon}_{4, 4}}
    \norm{I_N[v]}_{B^{\epsilon}_{8/3, 8/3}}^2 \notag\\
&\lesssim \frac{1}{\sqrt\beta} \norm{\phi_N}_{B^{-\epsilon}_{4, 4}}
    \norm{I_N[v]}_{L^4} \norm{I_N[v]}_{H^1} \notag\\
&\lesssim \frac{C_\rho}{\beta} \norm{\phi_N}_{B^{-\epsilon}_{4,4}}^4
    + \frac \rho \beta \norm{I_N[v]}_{L^4}^4
    + \rho \norm{I_N[v]}_{H^1}^2.
\end{align}

The only problematic term is the $I_N[v]^3$ term.
Here we need to use the mass of the translated field.
We split $I_N[v]$ into its mean $I_N^\circ[v]$ and the remaining oscillatory part $I_N^\perp[v]$.
We can then apply the following Poincaré inequality:

\begin{lemma}\label{thm:2d poincare}
On a periodic domain $[{0}, {L}]^d$, we have
\[
\norm{I_N^\perp[v]}_{L^2}^2
\leq \frac{1}{2 + (2\pi/L)^2} \int_0^N \norm{v_t}_2^2 \dt.
\]
\end{lemma}
\begin{proof}
As in \Cref{thm:2d drift h1}, we can estimate the norm as
\[
\sum_{k \in \Z^d \setminus \{0\}} \frac{1}{2 + C_L \abs k^2} \left[ \int_0^N \sigma_t(k)^2 \dt \right]
    \left[ \int_0^N v_t(k)^2 \dt \right].
\]
Since $\abs k \geq 1$, we can bound the prefactor with $1/(2 + C_L)$.
\end{proof}

We can use Young's inequality (now with explicit constants) and \Cref{thm:2d poincare} to estimate
\begin{align}\label{eq:2d partition using poincare}
\frac{1}{\sqrt\beta} \int_{\Tor^2} I_N[v]^2 I_N^\perp[v] \dx
&\leq \frac{\rho'}{4\beta} \norm{I_N[v]}_{L^4}^4
    + \frac{1}{\rho'} \norm{I_N^\perp[v]}_{L^2}^2 \notag\\
&\leq \frac{\rho'}{4\beta} \norm{I_N[v]}_{L^4}^4
    + \frac{1}{(2 + (2\pi/L)^2) \rho'} \int_0^N \norm{v_t}_2^2 \dt.
\end{align}
Everywhere above, we could choose $\rho$ to be arbitrarily small,
but here $\rho'$ needs to satisfy both $\rho' < 1$ and $(2 + (2\pi/L)^2) \rho' > 2$
in order to match \eqref{eq:2d variational controls}.
This choice is always possible, even without any assumptions on $L$.

Finally, we write the remaining part of the $I_N[v]^3$ term as
\begin{align}\label{eq:2d delta fixed}
&\E \frac{I_N^\circ[v]}{\sqrt\beta} \int_{\Tor^2} I_N[v]^2 \dx \notag\\
=\; &\E \frac{\hat\phi_N(0) + I_N^\circ[v] + (1-\delta')\sqrt\beta}{\sqrt\beta} \norm{I_N[v]}_{L^2}^2
    - (1-\delta') \E \norm{I_N[v]}_{L^2}^2.
\end{align}
(Recall that $\E \hat\phi_N(0) = 0$.)
We fix $\delta'$ so that
\begin{equation}
(1-\delta') \norm{I_N[v]}_{L^2}^2
\leq \frac \rho 2 \int_0^N \norm{v_t}_2^2 \dt.
\end{equation}
The remaining two terms can be bounded from below as
\begin{align}
&\E \frac{-\negpart{ \hat\phi_N(0) + I_N^\circ[v] + (1-\delta')\sqrt\beta }}{\sqrt\beta} \int_{\Tor^2} I_N[v]^2 \dx \notag\\
\geq\; &-C_\rho \E \negpart{ \hat\phi_N(0) + I_N^\circ[v] + (1-\delta')\sqrt\beta }^2
    - \frac \rho \beta \E \norm{I_N[v]}_{L^4}^4. \label{eq:2d kappa fixed}
\end{align}
We now fix $\kappa = C_\rho$ in \eqref{eq:2d partition approx}.
Therefore the first term of \eqref{eq:2d kappa fixed} is cancelled by a matching positive term.

\subsubsection*{Change of Wick ordering}

Let us then still check the correction terms stemming from changing the Wick ordering.
We now use \Cref{thm:wick change} to replace the Wick ordering by
\begin{align}
\frac{1}{4\beta} \wickm{\bullet^4}
&= \frac{1}{4\beta} \wickp{\bullet^4}
    - \frac{6 (C_N^- - C_N^+)}{4\beta} \wickp{\bullet^2} + \frac{3 (C_N^- - C_N^+)^2}{4\beta},
    \label{eq:2d wick change quartic}\\
\frac{1}{\sqrt\beta} \wickm{\bullet^3}
&= \frac{1}{\sqrt\beta} \wickp{\bullet^3} - \frac{3 (C_N^- - C_N^+)}{\sqrt\beta} \bullet.
    \label{eq:2d wick change cubic}
\end{align}
Here the difference of renormalization constants satisfies
\begin{equation}\label{eq:wick difference}
\sup_N \smallabs{C_N^+ - C_N^-}
\leq \frac{1}{2} + \sum_{k \in \Z^2 \setminus \{0\}}
    \abs{\frac{1}{C_L \abs k^2 + 2} - \frac{1}{C_L \abs k^2 - 1}}
< \infty.
\end{equation}

As we substitute $\bullet = \phi + I_N[v]$,
the raw Wick powers $\phi$ and $\wickp{\phi^2}$ vanish under the expectation.
The cross term of \eqref{eq:2d wick change quartic} is bounded by
\begin{equation}
\frac{3 (C_N^- - C_N^+)}{\beta} \E \int_{\Tor^2} \phi I_N[v] \dx
\lesssim \frac{C_\rho}{\beta^2} \norm{\phi}_{\besovinfty^{-\epsilon}}^2
    + \rho \norm{I_N[v]}_{H^1}^2,
\end{equation}
and the other term in \eqref{eq:2d wick change quartic} follows similarly:
\begin{equation}
\frac{3 (C_N^- - C_N^+)}{2 \beta} \E \int_{\Tor^2} I_N[v]^2 \dx
\lesssim \frac{C_\rho}{\beta} + \frac{\rho}{\beta} \norm{I_N[v]}_{L^4}^4 + \rho \norm{I_N[v]}_{L^2}^2.
\end{equation}
Then the final term in \eqref{eq:2d wick change cubic} is estimated as
\begin{equation}
\frac{3 (C_N^- - C_N^+)}{\sqrt\beta} \E \int_{\Tor^2} I_N[v] \dx
\lesssim \frac{C_\rho}{\beta} + \rho \norm{I_N[v]}_{L^2}^2.
\end{equation}
In summary, all of these correction terms could be controlled within the previous framework.

\subsection{Concentration of measure}\label{sec:2d concentration}

To finish the proof of \Cref{thm:2d partition overall}, we want to show that the integral over
\begin{equation}
\badregion_- \cup \badregion_+ \coloneqq \left\{ -\delta'\sqrt\beta < \hat u(0) < \delta'\sqrt\beta  \right\}
\end{equation}
is negligible in the computation of $\pf_N$.
In the following, we denote by $\pf^\flat_N$ the integral over $\goodregion_+$ computed in \Cref{sec:2d partition}
and by $\pf^\sharp_N$ the integral over $\badregion_- \cup \badregion_+$.
The precise result we will prove in this section is:

\begin{theorem}
For $\beta$ sufficiently large, we have
$\pf_N \leq 2 (1 + C/\beta) \pf^\flat_N$.
\end{theorem}

Since $\pf^\flat_N$ is finite, it will be sufficient to show
\begin{equation}\label{eq:2d concentration sufficient}
\frac{\pf^\sharp_N}{\pf^\sharp_N + 2 \pf^\flat_N} < \exp(-C\beta)
\quad\text{for some } C > 0.
\end{equation}
This formulation is useful since the denominator corresponds to the full partition function.
It is equivalent to showing that
\begin{equation}\label{eq:2d concentration variational}
-\frac 1 \beta \left[ \log \pf^\sharp_N - \log \pf_N \right] > C,
\end{equation}
which is set up for the variational formula.

Instead of doing the translation $u = 1 + \varphi$,
we will only add and subtract a quadratic term; that is, we write
\begin{equation}
\pf^\sharp_N = \int_{\badregion_\pm} \!\!\exp(-W_\beta(u)) \diff\Gaussian(u)
\coloneqq \int_{\badregion_\pm} \!\!
    \exp\left(-\int_{\Tor^2} \frac{1}{4\beta} \wick{\varphi^4} - \frac 3 2 \wick{\varphi^2} \dx \right)
    \diff\Gaussian(\varphi).
\end{equation}
With this choice, the Gaussian measure $\diff\Gaussian$ is again of covariance $(2 - \Laplace)\inv$.

We again need to move to the full integration region like in \Cref{thm:2d partition overall}.
We estimate
\begin{equation}
\int_{\badregion_+} \exp(-W_\beta(\varphi)) \diff\Gaussian(\varphi)
\leq \int \exp(-\beta \kappa h(\hat \varphi(0) / \sqrt\beta)) \exp(-W_\beta(\varphi)) \diff\Gaussian(\varphi),
\end{equation}
where $h$ is a smoothed indicator of $\R \setminus {[{-\delta'}, {\delta'}]}$
and $\kappa > 0$ will be fixed in the end.
By this and the Boué--Dupuis formula, the left-hand side of \eqref{eq:2d concentration variational}
is then bounded from above by
\begin{multline}\label{eq:2d concentration diff variational}
\inf_{v} \E \left[ \kappa h\bigg((\hat \phi_N(0) + \frac{\widehat{I_N[w]}(0))}{\sqrt\beta} \bigg)
    + \frac 1 \beta W_\beta(\phi_N + I_N[v]) + \frac{1}{2\beta} \int_0^N \norm{v_t}_2^2 \dt \right]\\
- \inf_{w} \E \left[ \frac 1 \beta W_\beta(\phi_N + I_N[w]) + \frac{1}{2\beta} \int_0^N \norm{w_t}_2^2 \dt \right].
\end{multline}

We need to show that the function $h$ pushes the two infima sufficiently far apart from each other.
Before estimating \eqref{eq:2d concentration diff variational},
let us perform one more simplification.
By the definition of the potential, we have
\begin{equation}
\frac 1 \beta W_\beta(\phi_N + I_N[v])
= W_1 \!\left( \frac{\phi_N}{\sqrt\beta} + \frac{I_N[v]}{\sqrt\beta} \right).
\end{equation}
We can therefore rename $v/\sqrt\beta$ to $v$ in both infima.
We are hence left with
\begin{multline}\label{eq:2d concentration final variational}
\inf_{v} \E \left[ \kappa h\bigg(\frac{\hat \phi_N(0)}{\sqrt\beta} + \widehat{I_N[v]}(0) \bigg)
    + W_1\!\left( \frac{\phi_N}{\sqrt\beta} + I_N[v] \right)
    + \frac{1}{2} \int_0^N \norm{v_t}_2^2 \dt \right]\\
- \inf_{w} \E \left[ W_1\!\left(\frac{\phi_N}{\sqrt\beta} + I_N[w]\right)
    + \frac{1}{2} \int_0^N \norm{w_t}_2^2 \dt \right].
\end{multline}
This formula suggests that passing $\beta \to \infty$ is going to lead to a deterministic limit.
Let us then prove this.
First, we collect an application of the triangle inequality:

\begin{lemma}\label{thm:2d concentration inf diff}
Let $\mathcal G$ be an arbitrary set, and $F_\beta, F$ some maps $\mathcal G \to \R$.
If
\[
\lim_{\beta \to \infty} \sup_{v \in \mathcal G} \abs{F_\beta(v) - F(v)} = 0,
\]
then we also have
\[
\lim_{\beta \to \infty} \abs{ \inf_{v \in \mathcal G} F_\beta(v) - \inf_{w \in \mathcal G} F(w) } = 0.
\]
\end{lemma}
\begin{proof}
It suffices to note that
\[
\inf_{v \in \mathcal G} F_\beta(v)
= \inf_{v \in \mathcal G} \big[ F_\beta(v) - F(v) + F(v) \big]
\leq \sup_{v \in \mathcal G} \big[ F_\beta(v) - F(v) \big] + \inf_{v \in \mathcal G} F(v).
\]
The other required inequality follows by symmetry.
\end{proof}

We use this lemma as follows:
In the second bracket of \eqref{eq:2d concentration final variational} we set
\begin{align}
F_\beta(w) &\coloneqq \E W_1\!\left( \frac{\phi_N}{\sqrt \beta} + I_N[w] \right)
    + \frac 1 2 \int_0^N \norm{w_t}_2^2 \dt, \\
F(w) &\coloneqq \frac 1 4 \int_{\Tor^2} I_N[w]^4 \dx - \frac 3 2 \int_{\Tor^2} I_N[w]^2 \dx
    + \frac 1 2 \int_0^N \norm{w_t}_2^2 \dt.
\end{align}
In the first bracket of \eqref{eq:2d concentration final variational},
we additionally include $h$ in a similar way.
Therefore all the randomness is now contained in the definition of $F_\beta$.
If we can estimate the supremum of $\abs{F_\beta - F}$, then the error in replacing $\inf F_\beta$ by $\inf F$
can be bounded by any $\epsilon > 0$ once $\beta$ is large enough.

For this, we need to be able to commute the supremum and the limit.
The following two results give sufficient uniform convergence bounds.

\begin{lemma}\label{thm:2d concentration good drifts}
There exists $K > 0$ such that the minimizers
of both infima in \eqref{eq:2d concentration final variational} are contained in
\[
\mathcal G \coloneqq \left\{ v \in \mathbb H_a \colon
    \E \left[ \norm{I_N[v]}_4^4 + \int_0^N \norm{v_t}_2^2 \dt \right] < K
    \text{ for all } N \geq 0
\right\}.
\]
This holds for all $\beta$ large enough.
The universial quantifier also applies to the implicit dependency of $I_t$ on $N$.
\end{lemma}
\begin{proof}
The argument in \Cref{sec:2d partition} lets us bound the quartic term of $W_1$ as
\begin{equation}
\E \frac{1}{4 \beta} \int_{\Tor^2} \wick{\left( \frac{\phi_N}{\sqrt\beta} + I_N[v] \right)^4} \dx
\geq -\frac{C_\rho}{\beta^2} + \frac{1 - \rho}{4\beta^2} \E \norm{I_N[v]}_4^4
    - \frac{\rho}{2\beta} \E \int_0^N \norm{v_t}_2^2 \dt.
\end{equation}
In the quadratic term, the Wick power $\wick{\phi_N^2}$ again contributes nothing to the expectation.
We can use a straightforward Young estimate like
\begin{equation}
\E \frac{3}{2\beta} \int_{\Tor^2} 2\phi I_N[v] + I_N[v]^2 \dx
\leq C_\rho + \frac{\rho}{\beta^2} \norm{I_N[v]}_4^4
\end{equation}
for the remaining two terms.
Therefore we have
\begin{equation}\label{eq:2d concentration sup bound}
\begin{split}
&\mathrel{\phantom{\geq}} \E \left[ \frac 1 \beta V_\beta(\phi_N + I_N[v])
    + \frac{1}{2\beta} \int_0^N \norm{v_t}_2^2 \dt \right]\\
&\geq -\frac{C_1}{\beta^2} - C_2
    + \frac{1-\rho}{4\beta^2} \E \norm{I_N[v]}_4^4
    + \frac{1-\rho}{2\beta} \E \int_0^N \norm{v_t}_2^2 \dt.
\end{split}
\end{equation}
In the version that contains $h$,
we also include $\kappa \norm{h}_\infty$ in $C_2$.
The left-hand side of \eqref{eq:2d concentration sup bound}
can be bounded from above by $0$ by plugging in $v = 0$ as before.
As we then rewrite $w = v / \sqrt{2\beta}$, we get
\begin{equation}
\E \left[ \norm{I_N[w]}_4^4 + \int_0^N \norm{w_t}_2^2 \dt \right]
\leq \frac{C_1/\beta^2 + C_2}{1 - \rho},
\end{equation}
which is exactly as stated in the claim.
\end{proof}

\begin{lemma}
In the set $\mathcal G$ defined above we have
\[
\abs{ \E h\bigg( \frac{\hat \phi_N(0)}{\sqrt\beta} + \widehat{I_N[v]}(0) \bigg)
    - h\left(\widehat{I_N[v]}(0) \right) } < \frac{C}{\sqrt\beta},
\]
and
\[
\abs{ \E W_1\!\left( \frac{\phi_N}{\sqrt\beta} + I_N[v] \right)
    - W_1\left( I_N[v] \right) } < \frac{C}{\sqrt\beta}.
\]
Both of these estimates are uniform in $v \in \mathcal G$ and $N > 0$.
\end{lemma}
\begin{proof}
Since $h$ is Lipschitz-continuous, the left-hand side of the first claim is bounded by
\[
C \E \abs{\frac{\hat \phi_N(0)}{\sqrt\beta}}.
\]
Since $\hat \phi_N(0)$ is an ordinary Gaussian random variable,
the expectation is finite and the quotient vanishes with correct rate.

In the second claim, we expand the potentials.
All purely $I_N[v]$-dependent terms are cancelled.
The remaining mixed terms can be estimated as in \Cref{sec:2d partition}.
Any expectations of $\phi_N$ and its Wick powers are bounded as before;
$\norm{I_N[v]}_4^4$ and the integral of $\norm{v_t}_2^2$ are bounded by the definition of $\mathcal G$;
and each term has a prefactor of $\beta^{-1/2}$ or better.
\end{proof}

Now we can apply \Cref{thm:2d concentration inf diff}
to replace the stochastic functionals with their deterministic limits.
The error $\epsilon > 0$ can be chosen freely,
at the expense of large enough $\beta$.
The variational problem \eqref{eq:2d concentration final variational} thus has a lower bound
\begin{multline}\label{eq:2d concentration limit variational}
\inf_{v} \left[ \kappa h(\widehat{I_N[v]}(0))
    + W_1( I_N[v])
    + \frac{1}{2} \int_0^N \norm{v_t}_2^2 \dt \right]\\
- \inf_{w} \left[ W_1(I_N[w])
    + \frac{1}{2} \int_0^N \norm{w_t}_2^2 \dt \right]
- \epsilon.
\end{multline}
Since we assume $L < 2\pi$, only constant functions minimize the potential $W_1$.
Thus by suitably large $\kappa$ we should be able to force $v$ to be near the saddle
whereas $w$ can lie in a potential well.
Let us still prove this rigorously.

\begin{lemma}
Both infima in \eqref{eq:2d concentration limit variational} can be taken over constant functions.
\end{lemma}
\begin{proof}
Since $h$ already only depends on the zero mode, we only need to look at the second infimum.
Let $w = w_0 + w_\perp$, where $w_0$ is the average of $w$.
Then the bracket becomes (abbreviating $I_0 = I_N[w_0]$ and $I_\perp = I_N[w_\perp]$)
\begin{multline}
\int_{\Tor^2} \frac 1 4 I_0^4 + I_0^3 I_\perp + \frac 3 2 I_0^2 I_\perp^2
    + I_0 I_\perp^3 + \frac 1 4 I_\perp^4
    - \frac 3 2 I_0^2 - 3 I_0 I_\perp - \frac 3 2 I_\perp^2 \dx\\
+ \frac 1 2 \int_0^N \norm{w_{0,t}}_2^2 \dt
+ \frac 1 2 \int_0^N \norm{w_{\perp,t}}_2^2 \dt.
\end{multline}
We need to show that the above expression is larger than the equivalent where $I_\perp = 0$:
\begin{equation}
\int_{\Tor^2} \frac 1 4 I_0^4 - \frac 3 2 I_0^2 \dx
+ \frac 1 2 \int_0^N \norm{w_{0,t}}_2^2 \dt.
\end{equation}
The terms with a single power of $I_\perp$ vanish in the integration,
so we only need to control $I_0 I_\perp^3$ and $-3 I_\perp^2 / 2$
with the three positive terms that depend on $I_\perp$.

By the Poincaré-type \Cref{thm:2d poincare}
and the assumption $L < 2\pi$ we have
\begin{equation}
\begin{split}
&\mathrel{\phantom{\geq}} -\frac 3 2 \int_{\Tor^2} I_\perp^2 \dx + \frac 1 2 \int_0^N \norm{w_{\perp,t}}_2^2 \dt\\
&\geq \left( \frac 1 2 - \frac{3}{2(2 + (2\pi/L)^2)} \right) \int_0^N \norm{w_{\perp,t}}_2^2 \dt\\
&\geq 0.
\end{split}
\end{equation}
On the other hand we can use Young's inequality to estimate
\begin{equation}
I_0 I_\perp I_\perp^2 \leq \frac{\lambda}{2} I_0^2 I_\perp^2 + \frac{1}{2\lambda} I_\perp^4.
\end{equation}
This implies (by choosing $\lambda = 3$)
\begin{equation}
I_0 I_\perp I_\perp^2 + \frac 3 2 I_0^2 I_\perp^2 + \frac 1 4 I_\perp^4
\geq \frac{1}{12} I_\perp^4
\geq 0.
\end{equation}
Thus any oscillatory perturbation to a constant function
is only going to grow the expression.
\end{proof}

Now \eqref{eq:2d concentration limit variational} is reduced to
\begin{multline}
\inf_{x \in \R} \inf_{\eta: I_{N}(\eta)=1}
\left[ \kappa h(x) + \frac{\smallabs{\Tor^2} x^4}{4} - \frac{3 \smallabs{\Tor^2} x^2}{2}
    + \frac{x^2}{2} \int_0^N \norm{\eta_t}_2^2 \dt \right]\\
- \inf_{y \in \R} \inf_{\eta: I_{N}(\eta)=1}
\left[ \frac{\smallabs{\Tor^2} y^4}{4} - \frac{3 \smallabs{\Tor^2} y^2}{2}
    + \frac{y^2}{2} \int_0^N \norm{\eta_t}_2^2 \dt \right]
- \epsilon,
\end{multline}
We set $\Xi \coloneqq \smallabs{\Tor^2}\inv \inf_{\eta: I_N(\eta)=1} \int_0^N \norm{\eta_t}_2^2 \dt$.
Note that
\begin{equation}
\int_0^N \norm{\eta_t}_2^2 \dt = 2 \smallabs{\Tor^2}
\text{ for the choice }
\widehat{\eta_t}(k) = \sqrt 2 \sigma_t (0) \I_{k = 0},
\end{equation}
so $\Xi \leq 2$.
We are now left to show that 
\begin{equation}
\begin{split}
&\mathrel{\phantom{>}} \inf_{x \in \R}
\left[ \kappa h(x) + \frac{\smallabs{\Tor^2} x^4}{4} - \frac{(3-\Xi) \smallabs{\Tor^2} x^2}{2} \right]\\
&> \inf_{y \in \R} \left[ \frac{\smallabs{\Tor^2} y^4}{4} - \frac{(3-\Xi) \smallabs{\Tor^2} y^2}{2} \right]. 
\end{split}  
\end{equation}
Notice that the infimum in the second line is attained (only) at $y = \pm \sqrt{3-\Xi}$,
and is given by $-\smallabs{\Tor^2} (3-\Xi)^2 / 4$. This shows that  
\begin{align}
\mathrel{\phantom{=}}&\inf_{x[-\delta,\delta] }
    \left[ \kappa h(x) + \frac{\smallabs{\Tor^2} x^4}{4} - \frac{(3-\Xi) \smallabs{\Tor^2} x^2}{2} \right]\\
&\geq \inf_{x[-\delta,\delta] }
    \left[ \frac{\smallabs{\Tor^2} x^4}{4} - \frac{(3-\Xi) \smallabs{\Tor^2} x^2}{2} \right]\\
&> \inf_{x \in \mathbb{R}}
    \left[ \frac{\smallabs{\Tor^2} x^4}{4} - \frac{(3-\Xi) \smallabs{\Tor^2} x^2}{2} \right].
\end{align}
On the other hand, since $h>0$ on $[-\delta,\delta]^c$ we have 
\begin{align*}
&\mathrel{\phantom{=}} \inf_{x[-\delta,\delta]^c }
    \left[ \kappa h(x) + \frac{\smallabs{\Tor^2} x^4}{4} - \frac{(3-\Xi) \smallabs{\Tor^2} x^2}{2} \right]\\
&\geq \inf_{x[-\delta,\delta]^c } \left[ \kappa h(x) - \smallabs{\Tor^2}\frac{(3-\Xi)^2}{4} \right]\\
&> - \frac{\smallabs{\Tor^2}(3-\Xi)^2}{4} .
\end{align*}
This shows that \eqref{eq:2d concentration limit variational} is strictly positive
once $\epsilon$ is small enough,
and thus \eqref{eq:2d concentration variational} is bounded as required.

\section{Dynamical correction}\label{sec:corr}

In this section, we show that transition state theory does not overestimate
the number of transitions by \eqref{eq:NLW} flow by too much:
If the solution manages to ``climb'' to the saddle,
it will have sufficient ``inertia'' to fall into the opposite well without turning back.
However, we are only able to do this in up to two dimensions,
and not at all for \eqref{eq:SdNLW}.

\begin{remark}
Throughout this section,
we translate time so that the saddle is hit at $t = 0$.
By symmetry, we will also only consider transitions from $\{ \hat u(0) < -1 + \delta \}$
to $\{ \hat u(0) > 1 - \delta \}$.
\end{remark}

Let us first motivate the following analysis by looking at the ODE
\begin{equation}\label{eq:corr zerodim system}
\begin{cases}
    \partial_{tt} u(t) - u(t) &= -u(t)^3,\\
    \hfill \partial_t u(0) &= Q > 0,\\
    \hfill u(0) &= 0.
\end{cases}
\end{equation}
Integrating the equation twice in time, we find that the solution satisfies
\begin{equation}\label{eq:corr zerodim integrated}
u(t) = \int_0^t \int_0^s [u(r) - u(r)^3] \dr \ds + Qt.
\end{equation}
In the region $0 \leq u(t) \leq \delta \ll 1$, we can linearize this as
\begin{equation}
u(t) \sim \int_0^t \int_0^s u(r) \dr \ds + Qt,
\end{equation}
which yields $u(t) \sim Q \sinh(t)$.
Intuitively, the behaviour of $u$ is dominated by the initial velocity $Q$ on the saddle,
followed by a rapid descent into the potential well.
Solving for $\tau$ in $u(\tau) = \delta$ yields $\tau \sim \log(Q\inv)$.

\subsection{The zero mode}

If we project the two-dimensional, non-stochastic equation \eqref{eq:NLW} to the zero mode,
we find that $w \coloneqq \proj_0 u$ solves
\begin{equation}\label{eq:nlw zero mode}
\partial_{tt} w(t) - w(t)
= -w(t)^3 - 3 w(t) \int_{\Tor^d} \wick{(\proj_\perp u)^2} \dx + \int_{\Tor^d} \wick{(\proj_\perp u)^3} \dx.
\end{equation}
Importantly, the first power of the oscillatory part $\proj_\perp u$ does not influence $w$.
If $\wick{(\proj_\perp u)^k}$ is of order $\beta^{-k/2}$
and the initial velocity $Q \coloneqq \partial_t w(0)$ of order $\beta^{-1/2}$,
then \eqref{eq:nlw zero mode} can be regarded as a small perturbation of the zero-dimensional equation
\eqref{eq:corr zerodim system}.

By these heuristics, we can say that a transition
has vanishingly small probability of turning back after crossing the saddle:
\begin{theorem}\label{thm:corr goal}
We fix $\delta > 0$ small enough and define the hitting times
\[
\tau_{+} \coloneqq \inf_{t > 0} \{ w(t) \geq 1 - \delta \}, \quad
\sigma_{+} \coloneqq \inf_{t > 0} \{ w(t) \leq 0 \}
\]
as in \Cref{thm:transmission coefficient}.
Consider the flow of \eqref{eq:NLW} in two spatial dimensions,
with initial data as in \Cref{thm:transmission coefficient}.
Then $\Prob_{\perp}(\tau_{+} < \sigma_{+})\geq 1 - C\beta^{-1/4 + 3\epsilon}$,
\end{theorem}

\begin{remark}
Since the stochastic forcing of \eqref{eq:SdNLW} is of order $\beta^{-1/2}$,
the same analysis does not apply to that equation.
We expect its analysis to require more involved probabilistic techniques.

Similarly, \eqref{eq:NLW} in three dimensions is more complicated to study,
since the $\phi^4_3$ measure is not absolutely continuous with respect to the GFF,
and the dynamics is more complicated due to the additional renormalization.
\end{remark}

The proof has two main steps:
First, we assume the oscillatory part to be small,
and show that the hitting times then satisfy $\tau_+ < \sigma_+$.
After that, we show the assumption to hold with large probability.

We decompose the solution as
\begin{equation}
    u = w + Z + \psi,
\end{equation}
where $w$ is the zero mode governed by \eqref{eq:nlw zero mode},
$Z$ solves the linear wave equation
\begin{equation}\label{eq:corr linear wave}
\begin{cases}
    \partial_{tt} Z(x,t) - (1 + \Laplace) Z(x,t) &= 0,\\
    \hfill \Law(Z(0)) &= \proj_\perp (\phi^4_2(\beta)),\\
    \hfill \Law(\partial_t Z(0)) &= \proj_\perp(\beta^{-1/2}\text{ white noise}),
\end{cases}
\end{equation}
and the nonlinear oscillatory part $\psi$ is left to solve
\begin{equation}
\begin{cases}\label{eq:oscillatory nonlinear}
    \partial_{tt} \psi(x,t) - (1 + \Laplace) \psi(x,t) &= - \proj_\perp \wick{(Z + w + \psi)(x,t)^3},\\
    \hfill \psi(0) &= 0,\\
    \hfill \partial_t \psi(0) &= 0.
\end{cases}
\end{equation}
Note that both $\psi$ and $Z$ are purely oscillatory.

\begin{assumption}\label{ass:moments}
Let $T = 2 \delta^{-1/2} \log(\beta)$.
For $k = 1, 2, 3$ we have the bound
\begin{equation}
\left[ \int_0^T \smallnorm{\wick{Z^k}}_{\besovinfty^{-\epsilon}}^2 \dt \right]^{1/2}
\leq \beta^{-k/2 + k\epsilon},
\end{equation}
and
\begin{equation}\label{eq:corr moments q}
Q \geq \beta^{-3/4 + 3\epsilon}.
\end{equation}
\end{assumption}

\begin{assumption}\label{ass:psi}
With $T$ as above, we assume that
\begin{equation*}
\sup_{0 \leq t \leq T} \norm{\psi(t)}_{H^{1-\epsilon}} \leq \beta^{-1/2+2\epsilon}.
\end{equation*}
\end{assumption}

The choice of $T = 2 \delta^{-1/2} \log(\beta)$ as the upper integration bound
stems from the following \emph{a priori} estimate for the stopping time.
This estimate also implies that $w$ is positive throughout,
meaning that $\tau_+ < \sigma_+$.

\begin{lemma}\label{thm:corr apriori}
If \Cref{ass:moments,ass:psi} hold and $\beta$ is large enough,
then the hitting time satisfies $\tau \leq 2 \delta^{-1/2} \log(\beta)$.
Moreover, $w'(t) > 0$ and
\[
0 < w(t) < \frac{3Q}{2} \sinh(t)
\]
for all $0 < t \leq \tau$.
\end{lemma}
\begin{proof}
Integrating \eqref{eq:nlw zero mode} twice in time,
we find that the counterpart to \eqref{eq:corr zerodim integrated} is
\begin{multline}\label{eq:corr zero mode full}
w(t) = Qt + \int_0^t \int_0^s [w(r) - w(r)^3] \dr \ds\\
    + \int_0^t \int_0^s \left[-3 w(r) \int_{\Tor^2} \wick{(Z+\psi)^2} \dx
        + \int_{\Tor^2} \wick{(Z+\psi)^3} \dx \right] \dr \ds.
\end{multline}
By \Cref{ass:moments} the second line can be bounded with
\begin{equation}
C \int_0^t (T^{1/2} + T) \left[ \beta^{-1+2\epsilon} + \beta^{-3/2+3\epsilon} \right] \dt
\leq C \beta^{-1+3\epsilon} t.
\end{equation}
Now when $\beta$ is large enough, we have $C \beta^{-1+3\epsilon} \leq Q/2$.
Therefore we can estimate
\begin{equation}
\eqref{eq:corr zero mode full}
\geq \frac{Qt}{2} + \int_0^t \int_0^s (1-(1-\delta)^2) w(r) \dr \ds.
\label{eq:corr zero mode estimate}
\end{equation}
This integral equation implies that in the $0 \leq w(t) \leq 1 - \delta$ regime,
\begin{equation}
w(t) \geq \frac{Q}{2\sqrt{2\delta - \delta^2}} \sinh(\sqrt{2\delta - \delta^2} t).
\end{equation}
This shows that $w(t)$ stays positive until the hitting time.
From the equation we can also solve an upper bound for $\tau$,
assuming $\beta$ large enough:
\begin{equation}
\begin{split}
\tau &\leq \frac{1}{\sqrt{2\delta - \delta^2}} \arcsinh\left( \frac{2\delta\sqrt{2\delta - \delta^2}}{Q} \right)\\
&\leq 2 \delta^{-1/2} \log(\beta).
\end{split}
\end{equation}
Here we crudely bounded $\arcsinh(x) \leq \log(3x) \leq 2 \log(x)$ for $x$ large.

If we integrate \eqref{eq:nlw zero mode} only once in time,
the same argument shows that $w'(t) > 0$ in the domain.
Finally, the upper bound for $w(t)$ follows by noting that $w(r) - w(r)^3 \leq w(r)$
and modifying \eqref{eq:corr zero mode estimate} accordingly.
\end{proof}

We now use this estimate to bootstrap a better estimate for the behaviour near the saddle.
We introduce an auxiliary hitting time:
\begin{definition}
With $\delta > 0$ as in \Cref{thm:corr goal}, we define the hitting time
\[
\theta_{+} \coloneqq \inf_{t > 0} \{ w(t) \geq \delta \}.
\]
\end{definition}
The rationale for this auxiliary object is that $\theta_+$ will be of order $\bigO(\log \beta)$,
whereas $\tau_+ - \theta_+$ will be of order $\bigO(1)$.
The next theorem shows that
$\theta_+$ is only an additive perturbation away from the zero-dimensional hitting time
sketched in \eqref{eq:corr zerodim system}.
The difference to $\tau_+$ is then shown in \Cref{thm:corr apriori theta}.

\begin{lemma}\label{thm:corr tau}
If \Cref{ass:moments,ass:psi} hold and $\beta$ is large enough, then
\[
\theta_+ \leq \arcsinh((1-\delta) Q\inv) + 2 \delta^2.
\]
\end{lemma}
\begin{proof}
We can write the evolution of $w$ as the system
\begin{equation}
\left\{
\begin{aligned}
    \partial_t w(t) &= p(t),\\
    \partial_t p(t) &= -V'(t),
\end{aligned}
\right.
\end{equation}
where
\begin{multline}
V(t) = \frac 1 4 w(t)^4 - \frac 1 2 w(t)^2\\
- \underbrace{\int_0^t \left[-3 w(s) \int_{\Tor^2} \wick{(Z+\psi)^2} \dx
    + \int_{\Tor^2} \wick{(Z+\psi)^3} \dx \right] \ds}_{G(t)}.
\end{multline}
Observe that by monotonicity of $w$,
we can bound $\abs{G(t)}$ from above by
\begin{equation}
\mathcal G(w) = w (T^{1/2} \smallnorm{\wick{(Z+\psi)^2}}_{L^2_T \besovinfty^{-\epsilon}(\Tor^2)})
    + T^{1/2} \smallnorm{\wick{(Z+\psi)^3}}_{L^2_T \besovinfty^{-\epsilon}(\Tor^2)}.
\end{equation}
Now let $\bar{w}$ solve the Hamiltonian system 
\begin{equation}
\left\{
\begin{aligned}
    \partial_t \bar{w}(t) &= \bar{p}(t),\\
    \partial_t \bar{p}(t) &= -\bar{V}(w(t)),
\end{aligned}
\right.
\end{equation}
where
\[
\bar{V}(\bar{w}(t))= -\frac 1 4 \bar{w}(t)^4 - \frac 1 2 \bar{w}(t)^2-\mathcal{G}(\bar{w}).
\]
Here $\mathcal{G}$ depends on $t$ only though $\bar{w}(t)$.
Let $\bar\theta_+$ be the first time that $\bar{w} \geq \delta$.
Observe that $\bar{w}\leq w$ which implies $\bar\theta_+ \geq \theta_+$.
The conservation of energy $p^2/2 + \bar{V}$ and the initial conditions $p(0) = Q$, $\bar{V}(0) = 0$ then imply that
\begin{equation}
\frac 1 2 Q^2
= \frac 1 2 p(t)^2 - \frac 1 2 \bar{w}(t)^2 + \frac 1 4 \bar{w}(t)^4 - \mathcal{G},
\end{equation}
which in turn implies
\begin{equation}
\frac{\diff \bar{w}}{\dt}
= \bar{p}(t)
= \sqrt{Q^2 + \bar{w}(t)^2 - \tfrac 1 2 \bar{w}(t)^4 +  \mathcal{G}}.
\end{equation}
We use this to write $\bar\theta_+$ via the line integral
\begin{equation}
\bar\theta_+
= \int_0^{\bar\theta_+} \dt
\leq \int_0^{\delta} \frac{\diff \bar{w}}{\sqrt{Q^2 + \bar{w}^2 - \tfrac 1 2 \bar{w}^4 - \mathcal G}},
\end{equation}

We can then rescale the variable ($\bar w = Qz$) to get
\begin{equation}
\theta_+ \leq\bar\theta_+ \leq \int_0^{\delta Q\inv}
    \frac{\diff z}{\sqrt{1 + z^2 - \tfrac{Q^2}{2} z^4 - Q^{-2} \mathcal G(z)}}.
\end{equation}
We Taylor-expand the function $t \mapsto (1 + z^2 + t)^{-1/2}$ around $t=0$.
The zeroth-order term gives
\begin{equation}
\theta_+ \approx \int_0^{\delta Q\inv} \frac{\diff z}{\sqrt{1 + z^2}} = \arcsinh(\delta Q\inv).
\end{equation}
\Cref{ass:moments,ass:psi} imply that $Q^{-2} \mathcal G$ is bounded by
\begin{equation}
\begin{split}
&\mathrel{\phantom{\leq}}
Q\inv z T^{1/2} \smallnorm{\wick{(Z+\psi)^2}}_{L^2_T \besovinfty^{-\epsilon}(\Tor^2)}
+ Q^{-2} T^{1/2} \smallnorm{\wick{(Z+\psi)^3}}_{L^2_T \besovinfty^{-\epsilon}(\Tor^2)}\\
&\leq \beta^{-1/4} z + \beta^{-\epsilon}.
\end{split}
\end{equation}
Here we took $\beta^{-\epsilon}$ to be small enough to cancel $T^{1/2} = \sqrt{2\delta^{-1/2} \log(\beta)}$.
The remaining terms of the Taylor expansion are thus bounded by
\begin{equation}
\begin{split}
&\mathrel{\phantom{\leq}} \sum_{k \geq 1} \frac{(2k-1)!!}{k!\, 2^k} \int_0^{\delta Q\inv}
    \frac{(\tfrac{Q^2}{2} z^4 + \beta^{-1/4} z + \beta^{-\epsilon})^k \diff z}{(1 + z^2)^{k+1/2}}\\
&\leq \sum_{k \geq 1} \frac{(2k-1)!!}{k!\, 2^k} \int_0^{\delta Q\inv}
    \frac{(Q^{2k} z^{4k} + 2^k \beta^{-k/4} z^{k} + 2^k \beta^{-\epsilon k}) \diff z}
        {\max\{1, z^{2k+1}\}}.
\end{split}
\end{equation}
As we estimate each term of the integral, we finally get an upper error bound
\begin{equation}
\sum_{k \geq 1} \frac{(2k-1)!!}{k!\, 2^k}
    \left[ \delta^{2k} + \beta^{-\epsilon k} \right]
\leq \sum_{k \geq 1} \delta^{2k}
\leq 2 \delta^2,
\end{equation}
which holds when $\beta$ is large enough.
\end{proof}

\begin{lemma}\label{thm:corr w integral}
Under \Cref{ass:moments,ass:psi}, we also have
\[
\int_0^{\theta_+} w(t) \dt \leq 2(\delta + \delta^3).
\]
\end{lemma}
\begin{proof}
By \Cref{thm:corr apriori,thm:corr tau}, we can now estimate
\begin{equation}
\begin{split}
\int_0^{\theta_+} w(t) \dt
&\leq \int_0^{\arcsinh(\delta Q\inv) + 2 \delta^2}
    \min\left\{ \frac{3Q}{2}\sinh(t), \delta \right\} \dt\\
&\leq \frac{3Q}{2} \int_0^{\arcsinh(\delta Q\inv)} \sinh(t) \dt
    + \int_0^{2 \delta^2} \delta \dt\\
&\leq \frac{3Q}{2} \sinh(\arcsinh(\delta Q\inv)) + 2\delta^3.
\qedhere
\end{split}
\end{equation}
\end{proof}

\begin{lemma}\label{thm:corr apriori theta}
Under \Cref{ass:moments,ass:psi} and $\beta$ large,
there exists $C > 0$ independent of $\beta$ such that
\[
\tau_+ \leq \theta_+ + C.
\]
\end{lemma}
\begin{proof}
We can start from \eqref{eq:corr zero mode full} and estimate it from below by
\begin{equation}
\begin{split}
w(t) &= \int_{\theta_+}^t \int_{\theta_+}^s [w(r) - w(r)^3] \dr \ds\\
    &\qquad + \int_0^t \int_0^s \left[-3 w(r) \int_{\Tor^2} \wick{(Z+\psi)^2} \dx
    + \int_{\Tor^2} \wick{(Z+\psi)^3} \dx \right] \dr \ds\\
&\geq \int_{\theta_+}^t \int_{\theta_+}^s w(r) [1 - w(r)] [1 + w(r)] \dr \ds - C \beta^{-1+3\epsilon} t\\
&\geq \frac{1}{2}\delta^2 (t-\theta_+)^2 - \frac \delta 2,
\end{split}
\end{equation}
where in the last line we used that $\beta$ is sufficiently large.
So for $w(t) \leq 1-\delta$ to be satisfied
we must have $\frac{1}{2}\delta^2 (t-\theta_+)^2 \leq 1-\delta/2$,
which yields $\tau_+ \leq \theta_+ + \sqrt{\frac{2 - \delta}{\delta^2}}$.
\end{proof}

\subsection{The oscillatory part}

Let us then consider the second part of the proof.
We show that \Cref{ass:moments} implies \Cref{ass:psi},
and that \Cref{ass:moments} holds with large probability.
For the first claim, we use a continuity argument.
We again split the transition into $\bigO(\log \beta)$ and $\bigO(1)$ parts
represented by $\theta_+$ and $\tau_+$.

The existence of $u$, and consequently of $\psi$, follows from an existing result:

\begin{theorem}\label{thm:corr oh-thomann}
\eqref{eq:NLW} has a mild solution $u \in C(\R_+;\; H^{-\epsilon}(\Tor^2))$
for almost all initial data sampled from the $(\phi^4_2 \times \text{white noise})$ measure.
\end{theorem}
\begin{proof}
Although Oh and Thomann state this result \cite[Theorem~1.5]{oh_invariant_2020}
only in the case of positive mass,
the local well-posedness theory is straightforward to modify:
One substitutes
\begin{equation}
\begin{split}
\partial_{tt} u(x, t) - (m^2 + \Laplace) u(x, t) &= - u(x, t)^3\\
&\Downarrow\\
\partial_{tt} u(x, t) + (1 - \Laplace) u(x, t) &= - u(x, t)^3 + (m^2 + 1) u(x, t),
\end{split}
\end{equation}
and modifies the fixpoint argument accordingly.
Existence of the invariant measure follows from \Cref{sec:2d partition} in the present article,
and the globalization argument then proceeds as is.
(See also \cite[Section~4]{barashkov_invariance_2022} for more details
on the local solution theory.)
\end{proof}

Since this result is global in time,
it does not matter that we only consider solutions starting from the time they hit the saddle.
By writing the equation in Duhamel form,
we see that the nonlinear term $\psi$ satisfies a fixpoint condition $\psi = F[\psi]$.
We can then find a quantitative bound for the solution.
We first derive the bound up to time $\theta_+$:

\begin{lemma}\label{thm:corr psi bound}
Assume that \Cref{ass:moments} holds and that $\beta$ is large enough.
For any $T^\ast \leq 2 \delta^{-1/2} \log(\beta)$, the operator
\[
F[\psi](t) \coloneqq -\int_0^{\min(t, \theta_+)} \frac{\sin(\inormm\nabla (t-s))}{\inormm\nabla}
    \proj_\perp \wick{(w+Z+\psi)^3}(s) \ds
\]
maps $B(0, \beta^{-1/2+\epsilon})$ into $B(0, \tfrac 1 2 \beta^{-1/2+\epsilon})$
in the space $C({[0, T^\ast]};\; H^{1-\epsilon}(\Tor^2))$.
\end{lemma}
\begin{proof}
Minkowski's integral inequality
and $H^{-\epsilon} \to H^{1-\epsilon}$ boundedness of the operator
$\proj_\perp \sin(\inormm\nabla (t-s)) / \inormm\nabla$ imply that
\begin{equation}\label{eq:corr fixpoint 1}
\norm{F[\psi](t)}_{H^{1-\epsilon}}
\lesssim \int_0^{\min(t, \theta_+)} \norm{\proj_\perp \wick{(w+Z+\psi)^3}(s)}_{H^{-\epsilon}} \ds.
\end{equation}
Let us then expand the trinomial.
Since $\proj_\perp w(s)^3$ vanishes, we have nine terms to estimate:
\begin{align}
&\mathrel{\phantom{+}} C \int_0^t \left[ \smallnorm{\wick{Z^3}(s)}_{H^{-\epsilon}}
    + w(s) \smallnorm{\wick{Z^2}(s)}_{H^{-\epsilon}} \right] \ds \label{eq:corr bounds 1}\\
&\mathrel+ C \int_0^t \big[
    \smallnorm{\wick{Z^2}(s)}_{\besovinfty^{-\epsilon}} \smallnorm{\psi(s)}_{H^{1-\epsilon}}
    + w(s) \smallnorm{Z(s)}_{\besovinfty^{-\epsilon}} \smallnorm{\psi(s)}_{H^{1-\epsilon}} \notag\\
&\hspace{4em} + \smallnorm{Z(s)}_{\besovinfty^{-\epsilon}} \smallnorm{\psi(s)}_{H^{1-\epsilon}}^2 \big] \ds
    \label{eq:corr bounds 2}\\
&\mathrel+ C \int_0^t \left[
    w(s) \smallnorm{\psi(s)}_{H^{1-\epsilon}}^2 + \smallnorm{\psi(s)}_{H^{1-\epsilon}}^3 \right] \ds
    \label{eq:corr bounds 3}\\
&\mathrel+ C \int_0^{\min(t, \theta_+)} \left[
    w(s)^2 \smallnorm{Z(s)}_{H^{-\epsilon}} + w(s)^2 \smallnorm{\psi(s)}_{H^{1-\epsilon}} \right] \ds.
    \label{eq:corr bounds 4}
\end{align}
The first line \eqref{eq:corr bounds 1} is bounded by $C \beta^{-1+3\epsilon}$
thanks to \Cref{ass:moments}.
Using also the assumption $\norm{\psi}_{L^\infty H^{1-\epsilon}} < \beta^{-1/2+\epsilon}$,
we see that $\eqref{eq:corr bounds 2} \leq C \beta^{-1+3\epsilon}$.
The $L^\infty$ bound also gives that $\eqref{eq:corr bounds 3} \leq C \beta^{-1+3\epsilon}$.

The last line is the dominating one, since here we get
\begin{equation}
\eqref{eq:corr bounds 4}
\leq C \int_0^{\min(t, \theta_+)} w(s)^2 \smallnorm{Z(s)}_{H^{-\epsilon}} \ds
    + C \delta \norm{\psi}_{L^\infty H^{1-\epsilon}} \int_0^t w(s) \ds.
\end{equation}
By \Cref{ass:moments} and \Cref{thm:corr w integral} we bound this by
$C \delta \beta^{-1/2+\epsilon}$.
By choosing $\delta$ small enough and $\beta$ large enough,
the sum of \eqref{eq:corr bounds 1}--\eqref{eq:corr bounds 4}
can be bounded by $\beta^{-1/2+\epsilon}/2$.
\end{proof}

\begin{remark}
\Cref{thm:corr psi bound} could also be viewed as the boundedness part of a fixpoint argument for $\psi$.
It is uncertain whether contractivity up to $\bigO(\log \beta)$ horizon times could be proved,
since one would need better control of $w$ as a function of $\psi$.
Because of this, we use the \emph{a priori} result \Cref{thm:corr oh-thomann} instead.
\end{remark}

Since $\psi = u - Z - w$ is continuous in time,
\Cref{thm:corr psi bound} implies that $\norm{\psi(t)}_{H^{1-\epsilon}}$
cannot violate \Cref{ass:psi} until after $w$ has already passed $\delta$.
We now extend this up to the time $\tau_+$ where $w$ hits $1-\delta$.

\begin{lemma}
Assume that \Cref{ass:moments} holds,
and that $\psi$ solves \eqref{eq:oscillatory nonlinear} with
$\norm{\psi(\theta_+)}_{H^{1-\epsilon}} \leq \beta^{-1/2+\epsilon}$.
Then for any $t \leq \tau_+$ we have
\begin{equation}
\norm{\psi(t)}_{H^{1-\delta}} \leq \beta^{-1/2+2\epsilon}.\label{eq:bound-oscillatory-short}
\end{equation}
\end{lemma}
\begin{proof}
We again repeat a continuity argument.
Assume that $\norm{\psi(t)}_{H^{1-\epsilon}} \leq \beta^{-1/2+2\epsilon}$
for all $\theta_+ \leq t \leq t_1$, where $t_1 \leq \tau_+$.
We can then again use the Duhamel formulation to estimate
\begin{equation}
\begin{split}
\norm{\psi(t)}_{H^{1-\epsilon}}
&= \eqref{eq:corr fixpoint 1} + \int_{\theta_+}^t \norm{\proj_\perp \wick{(w+Z+\psi)^3}(s)}_{H^{-\epsilon}} \ds\\
&\leq C \beta^{-1+3\epsilon} + C \beta^{-1/2+\epsilon}
    + C \int_{\theta_+}^{t} \norm{\psi(s)}_{H^{1-\epsilon}} \ds\\ 
&\leq C \beta^{-1/2+\epsilon} + C\int_{\theta_+}^{t} \norm{\psi(s)}_{H^{1-\epsilon}} \ds.
\end{split}
\end{equation}
Here we again used the rough estimate \Cref{thm:corr apriori} and $\delta \leq 1$.
Now Grönwall's inequality gives
\[
\sup_{\theta_+ \leq t \leq t_1} \norm{\psi(t)}_{H^{1-\epsilon}}
\leq \exp(C(\tau_+ - \theta_+)) \beta^{-1/2+\epsilon}
\leq \frac{1}{2} \beta^{-1/2+2\epsilon},\]
for $\beta$ large enough.
Thus the set where \eqref{eq:bound-oscillatory-short} is satisfied is open by continuity of $\psi$.
It is clearly closed, also by continuity of $\psi$.
Since it contains $t = \theta_+$ it is the whole interval $[\theta_+, \tau_+]$.  
\end{proof}

The estimate for $\psi$ still depends on a growth bound for the linear part $Z$
stated in \Cref{ass:moments}.
Its validity follows from the following stochastic estimate.

\begin{theorem}\label{thm:corr linear bounds}
The solution $Z$ to the linear wave equation \eqref{eq:corr linear wave} satisfies
\[
\E \smallnorm{\wick{Z^k}}_{L^p([0,T];\; B^{-\epsilon}_{p,p})}^p \leq C_{p,\epsilon} T \beta^{-kp/2}
\]
for $k=1,2,3$, all $T > 0$, and sufficiently large $p$.
\end{theorem}
\begin{proof}
See e.g.\ \cite[Section~4.1]{barashkov_invariance_2022}.
The idea is that $Z$ can be decomposed into a stationary Gaussian free field part
and a more regular part bounded uniformly in time.
(In the cited reference, the presence of weights makes the latter bound also depend on $T$.)
\end{proof}

\begin{corollary}\label{thm:corr proper probability}
\Cref{ass:moments} is satisfied with probability at least $1 - C\beta^{-1/4+3\epsilon}$.
Consequently, \Cref{thm:corr goal} holds.
\end{corollary}
\begin{proof}
For the moments of $Z$, we use Markov's inequality, Besov embedding,
and \Cref{thm:corr linear bounds} to estimate
\begin{align}
\Prob\!\left(
    \smallnorm{\wick{Z^k}}_{L^2([0,T];\; \besovinfty^{-\epsilon})} > \beta^{-k/2+k\epsilon} \right)
&\leq \frac{\E \smallnorm{\wick{Z^k}}_{L^2([0,T];\; \besovinfty^{-\epsilon})}^p}
    {\beta^{-kp/2+kp\epsilon}} \notag\\
&\leq \frac{T^{1/2-1/p} \E \smallnorm{\wick{Z^k}}_{L^p([0,T];\; B^{-\epsilon/2}_{p,p})}^p}
    {\beta^{-kp/2+kp\epsilon}} \notag\\
&\lesssim \frac{T^{1/2}}{\beta^{kp\epsilon}}.
\end{align}
Since we set $T = 2 \delta^{-1/2} \log(\beta)$, we can choose $p$ large enough (depending on $\epsilon$) so that
\begin{equation}
\Prob\!\left( \max_{k=1,2,3}
    \smallnorm{\wick{Z^k}}_{L^2([0,T];\; \besovinfty^{-\epsilon})} \leq \beta^{-k/2+k\epsilon} \right)
\leq 1 - C/\beta.
\end{equation}

On the other hand, for the assumption on $Q$ we need to use Gaussian asymptotics near $0$.
Since $\Law(Q) = \Law(\beta^{-1/2} X \mid X > 0)$ for standard normal variable $X$,
we have
\begin{align}
\Prob(Q \geq \beta^{-3/4+3\epsilon})
&= 1 - \Prob(\abs{X} \leq \beta^{-1/4 + 3\epsilon})\\
&\geq 1 - C \beta^{-1/4 + 3\epsilon}.\notag\qedhere
\end{align}
\end{proof}

\begin{remark}
We have not strived to achieve optimal convergence rate in this section.
The bound on $Q$ is the limiting term in the preceding proof,
as the bound on $\wick{Z^k}$ gives any polynomial rate.
The argument in \Cref{thm:corr tau} restricts the range of exponents
available in \Cref{ass:moments}.
\end{remark}

\section{TST in three dimensions}\label{sec:3d tst}

In three dimensions the $\phi^4$ model requires further renormalization.
The truncated and renormalized quartic potential is
\begin{equation}
\proj_N \wick{u_N^4} - \gamma_N \proj_N \wick{u_N^2} - \delta_N,
\end{equation}
where $\gamma_N$ and $\delta_N$ diverge as $N \to \infty$.
Note that these counterterms are in addition to, not in lieu of, the Wick ordering.
Since $\delta_N$ is a constant, it does not affect the dynamics of \eqref{eq:NLW},
but it simplifies working with the measure.

While our overall goal is to compute the exactly same quotient \eqref{eq:2d eyring-kramers}
as in two dimensions,
three issues complicate the details:
\begin{itemize}
    \item The Gaussian free field has only $-1/2 - \epsilon$ derivatives of regularity;
    \item Further Wick powers of the Gaussian free field have even worse regularity;
    \item The additional counterterms yield more terms to estimate as we substitute $u = \varphi + 1$
    and as we change the mass of the Wick ordering.
\end{itemize}
The first two points have been addressed in the literature on stochastic quantization
of the $\phi^4_3$ measure;
in the following, we reuse as much of Barashkov and Gubinelli's article \cite{barashkov_variational_2020}
as possible.

Let us first establish the form of the partition function.
The following lemma is the direct analogue of \Cref{thm:2d partition function form}.
The bracketed term in $V_\beta$ matches the Hamiltonian studied in \cite{barashkov_variational_2020}.

\begin{lemma}\label{thm:3d partition function form}
The partition function can now be written as
\begin{equation*}
\frac{\pf_N}{2} =
\sqrt{\prod_{\abs k \leq N} \frac{2\pi}{\beta L^3 (2 + C_L \abs k^2)}}
\exp\left(\! -\frac{3 C_N L^3}{2} + \frac{\beta L^3}{4} \right)
\int_{\halfspace_+} \!\!
    \exp\left( -V_\beta(\varphi) \right)
    \diff\Gaussian(\varphi),
\end{equation*}
where the potential is
\begin{multline*}
V_\beta(\varphi)
= \int_{\Tor^3} \big[ \lambda \wickm{\varphi_N^4} - \lambda^2 \gamma_{N,-} \wickm{\varphi_N^2} - \delta_{N,-} \big]\\
    + 2 \sqrt\lambda \wickm{\varphi_N^3}
    - \lambda^{3/2} \gamma_{N,-} \varphi
    - \frac \lambda 4 \gamma_{N,-} \dx.
\end{multline*}
Here $\lambda = 1/(4\beta)$ is chosen to match the notation of \cite{barashkov_variational_2020}.
\end{lemma}
\begin{proof}
We substitute $u = 1 + \varphi$ into the Hamiltonian
\begin{equation}
H(u) \coloneqq \frac 1 4 \wick{u_N^4} - \lambda^2 \gamma_N \wick{u_N^2} - \frac{\delta_N}{\beta}.
\end{equation}
The modification of the $\wick{u^4}$ and Gaussian terms proceeds
as in \Cref{thm:2d partition function form},
and the quadratic correction term is expanded with the binomial formula.
As we then perform the rescaling by $\sqrt\beta$, we get the claimed formula.
\end{proof}

\begin{theorem}\label{thm:3d partition overall}
We fix the divergent constants as
\begin{gather*}
\gamma_{N, \bullet} = \E_\bullet \int_{\Tor^3} \int_0^N (J_t \W_t^2)^2 \dt \dx, \text{ and}\\
\delta_{N, \bullet} = \E_\bullet \bigg[
{-\frac{\lambda^2}{2}} \int_0^N (J_t \W_t^3)^2 \dt
+ \frac{\lambda^3}{2} \W_N^2 (\W_N^{[3]})^2
- 4\lambda^4 \W_N (\W_N^{[3]})^3 \bigg],
\end{gather*}
where we use the notation defined below in \eqref{eq:3d gaussian notation}.
These constants depend on the covariance of the Gaussian field.
Then for $\beta$ large enough we have
\begin{equation*}
\pf_N = 2
\sqrt{\prod_{\abs k \leq N} \frac{2\pi}{\beta L^3 (2 + C_L \abs k^2)}}
\exp\left(\! -\frac{3 C_N L^3}{2} + \frac{\beta L^3}{4} \right)
(1 + \bigO(\beta\inv)).
\end{equation*}
Note that the constant $\delta_N$ depends nonlinearly on $\beta$.
\end{theorem}
\begin{proof}
We again do an approximation where we first restrict to a slightly smaller ``good'' set $\goodregion_+$;
compare with \Cref{thm:2d partition overall}.
This again introduces a $\kappa \negpart{\hat\varphi(0) + (1-\delta') \sqrt\beta}$ term.
The upper bound will be proved in \Cref{sec:3d partition upper},
and a lower bound consequently follows in \Cref{sec:3d partition lower}.

To show that the remaining ``bad'' region $\badregion_+$ contributes negligibly,
we look back to \Cref{sec:2d concentration}.
The proof translates to three dimensions verbatim.

We will again need to change the Wick ordering to match the positive mass.
This again introduces the second-order term $\wickm{\varphi^4_N} - \wickp{\varphi^4_N}$.
Now we also need to estimate terms like
$\delta_{N,-} - \delta_{N,+}$
and $(\gamma_{N,-} \wickm{\varphi^2_N}) - (\gamma_{N,+} \wickp{\varphi^2_N})$
stemming from the additional renormalization.
All this is done in \Cref{sec:3d wick}.
\end{proof}

The argument for the saddle integral
again follows from a straightforward modification of the stochastic quantization argument below;
this is the part already done in \cite{barashkov_variational_2020} except for the fixed zero mode.
Thus we have:

\begin{theorem}\label{thm:3d saddle}
The saddle integral defined analogously to \Cref{thm:2d saddle} satisfies
\[
\mathcal I_N = \sqrt{\prod_{0 < \abs k \leq N} \frac{2\pi}{\beta L^3 (C_L \abs k^2 - 1)}}
(1 + \bigO(\beta\inv)).
\]
\end{theorem}

Once these estimates are established,
the rest of the computation required for \Cref{thm:main tst}
proceeds as in the beginning of \Cref{sec:2d tst}.
The only difference is the dimension of products,
but the argument in \Cref{thm:prefactor renormalization} holds also in three dimensions.

\subsection{Upper bound for partition function}\label{sec:3d partition upper}

We introduce the following further notation to match \cite{barashkov_variational_2020}:
\begin{equation}\label{eq:3d gaussian notation}
    \W_T \coloneqq \phi_T, \quad
    \W_T^2 \coloneqq 12 \wickp{\phi_T^2}, \quad
    \W_T^3 \coloneqq 4 \wickp{\phi_T^3}, \quad
    \W_T^{[3]} \coloneqq I_T [J_t \W_t^3].
\end{equation}
With this notation, the variational problem becomes estimating
\begin{equation}\label{eq:3d variational everything}
\begin{split}
&\kappa \E \negpart{\hat\W_N(0) + \widehat{I_N[v]}(0) + (1-\delta')\sqrt\beta}^2\\
+\; &\E\int_{\Tor^3} \bigg[ \lambda \W_N^3 I_N[v]
    + \frac \lambda 2 \W_N^2 I_N[v]^2 + 4\lambda \W_N I_N[v]^3\\
&\qquad {}- 2 \lambda^2 \gamma_{N,+} \W_N I_N[v] - \lambda^2 \gamma_{N,+} I_N[v]^2 - \delta_{N,+} \bigg] \dx\\
+\; &\E\int_{\Tor^3} \bigg[
    \frac{\lambda^{1/2}}{2} \W_N^2 I_N[v]
    + 6\lambda^{1/2} \W_N I_N[v]^2 + 2\lambda^{1/2} I_N[v]^3\\
&\qquad {}- \lambda^{3/2} \gamma_{N,+} I_N[v] - \frac{\lambda \gamma_{N,+}}{4}
    \bigg] \dx\\
+\; &\E\left[ \lambda \int_{\Tor^3} I_N[v]^4 \dx
    + \frac 1 2 \int_0^N \norm{v_t}_2^2 \dt \right]
+ \text{change-of-mass terms}.
\end{split}
\end{equation}
Here we have already converted the Wick ordering to the $(2 - \Laplace)\inv$ covariance
and removed the pure Wick powers that have zero expectation
(\Cref{thm:wick change,thm:wick expectation}).
The change of mass introduces additional terms
(of at most second order in $\phi$ and $v$),
which we defer to \Cref{sec:3d wick}.

\begin{remark}
For the rest of this section, the covariance of the Gaussian field is always $(2-\Laplace)\inv$.
Therefore $\gamma_N$ and $\delta_N$ implicitly refer to $\gamma_{N,+}$ and $\delta_{N,+}$ respectively.
\end{remark}

Due to the appearance of $\gamma_N$ and $\delta_N$,
setting $v=0$ will not give a convergent upper bound.
The idea in \cite{barashkov_variational_2020} is to perform successive
changes of variables in the drift $v$.
As we substitute this ansatz into the formula and do some manipulation,
we find purely stochastic terms that will be cancelled by $\gamma_N$ and $\delta_N$,
and the remainder satisfies good bounds.

We update the ansatz in a few steps, more precisely at
\eqref{eq:3d ansatz 1}, \eqref{eq:3d ansatz 2}, and \eqref{eq:3d ansatz 3}.
This happens as part of finding the upper bound for $\pf_N$ (lower bound for the infimum);
a lower bound is then deduced in \Cref{sec:3d partition lower}.

As in the two-dimensional case, we use the last bracket of \eqref{eq:3d variational everything}
to control the remaining terms.
\Cref{thm:2d drift h1} applies here as well,
so we can use $\norm{I_N[v]}_{H^1}^2$ in lieu of the integral over $t$.
We also use the following stochastic bounds from \cite[Lemma~4]{barashkov_variational_2020}.
Further stochastic estimates are introduced as used.

\begin{theorem}\label{thm:3d gaussian moments}
In addition to the estimates in \Cref{thm:wick gff moments},
we have for all $1 \leq p < \infty$ the bounds
\[
\E \norm{J_T \W_T}_{\besovinfty^{-1/2-\epsilon}(\Tor^3)}^p < \infty, \quad\text{and}\quad
\E \smallnorm{\W_T^{[3]}}_{\besovinfty^{1/2-\epsilon}(\Tor^3)}^p < \infty.
\]
These bounds are uniform in $N$ and $T$.
\end{theorem}

\subsubsection*{Quartic terms}

Let us begin by outlining the argument in \cite[Lemma~5]{barashkov_variational_2020},
since it will be important for construction of the drift.
This corresponds to estimating the first bracket in \eqref{eq:3d variational everything}
by a small multiple of the last bracket.

\bigskip\emph{The $\W_N^3 I_N[v]$ term.}
By the Itô formula we can write
\begin{equation}
\E \int_{\Tor^3} \W_N^3 I_N[v] \dx
= \E \int_0^N \!\! \int_{\Tor^3} \W_t^3 \partial_t (I_t[v]) \dx \dt
    + \E \int_0^N \!\! \int_{\Tor^3} I_t[v] \diff\W_t^3 \dt.
\end{equation}
The latter integral is a martingale, and thus vanishes in expectation.
In the first integral we note that $\partial_t I_t[v] = J_t v_t$
and then use self-adjointness of $J_t$.
These observations give us
\begin{equation}\label{eq:3d w3iv 1}
\lambda \E \int_{\Tor^3} \W_N^3 I_N[v] \dx
= \lambda \E \int_0^N \!\! \int_{\Tor^3} (J_t \W_t^3) v_t \dt \dx.
\end{equation}
Let us now update our ansatz as
\begin{equation}\label{eq:3d ansatz 1}
v_t = -\lambda J_t \W_t^3 + w_t.
\end{equation}
This means that
\begin{equation}\label{eq:3d w3iv 2}
\frac 1 2 \int_0^N \norm{v_t}_2^2 \dt
= \int_0^N \int_{\Tor^3} \frac{\lambda^2}{2} (J_t \W_t^3)^2 - \lambda (J_t \W_t^3) w_t + w_t^2 \dx \dt.
\end{equation}
Combining \eqref{eq:3d w3iv 1} and \eqref{eq:3d w3iv 2} gives
\begin{align}
&\lambda \E \int_{\Tor^3} \W_N^3 I_N[v] \dx + \frac 1 2 \E \int_0^N \norm{v_t}_2^2 \dt \notag\\
=\; &{-\frac{\lambda^2}{2}} \E \int_0^N \!\! \int_{\Tor^3} (J_t \W_t^3)^2 \dx
    + \frac 1 2 \E \int_0^N \norm{w_t}_2^2 \dt.
\end{align}
The first term is purely stochastic and is cancelled by the first term of $\delta_N$.

\begin{remark}\label{rem:3d partition control}
The second term replaces the $\norm{v_t}_2^2$ term.
This means that we will now try to bound the remaining terms by
\[
\rho \E\left[ \lambda \int_{\Tor^3} I_N[v]^4 \dx + \frac 1 2 \int_0^N \norm{w_t}_2^2 \dt \right]
\]
for $\rho > 0$ small.
Like above, we may also use $\norm{I_N[w]}_{H^1}^2$ as the second bounding term.
We will further update the ansatz below,
but \Cref{thm:3d partition w bound}
guarantees that $\norm{I_N[w]}_{H^{1-\epsilon}}^2$ remains controlled
also with respect to the updated drift
(although this requires dropping the regularity to $1-\epsilon$).
\end{remark}

\bigskip\emph{The $\W_N I_N[v]^3$ term.}
Let us recall that we denote $I_N [J_t \W_t^3]$ by $\W_N^{[3]}$,
and that this object has
$1/2 - \epsilon$ derivatives of regularity.

As we substitute the updated ansatz \eqref{eq:3d ansatz 1}, this term becomes
\begin{multline}\label{eq:3d partition wv3}
\E \int_{\Tor^3} \bigg[ -4\lambda^4 \W_N (\W_N^{[3]})^3
    + 12\lambda^3 \W_N (\W_N^{[3]})^2 I_N[w]\\
    - 12\lambda^2 \W_N (\W_N^{[3]}) I_N[w]^2
    + 4\lambda \W_N I_N[w]^3 \bigg] \dx.
\end{multline}
We cancel the first term with $\delta_N$, and the rest are small:

\begin{lemma}
For any $\rho > 0$, we can bound the rest of \eqref{eq:3d partition wv3} by
\[
\lambda C_\rho + \rho \norm{I_N[w]}_{H^{1-\epsilon}}^2 + \rho\lambda \norm{I_N[v]}_4^4.
\]
\end{lemma}
\begin{proof}
\cite[Lemma~22]{barashkov_variational_2020}.
\end{proof}

\bigskip\emph{The $\W_N^2 I_N[v]^2$ term.}
As we substitute the ansatz into this term, we get
\begin{equation}\label{eq:w2i2 with ansatz}
\begin{split}
&\frac \lambda 2 \int_{\Tor^3} \W_N^2 I_N[v]^2 \dx\\
=\; &\frac{\lambda^3}{2} \int_{\Tor^3} \W_N^2 (\W_N^{[3]})^2 \dx
    - \lambda^2 \int_{\Tor^3} \W_N^2 \W_N^{[3]} I_N[w] \dx
    + \frac \lambda 2 \int_{\Tor^3} \W_N^2 I_N[w]^2 \dx.
\end{split}
\end{equation}
The first term is again purely stochastic and cancelled by $\delta_N$.
However, the second term is not a well-defined product due to lack of regularity.
We will decompose it as a paraproduct
\begin{equation}\label{eq:w2i2 decomposition}
-\lambda^2 \int_{\Tor^3} (\W_N^2 \paral \W_N^{[3]}) I_N[w]
    + (\W_N^2 \reson \W_N^{[3]}) I_N[w]
    + (\W_N^2 \parar \W_N^{[3]}) I_N[w] \dx.
\end{equation}
We further rewrite the last term of \eqref{eq:w2i2 decomposition} with the ansatz \eqref{eq:3d ansatz 1} as
\begin{align}
&\mathrel{\phantom{=}} -\lambda^2 \int_{\Tor^3} (\W_N^2 \parar \W_N^{[3]}) I_N[w] \dx \notag\\
&= \lambda \int_{\Tor^3} (\W_N^2 \parar I_N[v]) I_N[w]
    - (\W_N^2 \parar I_N[w]) I_N[w] \dx.
\end{align}
We can now show that some of these terms in \eqref{eq:w2i2 with ansatz} are small:

\begin{lemma}
For any $\rho > 0$ we can bound
\[
\E \abs{ \int_{\Tor^3} \frac \lambda 2 \W_N^2 I_N[w]^2
    - \lambda^2 (\W_N^2 \paral \W_N^{[3]}) I_N[w]
    - \lambda (\W_N^2 \parar I_N[w]) I_N[w] \dx }
\]
from above by
\[
\lambda C_\rho + \rho \norm{I_N[w]}_{H^{1-\epsilon}}^2 + \rho\lambda \norm{I_N[v]}_4^4.
\]
\end{lemma}
\begin{proof}
\cite[Lemma~19]{barashkov_variational_2020}.
In fact the proof gives a much stronger power of $\lambda$.
There is one more preparatory step,
which is decomposing $(\W_N^2 I_N[w]) I_N[w]$ as a paraproduct to be merged with the last term.
The trilinear operator involved is defined in
\cite[Proposition~10]{barashkov_variational_2020}.
\end{proof}

This lemma leaves us to estimate
\begin{equation}\label{eq:w2i2 before renorm}
-\lambda^2 \int_{\Tor^3} (\W_N^2 \reson \W_N^{[3]}) I_N[w] \dx
+ \lambda \int_{\Tor^3} (\W_N^2 \parar I_N[v]) I_N[w] \dx.
\end{equation}
We need to use the additional renormalization terms to counter these terms.
Let us first plug the ansatz into the counterterm
\begin{equation}
-2 \lambda^2 \int_{\Tor^3} \gamma_N \W_N I_N[v] \dx
= 2 \lambda^3 \int_{\Tor^3} \gamma_N \W_N \W_N^{[3]} \dx
    - 2 \lambda^2 \int_{\Tor^3} \gamma_N \W_N I_N[w] \dx.
\end{equation}
The first term has zero expectation,
since $\gamma_N$ is a finite constant (when $N$ is finite)
and the diagram
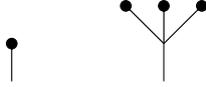
\begin{center}
\begin{tikzpicture}
    \begin{scope}
        \coordinate (A1) at (0,0.5);
        \filldraw (0,0) -- (A1) circle[whitenoise];
    \end{scope}
    \begin{scope}[xshift=2cm, yshift=0.5cm]
        \coordinate (B1) at (-0.5, 0.5);
        \coordinate (B2) at (0, 0.5);
        \coordinate (B3) at (0.5, 0.5);
        \filldraw (0,0) -- (B1) circle[whitenoise];
        \filldraw (0,0) -- (B2) circle[whitenoise];
        \filldraw (0,0) -- (B3) circle[whitenoise];
        \draw (0,-0.5) -- (0,0);
    \end{scope}
\end{tikzpicture}
\end{center}
of $\W_N \W_N^{[3]}$ cannot be contracted.
The second term is combined with the resonant product as follows:

\begin{lemma}
We can estimate
\[
\lambda^2 \E \abs{ \int_{\Tor^3} \left( (\W_N^2 \reson \W_N^{[3]}) + 2 \gamma_N \W_N \right) I_N[w] \dx }
\lesssim \lambda^2 C_\rho + \rho \norm{I_N[w]}_{H^{1-\epsilon}}^2.
\]
\end{lemma}
\begin{proof}
\cite[Lemma~24]{barashkov_variational_2020}.
\end{proof}

Before we can control the $(\W_N^2 \parar I_N[v]) I_N[w]$ term,
we will need to introduce another smoothing operator.
We use it to split the drift into a stochastically controlled part and a smoother part.

\begin{definition}
The operator $\Theta_t \colon L^2 \to L^2$ satisfies the following properties:
\begin{itemize}
    \item It is a Fourier multiplier with symbol in $C^\infty(\R^3)$, bounded uniformly in $t$;
    \item The composed operator $\Theta_t J_s$ vanishes whenever $s \geq t$;
    \item When $t > T_0$, the operator $1 - \Theta_t$ localizes frequency
        to an annulus of inner and outer radii proportional to $t$.
\end{itemize}
\end{definition}

The construction of $\Theta_t$ is presented in \cite[Eq.~19]{barashkov_variational_2020}.
An immediate consequence of the definition is that
\begin{equation}\label{eq:3d theta operator identity}
\Theta_t I_t[v] = \int_0^t \Theta_t J_s v_s \ds = \int_0^N \Theta_t J_s v_s \ds = \Theta_t I_N[v].
\end{equation}
This property simplifies working with the Itô formula,
whereas the spectral localization of $1 - \Theta_t$ gives good bounds
on regularized latter term of \eqref{eq:w2i2 before renorm}:

\begin{lemma}
For any $\rho > 0$, there exists $\delta > 0$ such that
\[
\begin{split}
&\mathrel{\phantom{\lesssim}} \lambda \int_{\Tor^3} (\W_N^2 \parar (1 - \Theta_N) I_N[v]) I_N[w] \dx\\
&\lesssim N^{-\delta} \left[
    \lambda C_\rho + \rho\lambda \norm{I_N[v]}_4^4 + \rho \norm{I_N[w]}_{H^{1-\epsilon}}^2 \right].
\end{split}
\]
\end{lemma}
\begin{proof}
\cite[Lemma~20]{barashkov_variational_2020}.
\end{proof}

That leaves us to consider
\begin{equation}\label{eq:3d partition paracontrol}
\lambda \int_{\Tor^3} (\W_N^2 \parar \Theta_N I_N[v]) I_N[w] \dx.
\end{equation}
We again use the Itô formula.
With some abuse of notation,
\begin{equation}
\lambda \int_0^N \int_{\Tor^3} ((\partial_t \W_t^2) \parar \Theta_t I_t[v]) I_t[w] \dx \dt
\end{equation}
is a martingale and thus has a vanishing expectation.
On the other hand, by \eqref{eq:3d theta operator identity} we can write
\begin{equation}
\lambda \int_0^N \!\! \int_{\Tor^3} (\W_t^2 \parar \partial_t (\Theta_t I_t[v])) I_t[w] \dx \dt
= \lambda \int_0^N \!\! \int_{\Tor^3} (\W_t^2 \parar (\partial_t \Theta_t) I_N[v]) I_t[w] \dx \dt,
\end{equation}
and this is bounded in expectation:

\begin{lemma}
For any $\rho > 0$ we have
\[
\lambda \E \abs{\int_0^N \!\! \int_{\Tor^3} (\W_t^2 \parar (\partial_t \Theta_t) I_N[v]) I_t[w] \dx \dt}
\leq \lambda C_\rho + \rho \norm{I_N[w]}_{H^{1-\epsilon}}^2 + \rho\lambda \norm{I_N[v]}_4^4.
\]
\end{lemma}
\begin{proof}
\cite[Lemma~21]{barashkov_variational_2020}.
Essentially, the construction of $\Theta_t$ implies that $\partial_t \Theta_t$
has $B^{s+\delta}_{p,r} \to B^{s}_{p,r}$ operator norm bounded by $C \inorm t^{-1-\delta}$.
This estimate lets us get rid of the ``time'' integral
as $\dt / \inorm t^{1+\delta}$ is a finite measure.
In spatial variables the estimation is then similar to previous lemmas.
\end{proof}

Now the final term remaining of \eqref{eq:3d partition paracontrol} is
\begin{equation}\label{eq:3d partition paracontrol 2}
\lambda \int_0^N \!\! \int_{\Tor^3} (\W_t^2 \parar \Theta_t I_N[v]) J_t w_t \dx,
\end{equation}
where we used $\partial_t I_t[w] = J_t w_t$.
We will update the ansatz again as in \eqref{eq:3d ansatz 1}.
We set
\begin{equation}\label{eq:3d ansatz 2}
w_t = -\lambda J_t (\W_t^2 \parar \Theta_t I_N[v]) + h_t,
\end{equation}
and use the quadratic $\norm{w_t}_2^2$ term to cancel \eqref{eq:3d partition paracontrol 2},
leaving us with
\begin{equation}\label{eq:3d jw2 parar theta}
-\frac{\lambda^2}{2} \int_0^N \!\! \int_{\Tor^3} [J_t (\W_t^2 \parar \Theta_t I_N[v])]^2 \dx \dt
+ \frac 1 2 \int_0^N \norm{h_t}_2^2 \dt.
\end{equation}
We will show below in \Cref{thm:3d partition w bound}
that bounds that use $\norm{I_N[w]}_{H^{1-\epsilon}}^2$ still remain valid
despite the updated ansatz.

The expression is now set up for using the $-\lambda^2 \gamma_N I_N[v]^2$ counterterm.
In it, we again decompose $I_N[v] = \Theta_N I_N[v] + (1-\Theta_N) I_N[v]$.
We will only use the $-\lambda^2 \gamma_N (\Theta_N I_N[v])^2$ part for renormalization,
and show the rest to be small.
We write
\begin{multline}\label{eq:3d gamma theta renorm}
\lambda^2 \int_{\Tor^3} \gamma_N (\Theta_N I_N[v])^2 \dx
= \lambda^2 \int_{\Tor^3} \int_0^N (\partial_t \gamma_t) (\Theta_t I_N[v])^2 \dt \dx\\
+ 2\lambda^2 \int_{\Tor^3} \int_0^N \gamma_t \, ((\partial_t \Theta_t) I_N[v]) \, \Theta_t I_N[v] \dt \dx.
\end{multline}

\begin{lemma}
For any $\rho > 0$ we can estimate the combination of \eqref{eq:3d jw2 parar theta}
with the first part of \eqref{eq:3d gamma theta renorm},
\[
\E \abs{\lambda^2 \int_0^N \!\! \int_{\Tor^3} \frac 1 2 [J_t (\W_t^2 \parar \Theta_t I_N[v])]^2
    + (\partial_t \gamma_t) (\Theta_t I_N[v])^2 \dx \dt},
\]
from above by
\[
\lambda^2 C_\rho + \rho \norm{I_N[w]}_{H^{1-\epsilon}}^2 + \rho\lambda \norm{I_N[v]}_4^4.
\]
\end{lemma}
\begin{proof}
\cite[Lemma~24]{barashkov_variational_2020}.
\end{proof}

The remaining parts of the counterterm are small:

\begin{lemma}
For any $\rho > 0$,
\begin{multline*}
\lambda^2 \E \bigg| \int_{\Tor^3} \gamma_N ((1-\Theta_N) I_N[v])^2
    + 2\gamma_N (\Theta_N I_N[v]) \, ((1-\Theta_N) I_N[v])\\
    + 2 \int_0^N \gamma_t ((\partial_t \Theta_t) I_N[v]) \, (\Theta_t I_N[v]) \dt \dx \bigg|
\end{multline*}
is bounded from above by
\[
\lambda^2 C_\rho + \rho \norm{I_N[w]}_{H^{1-\epsilon}}^2 + \rho\lambda \norm{I_N[v]}_4^4.
\]
\end{lemma}
\begin{proof}
\cite[Lemma~23]{barashkov_variational_2020}.
\end{proof}

\subsubsection*{Cubic terms}

Let us then estimate the second bracket in \eqref{eq:3d variational everything}.
This bracket contains two renormalizing terms.
Importantly, we may no longer cancel purely stochastic terms with $\delta_N$, as we will need to be able to compare
$\delta_N^+$ and $\delta_N^-$ in \Cref{sec:3d wick}.
Recall that $\delta_N^-$ is used in the context of the saddle integral,
where there are no cubic terms to estimate.

\bigskip\emph{The $\W_N I_N[v]^2$ term.}
This term is easiest of the three since $\W_N$ is not too irregular.
As we put the first ansatz \eqref{eq:3d ansatz 1} into $I_N[v]$, we get
\begin{equation}\label{eq:3d w i2}
\E \int_{\Tor^3} 6\lambda^{5/2} \W_N (\W_N^{[3]})^2
    - 12 \lambda^{3/2} \W_N \W_N^{[3]} I_N[w] + 6 \lambda^{1/2} \W_N I_N[w]^2 \dx.
\end{equation}
The first term vanishes by Wick's theorem since the corresponding diagram
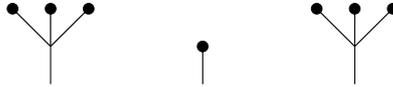
\begin{center}
\begin{tikzpicture}
    \begin{scope}
        \coordinate (A1) at (0,0.5);
        \filldraw (0,0) -- (A1) circle[whitenoise];
    \end{scope}
    \begin{scope}[xshift=-2cm, yshift=0.5cm]
        \coordinate (B1) at (-0.5, 0.5);
        \coordinate (B2) at (0, 0.5);
        \coordinate (B3) at (0.5, 0.5);
        \filldraw (0,0) -- (B1) circle[whitenoise];
        \filldraw (0,0) -- (B2) circle[whitenoise];
        \filldraw (0,0) -- (B3) circle[whitenoise];
        \draw (0,-0.5) -- (0,0);
    \end{scope}
    \begin{scope}[xshift=2cm, yshift=0.5cm]
        \coordinate (C1) at (-0.5, 0.5);
        \coordinate (C2) at (0, 0.5);
        \coordinate (C3) at (0.5, 0.5);
        \filldraw (0,0) -- (C1) circle[whitenoise];
        \filldraw (0,0) -- (C2) circle[whitenoise];
        \filldraw (0,0) -- (C3) circle[whitenoise];
        \draw (0,-0.5) -- (0,0);
    \end{scope}
\end{tikzpicture}
\end{center}
cannot be contracted.

The second term needs to be estimated as a paraproduct decomposition.
First, we consider
\begin{equation}
\begin{split}
&\mathrel{\phantom{\lesssim}} \lambda^{3/2} \E \int_{\Tor^3} (\W_N \paral \W_N^{[3]}) I_N[w] \dx\\
&\lesssim \lambda^{3/2} \E \smallnorm{\W_N \paral \W_N^{[3]}}_{H^{-2\epsilon}}
    \norm{I_N[w]}_{H^{1-\epsilon}}\\
&\lesssim \lambda^{3/2} \E \norm{\W_N}_{H^{-1/2-\epsilon}} \smallnorm{\W_N^{[3]}}_{H^{1/2-\epsilon}}
    \norm{I_N[w]}_{H^{1-\epsilon}},
\end{split}
\end{equation}
which is again finished with Young's inequality.
The prefactor contains $\beta^{-3/2}$.
Similarly, $\W_N \parar \W_N^{[3]}$ has regularity $-1/2-\epsilon$,
which is enough for the product with $I_N[w]$ to make sense.
This leaves us with $(\W_N \reson \W_N^{[3]}) I_N[w]$.
It is, however, straightforward due to the following stochastic estimate:

\begin{lemma}\label{thm:3d w1w3 estimate}
We have the uniform-in-$N$ estimate
\[
\E \smallnorm{\W_N \reson \W_N^{[3]}}_{\besovinfty^{-\epsilon}}^p \lesssim 1.
\]
\end{lemma}
\begin{proof}
This is stated in \cite[Lemma~25]{barashkov_variational_2020},
and is a variant of other results there.
It is also discussed in \cite{mourrat_construction_2017}.
\end{proof}

In the third term of \eqref{eq:3d w i2}, we only need to use \Cref{thm:besov lp product} and Young to estimate
\begin{equation}
\begin{split}
&\mathrel{\phantom{\lesssim}} \lambda^{1/2} \E \int_{\Tor^3} \W_N I_N[w]^2 \dx\\
&\lesssim \lambda^{1/2} \E \norm{\W_N}_{\besovinfty^{-1/2-\epsilon}}
    \smallnorm{I_N[w]^2}_{B_{1,1}^{1/2+2\epsilon}}\\
&\lesssim \lambda^{1/2} \E \norm{\W_N}_{\besovinfty^{-1/2-\epsilon}}
    \norm{I_N[w]}_{L^2} \norm{I_N[w]}_{H^{1-\epsilon}}\\
&\leq \lambda C_\rho \E \norm{\W_N}_{\besovinfty^{-1/2-\epsilon}}^4
    + \lambda C_\rho \norm{I_N[w]}_{L^4}^4 + \rho \norm{I_N[w]}_{H^{1-\epsilon}}^2.
\end{split}
\end{equation}

\bigskip\emph{The $\W_N^2 I_N[v]$ term.}
Here we substitute the full ansatz to get
\begin{equation}
\frac 1 2 \E \int_{\Tor^3} \left[ -\lambda^{3/2} \W_N^2 \W_N^{[3]}
    - \lambda^{3/2} \W_N^2 I_N \left(J_t (\W_t^2 \parar \Theta_t I_N[v])\right)
    + \lambda^{1/2} \W_N^2 I_N[h] \right] \dx.
\end{equation}
The first term vanishes by an application of Wick's theorem since
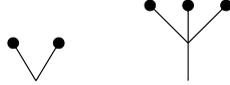
\begin{center}
\begin{tikzpicture}
    \begin{scope}
        \coordinate (B1) at (-0.3, 0.5);
        \coordinate (B2) at (0.3, 0.5);
        \filldraw (0,0) -- (B1) circle[whitenoise];
        \filldraw (0,0) -- (B2) circle[whitenoise];
    \end{scope}
    \begin{scope}[xshift=2cm, yshift=0.5cm]
        \coordinate (C1) at (-0.5, 0.5);
        \coordinate (C2) at (0, 0.5);
        \coordinate (C3) at (0.5, 0.5);
        \filldraw (0,0) -- (C1) circle[whitenoise];
        \filldraw (0,0) -- (C2) circle[whitenoise];
        \filldraw (0,0) -- (C3) circle[whitenoise];
        \draw (0,-0.5) -- (0,0);
    \end{scope}
\end{tikzpicture}
\end{center}
cannot be contracted.
The remaining two terms correspond to the two renormalization terms still left.

We use the definition of $I_T$ to write the middle term as
\begin{equation}\label{eq:3d w2 iv before commutation}
-\frac{\lambda^{3/2}}{2} \E \int_0^N \int_{\Tor^3}
 (J_t \W_t^2) J_t(\W_t^2 \parar \Theta_t I_N[v]) \dx \dt.
\end{equation}
We can move $J_t$ inside the paraproduct by \cite[Proposition A.10]{barashkov_variational_2020},
and the error is bounded by
\begin{equation}
\begin{split}
&\mathrel{\phantom{\lesssim}}
   \norm{ J_t (\W_t^2 \parar \Theta_t I_N(v))- ((J_t \W_t^2) \parar \Theta_t I_N(v))}_{H^{1-2\delta}}\\
&\lesssim \frac{1}{t^{1/2+\delta}}\smallnorm{\W^2_t}_{\mathcal{C}^{-1-\delta}} \norm{I_N(v)}_{H^1},
\end{split}
\end{equation}
which implies 
\begin{align}
&\mathrel{\phantom{\lesssim}} \abs{\frac{\lambda^{3/2}}{2} \E \int_0^N \int_{\Tor^3}
    (J_t \W_t^2) (J_t(\W_t^2 \parar \Theta_t I_N[v])-(J_t\W_t^2 \parar \Theta_t I_N[v]) \dx \dt} \notag\\
&\lesssim \int_0^N \frac{1}{t^{1+\delta}}\smallnorm{\W^2_t}^2_{\mathcal{C}^{-1-\delta}} \norm{I_N(v)}_{H^1} \dt.
\end{align}
Now it remains to study the term 
\begin{equation}
\frac{\lambda^{3/2}}{2} \E \int_0^N \int_{\Tor^3} (J_t \W_t^2)(J_t\W_t^2 \parar \Theta_t I_N[v]).
\end{equation}
This is set up the following commutator estimate that states that
paraproduct and resonant product are ``almost adjoint'' to each other.

\begin{lemma}\label{thm:paraproduct almost adjoint}
Assume that $\alpha + \gamma < 0$ and $\alpha + \beta + \gamma = 0$,
and let $p_1, \ldots, p_3$ and $r_1, r_2$ be Hölder conjugates.
Then
\[
\int_{\Tor^d} (f \parar g) h \dx
= \int_{\Tor^d} (f \reson h) g \dx + R(f, g, h),
\]
where the residual term satisfies
\[
\abs{R(f, g, h)} \lesssim
    \norm{f}_{B^\alpha_{p_1, r_1}} \norm{g}_{B^\beta_{p_2, \infty}} \norm{h}_{B^\gamma_{p_3, r_2}}.
\]
\end{lemma}
\begin{proof}
\cite[Lemma~A.6]{gubinelli_semilinear_2020},
with minor modification to use arbitrary $p_i$.
\end{proof}

So now \eqref{eq:3d w2 iv before commutation} can be estimated with
\begin{equation}\label{eq:3d w2 iv after commutation}
-\frac{\lambda^{3/2}}{2} \E \int_0^N \!\! \int_{\Tor^3}
    (J_t \W_t^2 \reson J_t \W_t^2) (\Theta_t I_N[v]) \dx \dt,
\end{equation}
where the estimation error is bounded by
\begin{equation}
\lambda^{3/2} \E \int_0^N 
    \smallnorm{J_t \W_t^2}^2_{\besovinfty^{-\epsilon}}
    \norm{\Theta_t I_N[v]}_{H^{3\epsilon}} \frac{\dt}{\inorm t^{1+\epsilon}},
\end{equation}
which can be estimated as usual by Young.
Now we renormalize \eqref{eq:3d w2 iv after commutation} with the $\lambda^{3/2} \gamma_N I_N[v]$ term.
We introduce $\Theta_N$ in there and differentiate in time:
\begin{equation}\label{eq:3d cubic gamma theta diff}
\lambda^{3/2} \gamma_N \Theta_N I_N[v]
= \lambda^{3/2} \int_0^N  (\partial_t \gamma_t) \Theta_t I_N[v]
    + \gamma_t (\partial_t \Theta_t) I_N[v] \dt.
\end{equation}
We use the $\lambda^{3/2} (\partial_t \gamma_t) \Theta_t I_N[v]$ part
to renormalize \eqref{eq:3d w2 iv after commutation}.
By duality we are to estimate
\begin{equation}
\frac{\lambda^{3/2}}{2} \E \int_0^N
    \smallnorm{J_t \W_t^2 \reson J_t \W_t^2 + 2 \partial_t \gamma_t}_{H^{-\epsilon}}
    \smallnorm{\Theta_t I_N[v]}_{H^\epsilon} \dt.
\end{equation}
Recall that $\Theta_t$ is bounded uniformly in $t$,
so we can move $\norm{I_N[v]}_{H^\epsilon}$ outside the integral.
Therefore the expression is bounded by Young and the following lemma:

\begin{lemma}\label{thm:3d partition jw2}
For any $\epsilon > 0$ and $p, q, r \in (1, \infty)$ we have
\[
\E \left[ \int_0^\infty \norm{J_t \W_t^2 \reson J_t \W_t^2 + 2\partial_t \gamma_t}_{B^{-\epsilon}_{p,r}} \dt \right]^q
\lesssim_{\epsilon, p, q, r} 1.
\]
This estimate is uniform in the implicit truncation $N$.
\end{lemma}
\begin{proof}
\cite[Lemma~25]{barashkov_variational_2020}.
There is a sign error in the published version of the cited article;
the correct sign is used in e.g.\ Lemma~24 there.
\end{proof}

There are still some remaining pieces of the renormalization term.
However, they are small:

\begin{lemma}
We can estimate the second part of \eqref{eq:3d cubic gamma theta diff} by
\begin{equation}
\lambda^{3/2} \E \abs{ \int_0^N \!\! \int_{\Tor^3} \gamma_t (\partial_t \Theta_t) I_N[v] \dx \dt }
\lesssim \lambda^{3/2} \E \norm{I_N[v]}_{L^4}^4,
\end{equation}
and when $N$ is large enough,
\begin{equation}
\lambda^{3/2} \E \int_{\Tor^3} \gamma_N (1 - \Theta_t) I_N[v] \dx = 0.
\end{equation}
\end{lemma}
\begin{proof}
It suffices to note that $\partial_t \Theta_t$ and $(1-\Theta_t)$ are supported
away from the origin when $t$ is large enough.
This means that the spatial integrals vanish for some $t \geq t_0$.

It is shown in \cite[Lemma~25]{barashkov_variational_2020} that
$\abs{\gamma_t} \lesssim 1 + \log\inorm t$.
Finally, $\partial_t \Theta_t$ is an uniformly bounded operator by construction.
These facts let us bound the first integral by
\[
C \lambda^{3/2} t_0 (1 + \log\langle t_0 \rangle) \int_{\Tor^3} I_N[v] \dx
\lesssim \lambda^{3/2} (1 + \norm{I_N[v]}_{L^4}^4).
\qedhere
\]
\end{proof}

\bigskip%
This leaves us with the $\W_N^2 I_N[h]$ term.
Here we repeat the argument of \eqref{eq:3d ansatz 1}:
first use the Itô formula and self-adjointness to write
\begin{equation}\label{eq:3d before ansatz 3}
\frac{\lambda^{1/2}}{2} \E \int_{\Tor^3} \W_N^2 I_N[h] \dx
= \frac{\lambda^{1/2}}{2} \E \int_0^N \!\! \int_{\Tor^3} (J_t \W_t^2) h_t \dx \dt,
\end{equation}
and then update the ansatz as
\begin{equation}\label{eq:3d ansatz 3}
h_t = -\frac{\lambda^{1/2}}{2} J_t \W_t^2 + g_t.
\end{equation}
As we repeat the earlier argument, we find that
\begin{align}
&\frac{\lambda^{1/2}}{2} \E \int_{\Tor^3} \W_N^2 I_N[h] \dx
    + \frac 1 2 \int_0^N \norm{h_t}_2^2 \dt \notag\\
=\; &-\frac \lambda 8 \E \int_0^N \!\! \int_{\Tor^3} (J_t \W_t^2)^2 \dx \dt
    + \frac 1 2 \int_0^N \norm{g_t}_2^2 \dt.
\end{align}
We can replace $(J_t \W_t^2)^2$ by the resonant product $J_t \W_t^2 \reson J_t \W_t^2$,
since the paraproducts do not contribute to the zero mode.
We then use the $-\lambda \gamma_N / 4$ renormalization term.
By the bounded domain, we can use duality (\Cref{thm:besov multiplication}) to estimate
\begin{align*}
&\frac \lambda 8 \E \int_0^N \!\!
    \int_{\Tor^3} \abs{ J_t \W_t^2 \reson J_t \W_t^2 + 2\partial_t \gamma_t } \dx \dt \notag\\
\lesssim\; &\lambda \norm{1}_{H^{\epsilon}}^2
    + \lambda \E\left[ \int_0^N \norm{J_t \W_t^2 \reson J_t \W_t^2 + 2\partial_t \gamma_t}_{H^{-\epsilon}}
            \dt \right]^2.
\end{align*}
This is again bounded by \Cref{thm:3d partition jw2}, which finishes the proof of this term.
It only remains to verify that previous estimates are still valid with the updated ansatz.
This corresponds to \cite[Eq.~(52)]{barashkov_variational_2020}.

\begin{lemma}\label{thm:3d partition w bound}
We have
\[
\E \norm{I_N[w]}_{H^{1-\epsilon}}^2
\lesssim \lambda + \lambda \norm{I_N[v]}_{L^4}^4 + \int_0^N \norm{g_t}_2^2 \dt.
\]
\end{lemma}
\begin{proof}
The left-hand side is expanded as
\begin{equation}
\norm{\lambda I_N [J_t (\W_t^2 \parar \Theta_t I_N[v])]
    + \frac{\sqrt\lambda}{2} I_N[J_t \W_t^2] - I_N[h]}_{H^{1-\epsilon}}^2.
\end{equation}
By Young's inequality this is bounded by
\begin{equation}
C\norm{\lambda I_N [J_t (\W_t^2 \parar \Theta_t I_N[v])]}_{H^{1-\epsilon}}^2
+ C\norm{\sqrt\lambda I_N [J_t \W_t^2]}_{H^{1-\epsilon}}^2
+ C\norm{I_N[h]}_{H^{1-\epsilon}}^2.
\end{equation}
Note that the prefactor does not matter since this lemma will always be used with
$\rho \norm{I_N[w]}_{H^{1-\epsilon}}^2$ terms where $\rho$ is small.
We can then apply \Cref{thm:2d drift h1} to each term.
The first term then equals
\begin{equation}
2\lambda^2 \int_0^N \norm{J_t (\W_t^2 \parar \Theta_t I_N[v])}_{H^{-\epsilon}}^2 \dt
\lesssim \lambda^2 \int_0^N \frac{1}{\inorm t^{1+\epsilon}}
    \norm{\W_t^2 \parar \Theta_t I_N[v]}_{H^{-1-\epsilon/2}}^2 \dt
\end{equation}
by \Cref{thm:boue-dupuis j regularity}.
Now the paraproduct bound (\Cref{thm:besov multiplication})
and $L^p$ embedding (\Cref{thm:besov embeddings}) give that this is bounded by
\begin{equation}
\lambda^2 \norm{I_N[v]}_{L^4}^2
    \int_0^N \frac{1}{\inorm t^{1+2\epsilon}} \norm{\W_t^2}_{B^{-1-\epsilon/2}_{4,2}}^2,
\end{equation}
and it remains to apply Young's inequality
and the probabilistic bound \cite[Lemma~4]{barashkov_variational_2020}.
We then apply the same argument to the second term.
\end{proof}

\bigskip\emph{The $I_N[v]^3$ term.}
The estimation of this term parallels the two-dimensional case,
although the ansatz complicates matters slightly.
We will again split $I_N[v]$ into its mean $I_N^\circ[v]$ and oscillatory part $I_N^\perp[v]$.

The argument for $I_N[v]^2 I_N^\circ[v]$ is exactly as in two dimensions;
see \eqref{eq:2d kappa fixed}.

Let us then substitute the final ansatz \eqref{eq:3d ansatz 3} into $I_N^\perp[v]$ to get
\begin{align}
&2\lambda^{1/2} \E \int_{\Tor^3} I_N[v]^2 I_N^\perp[v] \dx \notag\\
=\; &2\E \int_{\Tor^3} \bigg[ -\lambda^{3/2} I_N[v]^2 I_N^\perp [J_t \W_t^3]
    - \lambda^{3/2} I_N[v]^2 I_N^\perp[J_t (\W_t^2 \parar \Theta_t I_N[v])] \notag\\
    &\quad - \lambda I_N[v]^2 I_N^\perp[J_t \W_t^2]
    + \lambda^{1/2} I_N[v]^2 I_N^\perp[g] \bigg] \dx.
\end{align}
All but the last term can be estimated with Young as before;
recall that projection operators are bounded so we may replace $I_N^\perp$ by $I_N$.
Thanks to the factors of $\lambda$ and $\lambda^{3/2}$, the bound is of form
\begin{equation}
\lambda\rho \norm{I_N[v]}_{L^4}^4
+ \lambda C_\rho \E\left[ \smallnorm{\W_N^{[3]}}_{L^2}^2
    + \norm{I_N J_t(\W_t^2 \parar \Theta_t I_N[v])}_{L^2}^2
    + \smallnorm{\W_t^{[2]}}_{L^2}^2 \right].
\end{equation}
The final term is then estimated with Poincaré's inequality
exactly as in \eqref{eq:2d partition using poincare}.
Note that $2\sqrt\lambda = 1/\sqrt\beta$ in the notation of \Cref{sec:2d partition}.

\subsubsection*{Conclusion}

All together, we have now showed that
\begin{equation}
\eqref{eq:3d variational everything}
\geq -\lambda C_\rho
    + (1-\rho) \E \left[ \lambda \int_{\Tor^3} I_N[v]^4 \dx + \frac 1 2 \int_0^N \norm{g_t}_2^2 \dt \right].
\label{eq:3d upper conclusion} 
\end{equation}
Since the bracketed term is non-negative and we can choose $\rho < 1$,
this shows that $-\lambda C_\rho = -C_\rho / 4\beta$ is a lower bound for the variational problem.
As in \Cref{thm:2d variational upper}, the result also holds with a small $q > 1$ factor.
Consequently,
\begin{equation}
\int_{\goodregion_+} \exp(-q \kappa \negpart{\hat\phi(0) + (1-\delta') \sqrt\beta}^2)
    \exp(-q V_\beta (\phi)) \diff\Gaussian(\phi)
\leq \exp\left(\frac{C_\rho}{4\beta}\right),
\end{equation}
and this finishes the upper bound of \Cref{thm:3d partition overall}.

\subsection{Lower bound for partition function}\label{sec:3d partition lower}

As a counterpart to \eqref{eq:3d upper conclusion}, an upper bound for the variational problem is
\begin{equation}
\eqref{eq:3d variational everything}
\leq \lambda C_\rho
    + (1+\rho) \E \left[ \lambda \int_{\Tor^3} I_N[v]^4 \dx + \frac 1 2 \int_0^N \norm{g_t}_2^2 \dt \right].
\end{equation}
We need to choose a drift $v$ such that this expression is finite and of order $\lambda$.
The following argument is based on \cite[Section~5]{barashkov_variational_2020},
but somewhat simplified since our bounds may freely depend on domain size.
Another minor modification is due to the $J_t \W^2_t$ term coming from \eqref{eq:3d ansatz 3}.

As part of the previous section, we have decomposed the drift as
\begin{equation}\label{eq:3d partition ansatz final}
v_t = -\lambda J_t \W_t^3 - \lambda J_t (\W_t^2 \parar \Theta_t I_N[v])
    - \frac{\lambda^{1/2}}{2} J_t \W_t^2 + g_t.
\end{equation}
We will first introduce one more regularizing operator.
Its usefulness is due to the following two lemmas
that let us control the ansatz.

\begin{definition}
Let $\chi$ be a smoothed indicator of $B(0, 1) \subset \R$.
We define $\mathcal X_t$ to be a Fourier multiplier with symbol
\[
\xi \mapsto \chi\!\left( \frac{\abs\xi}{(1 + \norm{\W_t^2}_{\besovinfty^{-1-\epsilon}})^{1/\epsilon}} \right).
\]
That is, $\mathcal X_t$ restricts frequencies to a ball
whose radius rapidly increases with the norm of $\W_t^2$.
\end{definition}

\begin{lemma}\label{thm:3d lower x high}
We have
\[
\norm{(1-\mathcal X_t) \W_t^2}_{\besovinfty^{-1-3\epsilon}} \lesssim 1
\]
uniformly in $t$.
\end{lemma}
\begin{proof}
Let us recall that if $f$ is Fourier-supported outside $B(0, R)$ and $\delta > 0$,
then Bernstein's inequality states that
\begin{equation}
\norm{f}_{B^s_{p,r}} \lesssim R^{-\delta} \norm{f}_{B^{s+\delta}_{p,r}}.
\end{equation}
We also use the uniform boundedness of projection operators (on $L^p$, $1 \leq p < \infty$)
together with Besov embedding to see that $1 - \mathcal X_t$ is bounded
from $\besovinfty^{-1-2\epsilon}$ to $\besovinfty^{-1-\epsilon}$.
This gives us the bound
\begin{equation}
\norm{(1-\mathcal X_t) \W_t^2}_{\besovinfty^{-1-3\epsilon}}
\lesssim \frac{\smallnorm{\W_t^2}_{\besovinfty^{-1-\epsilon}}}
    {1 + \smallnorm{\W_t^2}_{\besovinfty^{-1-\epsilon}}}.
\end{equation}
Thus we are left with a quotient bounded by $1$.
\end{proof}

\begin{lemma}\label{thm:3d lower x low}
We have
\[
\norm{\mathcal X_t \W_t^2}_{\besovinfty^{-1+\epsilon}}
\lesssim 1 + \norm{\W_t^2}_{\besovinfty^{-1-\epsilon}}^2.
\]
\end{lemma}
\begin{proof}
A converse version of Bernstein's inequality states that
if $f$ is Fourier-supported inside $B(0, R)$, then
\begin{equation}
\norm{f}_{B^{s+\delta}_{p,r}} \lesssim R^{\delta} \norm{f}_{B^s_{p,r}}.
\end{equation}
In our case this yields
\begin{equation}
\norm{\mathcal X_t \W_t^2}_{\besovinfty^{-1+\epsilon}}
\lesssim (1 + \norm{\W_t^2}_{\besovinfty^{-1-\epsilon}})
    \norm{\W_t^2}_{\besovinfty^{-1-\epsilon}}.
\end{equation}
We again used uniform boundedness of $\mathcal X_t$.
\end{proof}

Let us then define $\check v$ as the solution to the equation
\begin{equation}
\check v_t = -\lambda J_t \W_t^3
    - \frac{\lambda^{1/2}}{2} J_t \W_t^2
    - \lambda J_t ((1-\mathcal X_t) \W_t^2 \parar \Theta_t I_t[\check v]).
\end{equation}
For all $t \geq 0$, $1 \leq p \leq \infty$, and almost all realizations of the random field,
a solution exists in $B^{-1/2-3\epsilon}_{p,p}$ by Banach's fixed-point theorem,
at least assuming $\lambda$ small enough.
The proof is a modification of the following lemma,
which gives an $L^4$ estimate:

\begin{lemma}
The $L^4$ norm of this drift satisfies uniformly in $N$ the estimate
\[
\E \norm{I_N[\check v]}_{L^4}^4 \lesssim \lambda^2.
\]
\end{lemma}
\begin{proof}
By Minkowski's integral inequality, we may write
\begin{multline}
\norm{I_N[\check v]}_{B^{1/2-2\epsilon}_{p,p}}
\leq \int_0^N \norm{\lambda J_t^2 \W_t^3 + \frac{\lambda^{1/2}}{2} J_t^2 \W_t^2
        }_{B^{1/2-2\epsilon}_{p,p}} \dt\\
    + \lambda \int_0^N \norm{J_t^2 ((1-\mathcal X_t) \W_t^2 \parar \Theta_t I_t[\check v])
        }_{B^{1/2-2\epsilon}_{p,p}} \dt.
\end{multline}
By \Cref{thm:boue-dupuis j regularity} the right-hand side is bounded by
\begin{multline}
\lambda^{1/2} \int_0^N \norm{J_t (\W_t^3 + \W_t^2)
    }_{B^{-1/2-\epsilon}_{p,p}} \frac{\dt}{\inorm t^{1/2+\epsilon}}\\
+ \lambda \int_0^N \norm{(1-\mathcal X_t) \W_t^2 \parar \Theta_t I_t[\check v]
    }_{B^{-3/2-\epsilon}_{p,p}} \frac{\dt}{\inorm t^{1+\epsilon}}.
\end{multline}
The second term is bounded by
\begin{equation}
\lambda \int_0^N \norm{I_t[\check v]}_{B^{1/2-2\epsilon}_{p,p}} \frac{\dt}{\inorm t^{1+\epsilon}}
\end{equation}
by \Cref{thm:3d lower x high},
and then Grönwall's inequality gives the bound
\begin{equation}\label{eq:3d lower it bound}
\norm{I_N[\check v]}_{B^{1/2-2\epsilon}_{p,p}}
\lesssim \lambda^{1/2} \int_0^N \norm{J_t (\W_t^3 + \W_t^2)
    }_{B^{-1/2-\epsilon}_{p,p}} \frac{\dt}{\inorm t^{1/2+\epsilon}}.
\end{equation}
The fourth power of this expression has finite expectation due to
the end of \cite[Lemma~6]{barashkov_variational_2020},
and the result then follows from $L^p$ embedding (\Cref{thm:besov embeddings}).
\end{proof}

As we plug $\check v$ into \eqref{eq:3d partition ansatz final},
the remainder is
\begin{equation}
\check g_t = \lambda J_t (\mathcal X_t \W_t^2 \parar \Theta_t I_N[\check v]).
\end{equation}
Its $L^2$ norm also satisfies a suitable bound:

\begin{lemma}
We have the uniform-in-$N$ bound
\[
\E \int_0^N \norm{\check g_t}_2^2 \dt \lesssim \lambda^3.
\]
\end{lemma}
\begin{proof}
Let us first use \Cref{thm:boue-dupuis j regularity} to estimate
\begin{align}
&\lambda^2 \E \int_0^N
    \norm{J_t (\mathcal X_t \W_t^2 \parar \Theta_t I_N[\check v])}_2^2 \dt \notag\\
\lesssim\; &\lambda^2 \E \int_0^N
    \norm{\mathcal X_t \W_t^2 \parar \Theta_t I_N[\check v]}_{H^{-1+\epsilon}}^2
    \frac{\dt}{\inorm t^{1+2\epsilon}}.
\end{align}
By a paraproduct estimate, this is bounded by
\begin{equation}
\lambda^2 \E \int_0^N
    \norm{\mathcal X_t \W_t^2}_{\besovinfty^{-1+\epsilon}}^2
    \norm{I_N[\check v]}_{L^2}^2
    \frac{\dt}{\inorm t^{1+2\epsilon}},
\end{equation}
and by \Cref{thm:3d lower x low} and Young's inequality we can estimate it by
\begin{equation}
\lambda^2 \E \int_0^N
    \left[ \lambda (1 + \norm{\mathcal X_t \W_t^2}_{\besovinfty^{-1-\epsilon}}^2)^4
    + \lambda\inv \norm{I_N[\check v]}_{L^2}^4 \right]
    \frac{\dt}{\inorm t^{1+2\epsilon}}.
\end{equation}
Now $\smallnorm{\mathcal X_t \W_t^2}_{\besovinfty^{-1-\epsilon}}^8$
has finite expectation by \cite[Lemma~6]{barashkov_variational_2020},
whereas $\E \smallnorm{I_N[\check v]}_{L^2}^4$ is of order $\lambda^2$ by \eqref{eq:3d lower it bound}.
\end{proof}

These estimates show that our choice of drift $\check v_t$, $\check g_t$ in \eqref{eq:3d partition ansatz final}
satisfies the necessary upper bound for the infimum.
This proves the lower bound of \Cref{thm:3d partition overall}.

\subsection{Change of Wick ordering}\label{sec:3d wick}

For the computations in \Cref{sec:3d partition upper},
we needed to change the Wick ordering to match the $(2-\Laplace)\inv$ covariance
of the translated field.
The change of mass also changes the divergent renormalization terms.
We will need to verify that this difference is of order $\beta\inv$.

In the following, we emphasize the different masses with subscripts like $\gamma_+$ and $\gamma_-$.
On the other hand, we do not show the dependence on $N$ where it is not relevant.

Let us consider the potential in \Cref{thm:3d partition function form}.
We change the Wick ordering with \Cref{thm:wick change}
and write $\gamma_-$ as $\gamma_+ + (\gamma_- - \gamma_+)$.
The quartic term
\begin{equation}
\lambda \wickm{\phi^4} - \lambda^2 \gamma_- \wickm{\phi^2} - \delta_-
\end{equation}
then becomes
\begin{equation}\label{eq:3d wick quartic}
\begin{split}
&\bigg[ \lambda \wickp{\phi^4} - \lambda^2 \gamma_+ \wickp{\phi^2} - \delta_+ \bigg]\\
{}+{} &3\lambda (C_- - C_+)^2
    - 6\lambda (C_- - C_+) \wickp{\phi^2} - \lambda^2 (\gamma_- - \gamma_+) \wickp{\phi^2}\\
{}-{} &(\delta_- - \delta_+) + \lambda^2 \gamma_- (C_- - C_+).
\end{split}
\end{equation}
The term in brackets was studied in the preceding sections.
Correspondingly, the renormalized cubic term
\begin{equation}
2 \sqrt\lambda \wickm{\phi^3} - \lambda^{3/2} \gamma_- \phi - \frac \lambda 4 \gamma_-
\end{equation}
becomes
\begin{equation}\label{eq:3d wick cubic}
\begin{split}
&\bigg[ 2 \sqrt\lambda \wickp{\phi^3} - \lambda^{3/2} \gamma_+ \phi - \frac \lambda 4 \gamma_+ \bigg] \\
{}-{} &6\sqrt\lambda (C_- - C_+) \phi
- \lambda^{3/2} (\gamma_- - \gamma_+) \phi
- \frac \lambda 4 (\gamma_- - \gamma_+),
\end{split}
\end{equation}
where the term in brackets was previously analyzed.

The differences in \eqref{eq:3d wick cubic} and the second line of \eqref{eq:3d wick quartic}
are all uniformly bounded,
and therefore the lines contribute a factor of order $\lambda$.
This is formalized in the two lemmas below.

\begin{lemma}\label{thm:3d wick const difference}
$C_{N,+} - C_{N,-}$ is bounded uniformly in $N$.
\end{lemma}
\begin{proof}
It suffices to note that \eqref{eq:wick difference} is finite also as a three-dimensional sum.
\end{proof}

\begin{lemma}\label{thm:3d wick gamma difference}
$\gamma_{N,+} - \gamma_{N,-}$ is bounded uniformly in $N$.
\end{lemma}
\begin{proof}
Let us recall that $\gamma_N$ was defined as
\begin{equation}
\gamma_N = -\frac 1 2 \E \int_{\Tor^3} \int_0^N (J_t \W_t^2)^2 \dt \dx.
\end{equation}
From \eqref{eq:chaos w2}, the chaos decomposition of the integrand is
\begin{equation}
\widehat{J_t \W^2_t}(k) = 24 \sum_{n_1 + n_2 = k}
\int_0^N \! \int_0^{t_1} \frac{\sigma_{t}(k) \sigma_{t_1}(n_1) \sigma_{t_2}(n_2)}
        {\inormc{k} \inormc{n_1} \inormc{n_2}}
    \diff B^{n_2}_{t_2} \diff B^{n_1}_{t_1},
\end{equation}
where the sum is over $\smallabs{n_1}, \smallabs{n_2} \leq N$.
Itô isometry implies
\begin{equation}
\E \widehat{J_t \W^2_t}(k)^2 = 24^2 \sum_{n_1 + n_2 = k}
\int_0^N \! \int_0^{t_1} \frac{\sigma_{t}(k)^2 \sigma_{t_1}(n_1)^2 \sigma_{t_2}(n_2)^2}
        {\inormc{k}^2 \inormc{n_1}^2 \inormc{n_2}^2}
    \dt_2 \dt_1.
\end{equation}
As all the summands are positive,
we may now let $n_1$ and $n_2$ range over all $\Z^3$;
this is necessary for uniformity in $N$.
In the negative-mass case, we still assume $n_1$ and $n_2$ to be non-zero.
We can then integrate over $t$ and use Plancherel's identity to write
\begin{equation}\label{eq:3d wick w2 chaos final}
\E \int_0^N \!\! \int_{\Tor^3} (J_t \W_t^2)^2 \dx \dt
= 24^2 \sum_{n_1, n_2}
\int_0^N \!\!\! \int_0^N \!\! \int_0^{t_1} \!\! 
\frac{\sigma_{t}(n_{12})^2 \sigma_{t_1}(n_1)^2 \sigma_{t_2}(n_2)^2}
    {\inormc{n_{12}}^2 \inormc{n_1}^2 \inormc{n_2}^2}
    \dt_2 \dt_1 \dt.
\end{equation}
Here we abbreviated $n_{12} = n_1 + n_2$.

Let us then compare the expressions for $\gamma_+$ and $\gamma_-$.
We first note that if $n_1 = 0$ (or vice versa $n_2 = 0$),
then the summands in $\gamma_+$ are $\inormp{n_2}^{-4}$.
Therefore we only need to consider the contribution of $n_1, n_2 \neq 0$.
Moreover, the integrals are independent of mass and bounded by $1$,
so we may disregard them.

We are left with a telescoping sum where we change one mass at a time:
\begin{multline}
\gamma_+ - \gamma_- \simeq \bigg[
    \sum_{n_1, n_2 \neq 0}
        \frac{1}{\inormp{n_{12}}^2 \inormp{n_1}^2}
        \left( \frac{1}{\inormp{n_2}^2} - \frac{1}{\inormm{n_2}^2} \right)\\
    + \sum_{n_1, n_2 \neq 0}
        \frac{1}{\inormp{n_{12}}^2 \inormm{n_2}^2}
        \left( \frac{1}{\inormp{n_1}^2} - \frac{1}{\inormm{n_1}^2} \right)\\
    + \sum_{n_1, n_2 \neq 0}
        \frac{1}{\inormm{n_1}^2 \inormm{n_2}^2}
        \left( \frac{1}{\inormp{n_{12}}^2} - \frac{1}{\inormm{n_{12}}^2} \right)
\bigg].
\end{multline}
The first term can be written as
\begin{equation}
\sum_{n_1, n_2 \neq 0}
    \frac{-3}{\inormp{n_{12}}^2 \inormp{n_1}^2 \inormp{n_2}^2 \inormm{n_2}^2}.
\end{equation}
Now it remains to note that $\inormp{n_2}^2 \simeq \inormm{n_2}^2$, and that
\begin{equation}
\sum_{n_2 \neq 0} \frac{3}{\inormp{n_2}^4}
    \sum_{\substack{n_1 \neq 0, n_{12} \in \Z^3\\ n_1 + n_{12} = n_2}}
        \frac{1}{\inormp{n_{12}}^2 \inormp{n_1}^2}
\lesssim \sum_{n_2 \neq 0} \frac{3}{\inormp{n_2}^5}
< \infty
\end{equation}
by \Cref{thm:discrete convolution}.
The bounds for the other two summands are identical.
\end{proof}

With these two estimates, it follows that the prefactors of $\phi$ and $\wickp{\phi^2}$
in \eqref{eq:3d wick quartic} and \eqref{eq:3d wick cubic} are bounded.
In the variational formulation we can estimate these terms as follows:

\begin{lemma}
We have for any $\rho > 0$ the estimates
\begin{align*}
\sqrt\lambda \abs{\E \int_{\Tor^3} (\W + I_N[v]) \dx}
&\leq \lambda C_\rho + \rho \norm{I_N[v]}_{H^{1-\epsilon}}^2,\\
\lambda \abs{\E \int_{\Tor^3} \wickp{(\W + I_N[v])^2} \dx}
&\leq \lambda C_\rho + \lambda\rho\norm{I_N[v]}_4^4 + \lambda \norm{I_N[w]}_{H^{1-\epsilon}}^2.
\end{align*}
\end{lemma}
\begin{proof}
We only show the details for the second estimate.
As before, $\E \W^2 = 0$.
We then need to control $2 \W I_N[v] + I_N[v]^2$
with the terms from \Cref{rem:3d partition control}.
As we substitute the initial ansatz \eqref{eq:3d ansatz 1}, we get
\begin{equation}
-\lambda \E \int_{\Tor^3} 2\lambda \W_N \W^{[3]}_N - 2\W_N I_N[w]
    + (\lambda \W_N^{[3]} + I_N[w])^2 \dx.
\end{equation}
The regularity of the first product is technically not good enough,
but since paraproducts do not contribute to the zero mode,
we only need to consider the resonant product.
We can then use duality and \Cref{thm:3d w1w3 estimate} to estimate
\begin{equation}
\lambda^2 \E \smallnorm{\W_N \reson \W^{[3]}_N}_1
\lesssim \lambda^2 \norm{1}_{B_{1,1}^{\epsilon}}
    \E \smallnorm{\W_N \reson \W^{[3]}_N}_{\besovinfty^{-\epsilon}}.
\end{equation}
This stochastic term has finite expectation by \Cref{thm:3d w1w3 estimate}.
Duality also gives
\begin{equation}
\lambda \E \norm{\W_N I_N[w]}_1
\leq \lambda \E \left[ C_\rho \norm{\W_N}_{H^{-1/2-\epsilon}}^2
    + \rho \norm{I_N[w]}_{H^{1-\epsilon}}^2 \right].
\end{equation}
The square term is then estimated by
\begin{equation}
\lambda \E \bigg[ \lambda^2 \smallnorm{\W^{[3]}_N}_{L^2}^2
+ \lambda (\smallnorm{\W^{[3]}_N}_{L^2}^2 + \smallnorm{I_N[w]}_{L^2}^2)
+ \norm{I_N[w]}_{L^2}^2
\bigg].
\end{equation}
All the $L^2$ norms are bounded by the relevant Besov norms.
The factor $\lambda$ ensures that the prefactor of $I_N[w]$ is small.
\end{proof}

This leaves us with the last line of \eqref{eq:3d wick quartic}.
Both terms on this line diverge logarithmically in $N$,
but it turns out that the divergences cancel each other.
The $\lambda^2 \gamma_+ (C_+ - C_-)$ counterterm cancels the leading term of $\delta_{N,+} - \delta_{N,-}$,
and the other terms of $\delta_{N,+} - \delta_{N,-}$ vanish rapidly enough as $\beta \to \infty$.
We formalize this in the following series of lemmas.

\begin{lemma}
We can estimate
\[
\abs{-\frac{\lambda^2}{2} \int_{\Tor^3} \int_0^N (J_{t,+} \W_t^3)^2 - (J_{t,-} \W_t^3)^2 \dt \dx
- \lambda^2 \gamma_{N,+} (C_{N,+} - C_{N,-})}
\lesssim \lambda^2
\]
uniformly in $N$.
\end{lemma}
\begin{proof}
For brevity, we will denote the leading $(J_t \W^3_t)^2$ term of $\delta$ by $\delta'$ within this proof.

Let us first write the chaos decomposition of $\gamma_+ (C_+ - C_-)$.
Multiplying \eqref{eq:3d wick w2 chaos final} by $(C_{N,+} - C_{N,-})$
from \Cref{thm:3d wick const difference},
the sum coefficients become
\begin{equation}\label{eq:3d wick main lemma gamma sum}
-\frac{24^2 \lambda^2}{2} \sum_{n_1, n_2, n_3}
\frac{\inormm{n_3}^2 - \I_{n_3 \neq 0} \inormp{n_3}^2}
    {\inormp{n_{12}}^2 \inormp{n_1}^2 \inormp{n_2}^2 \inormp{n_3}^2 \inormm{n_3}^2},
\end{equation}
and the frequency-dependent integral part is
\begin{equation}\label{eq:3d wick main lemma gamma integrals}
\int_0^N \!\!\! \int_0^t \!\! \int_0^{t_1} \!\!\! \int_0^N
    \sigma_{t}(n_{12})^2 \sigma_{t_1}(n_1)^2 \sigma_{t_2}(n_2)^2 \sigma_{t_3}(n_3)^2
    \dt_3 \dt_2 \dt_1 \dt.
\end{equation}
Let us then decompose $\delta'_N$.
By \eqref{eq:chaos w3}, we can write
\begin{equation}
\widehat{J_t \W^3_t}(k) = 24 \sum_{\substack{n_1 + n_2\\ + n_3 = k}}
    \int_0^t \! \int_0^{t_1} \!\! \int_0^{t_2}
    \frac{\sigma_{t}(k)}{\inormc{k}}
    \prod_{i=1}^3 \frac{\sigma_{t_i}(n_i)}{\inormc{n_i}}
        \diff B^{n_3}_{t_3} \diff B^{n_2}_{t_2} \diff B^{n_1}_{t_1}.
\end{equation}
Again by Itô isometry and Plancherel, this implies
\begin{multline}\label{eq:3d wick main lemma delta}
-\frac{\lambda^2}{2} \E \int_{\Tor^d} \int_0^N (J_t \W^3_t)^2 \dt \dx
= -\frac{24^2 \lambda^2}{2} \!\sum_{n_1, n_2, n_3}\!
    \frac{1}{\inormc{n_{123}}^2}
    \prod_{i=1}^3 \frac{1}{\inormc{n_i}^2} \\
    \int_0^N \!\! \int_0^t \! \int_0^{t_1} \!\! \int_0^{t_2}
        \sigma_{t}(n_{123})^2 \prod_{i=1}^3 \sigma_{t_i}(n_i)^2
        \dt_3 \dt_2 \dt_1 \dt.
\end{multline}
We again abbreviate $n_{123} = n_1 + n_2 + n_3$,
and in the negative-mass case sum over non-zero frequencies only.

The integrals
\eqref{eq:3d wick main lemma gamma integrals} and \eqref{eq:3d wick main lemma delta}
differ in the region $t_3 \in {[{t_2}, {N}]}$.
However, in this case we have $\smallabs{n_3} \gtrsim \smallabs{n_2}$
by the support of $\sigma$.
Then we can move $1-\epsilon$ orders of decay from $n_3$ to $n_2$ and estimate
\begin{equation}
\eqref{eq:3d wick main lemma gamma sum}
\lesssim \sum_{n_1, n_2, n_3}
\frac{1}
    {\inormp{n_{12}}^2 \inormp{n_1}^2 \inormp{n_2}^{3-\epsilon} \inormp{n_3}^{3+\epsilon}},
\end{equation}
which is summable by \Cref{thm:discrete convolution}.
Therefore we can restrict both integrals to $t_3 \leq t_2$ in the following.

We also note that if $n_3 = 0$, then $n_{12} = n_{123}$ and the coefficients of $\gamma_+ (C_+ - C_-)$
and $\delta'_+$ match exactly (and again, $\delta'_-$ only contains terms with $n_3 \neq 0$).
With an argument similar to that in \Cref{thm:3d wick gamma difference},
we see that we can also restrict to $n_1, n_2 \neq 0$.

Now we are ready to repeat the telescoping sum argument on $\delta'_N$.
For convenience we set $n_{0}=n_{123}$.
We can write
\begin{align}
    \delta'_{+}-\delta'_{-}&=\sum_{j=0}^3 \delta_{j}.\\
    \delta_{j}&= \sum_{n_1,n_3,n_3} I(n_1,n_2,n_3)
        \Bigg(\prod_{i=0}^j \frac{1}{\inormp{n_i}^2} \Bigg)
        \frac{3}{\inormp{n_j}^2\inormm{n_j}^2}
        \Bigg(\prod_{i=j+1}^3\frac{1}{\inormm{n_i}}\Bigg), \notag
\end{align}
where we have denoted
\begin{equation}
I(n_1,n_2,n_3)=\int_0^N \!\! \int_0^t \! \int_0^{t_1} \!\! \int_0^{t_2}
    \sigma_{t}(n_{123})^2 \prod_{i=1}^3 \sigma_{t_i}(n_i)^2
    \dt_3 \dt_2 \dt_1 \dt.
\end{equation}
Let us first treat $\delta_3$.
It turns out to be divergent, so we have to combine it with $\gamma_{+}(C_+ - C_-)$
expanded in \eqref{eq:3d wick main lemma gamma sum}.
We can write
\begin{equation}\label{eq:3d wick delta3 renorm}
    \delta_3 - \gamma_{+}(C_+ - C_-) = -\frac{24^2 \lambda^2}{2} (A_1 + A_2),
\end{equation}
where
\begin{equation}
\begin{split}
A_1 &= \sum_{n_1, n_2, n_3 \neq 0}\int_0^N \!\!\! \int_0^t \!\! \int_0^{t_1} \!\!\! \int_0^{t_2}
    (\sigma_{t}(n_{12})^2 -\sigma_{t}(n_{123})^2) \prod_{i=1}^3 \sigma_{t_i}(n_i)^2
    \dt_3 \dt_2 \dt_1 \dt \\
&\hspace{4em}\times \frac{\inormm{n_3}^2 - \I_{n_3 \neq 0} \inormp{n_3}^2}
    {\inormp{n_{12}}^2 \inormp{n_1}^2 \inormp{n_2}^2 \inormp{n_3}^2 \inormm{n_3}^2},\\
&\lesssim \sum_{n_1, n_2, n_3}\int_0^N \!\!\! \int_0^t \!\! \int_0^{t_1} \!\!\! \int_0^{t_2}
    \left(\frac{\abs{n_3}}{\inorm t^2}\right) \prod_{i=1}^3 \sigma_{t_i}(n_i)^2
    \dt_3 \dt_2 \dt_1 \dt \\
&\hspace{4em}\times \frac{3}
    {\inormp{n_{12}}^2 \inormp{n_1}^2 \inormp{n_2}^2 \inormp{n_3}^2 \inormm{n_3}^2}.
\end{split}
\end{equation}
Now using the fact that $t_2,t_3\lesssim t$ we can bound this by 
\begin{equation}
\sum_{n_1, n_2} \frac{C_\epsilon}{\inormp{n_{12}}^2 \inormp{n_1}^2 \inormp{n_2}^{3-\epsilon}}
\lesssim \sum_{n_{12}} \frac{1}{\inormp{n_{12}}^4},
\end{equation}
where we applied \Cref{thm:discrete convolution} to get the last inequality.

The remaining term $A_2$ is given by
\begin{equation}
\sum_{n_1,n_2,n_3} I(n_1,n_2,n_3) \frac{3(\inormp{n_{12}} - \inormp{n_{123}})}
    {\inormp{n_{123}}^2 \inormp{n_{12}}^2 \inormp{n_1}^2 \inormp{n_2}^2 \inormp{n_3}^2 \inormm{n_3}^2}.
\end{equation}
Now $\inormp{n_{123}}^2 - \inormp{n_{12}}^2$ can be written
as $\frac 1 2 (n_{12} \cdot n_{123}) + \abs{n_3}^2$.
The inner product vanishes by symmetry of the sum,
and for the second term we use $\smallabs{n_3} \simeq \inorm{n_3}$
and repeated application of \Cref{thm:discrete convolution}:
\begin{equation}
\begin{split}
&\mathrel{\phantom{\lesssim}} \sum_{n_1, n_2, n_3 \neq 0}
    \frac{1}{\inormp{n_{123}}^2 \inormp{n_{12}}^2 \inormp{n_1}^2 \inormp{n_2}^2 \inormp{n_3}^2}\\
&\lesssim \sum_{n_1, n_2 \neq 0}
    \frac{1}{\inormp{n_{12}}^3 \inormp{n_1}^2 \inormp{n_2}^2}\\
&\lesssim \sum_{n_{12} \neq 0}
    \frac{1}{\inormp{n_{12}}^4}.
\end{split}
\end{equation}

The coefficient $-24^2 \lambda^2 / 2$ appearing in \eqref{eq:3d wick delta3 renorm}
gives the final convergence rate $\lambda^2$.
It also appears in the other telescoping sum terms.

Let us then bound $\delta_2$.
This introduces a problem: we have a factor of $\inorm{n_2}^{-4}$,
but \Cref{thm:discrete convolution} only lets us take advantage of $-3+\epsilon$ orders of decay.%
\footnote{This issue is discussed above Lemma~4.2 in \cite{mourrat_construction_2017}.
However, our argument is more related to \cite[Lemma~4.6]{chandra_phase_2022},
where the issue appears due to presence of $J_t^2$ instead of $J_{t,+} - J_{t,-}$.}
We now use the assumption that $t_2 \geq t_3$
and thus $\smallabs{n_2} \gtrsim \smallabs{n_3}$.
We can thus bound the difference by
\begin{align}\label{eq:3d wick moved decay}
\sum_{n_1, n_2, n_3 \neq 0} \frac{1}
{\inorm{n_{123}}^2 \inorm{n_1}^2 \inorm{n_2}^{3+\epsilon} \inorm{n_3}^{3-\epsilon}}
&\lesssim \sum_{n_1, n_2 \neq 0} \frac{1}
{\inorm{n_1}^2 \inorm{n_2}^{3+\epsilon} \inorm{n_{12}}^{2-\epsilon}} \notag\\
&\lesssim \sum_{n_2 \neq 0} \frac{1}
{\inorm{n_2}^{3+\epsilon} \inorm{n_2}^{1-\epsilon}}.
\end{align}
The same argument can be repeated for $\inormp{n_1}$ and $\inormp{n_{123}}$,
which will give bounds on $\delta_0$ and $\delta_1$.
This finishes the proof.
\end{proof}

The remaining two terms of $\delta_+ - \delta_-$ then follow from similar stochastic estimates.
We only sketch the proofs, since the details are not relevant to the discussion.

\begin{lemma}\label{thm:3d wick delta w2 w32}
We can estimate
\[
\abs{\frac{\lambda^3}{2} \E \int_{\Tor^3} \W_{N,+}^{2} (\W_{N,+}^{[3]})^2
    - \W_{N,-}^{2} (\W_{N,-}^{[3]})^2 \dx}
\lesssim \lambda^3
\]
uniformly in $N$.
\end{lemma}
\begin{proof}
Let us first simplify the expression for $\E [\W_{N}^{2} (\W_{N}^{[3]})^2]$.
Up to the order of dots, there is only one way to contract the corresponding diagram:
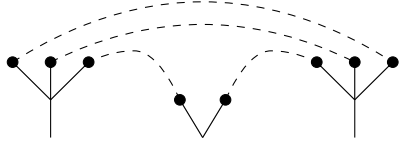
\begin{center}
\begin{tikzpicture}
    \begin{scope}
        \coordinate (A1) at (-0.3,0.5);
        \coordinate (A2) at (0.3,0.5);
        \filldraw (0,0) -- (A1) circle[whitenoise];
        \filldraw (0,0) -- (A2) circle[whitenoise];
    \end{scope}
    \begin{scope}[xshift=-2cm, yshift=0.5cm]
        \coordinate (B1) at (-0.5, 0.5);
        \coordinate (B2) at (0, 0.5);
        \coordinate (B3) at (0.5, 0.5);
        \filldraw (0,0) -- (B1) circle[whitenoise];
        \filldraw (0,0) -- (B2) circle[whitenoise];
        \filldraw (0,0) -- (B3) circle[whitenoise];
        \draw (0,-0.5) -- (0,0);
    \end{scope}
    \begin{scope}[xshift=2cm, yshift=0.5cm]
        \coordinate (C1) at (-0.5, 0.5);
        \coordinate (C2) at (0, 0.5);
        \coordinate (C3) at (0.5, 0.5);
        \filldraw (0,0) -- (C1) circle[whitenoise];
        \filldraw (0,0) -- (C2) circle[whitenoise];
        \filldraw (0,0) -- (C3) circle[whitenoise];
        \draw (0,-0.5) -- (0,0);
    \end{scope}

    \draw[contract] (A1) parabola[bend pos=0.5] bend +(0,0.4) (B3);
    \draw[contract] (A2) parabola[bend pos=0.5] bend +(0,0.4) (C1);
    \draw[contract] (B1) parabola[bend pos=0.5] bend +(0,0.8) (C3);
    \draw[contract] (B2) parabola[bend pos=0.5] bend +(0,0.5) (C2);
\end{tikzpicture}
\end{center}
Therefore
\begin{multline}
\E [\W_{N}^{2} (\W_{N}^{[3]})^2]
= C \sum_{n_1, \ldots, n_4} \frac{1}{\inormc{n_{134}}^2 \inormc{n_{234}}^2
    \inormc{n_1}^2 \inormc{n_2}^2 \inormc{n_3}^2 \inormc{n_4}^2}
    \iiint \!\!\! \iiint \cdots,
\end{multline}
where $C$ is a combinatorial constant.
We omit the integral part since it is independent of the mass
and not used in what follows.
Thus we are left with two cases:
\begin{itemize}
    \item Comparison of different masses when all wavenumbers are non-zero;
    \item Additional positive-mass terms when some of the wavenumbers are zero.
\end{itemize}
The latter case is straightforward to verify with \Cref{thm:discrete convolution}.
Let us present here only the subcase $n_1 = 0$ (which is symmetric with $n_2 = 0$):
\begin{align}
\sum_{n_2, n_3, n_4} \frac{1}{\inormp{n_{34}}^2 \inormp{n_{234}}^2
    \inormp{n_2}^2 \inormp{n_3}^2 \inormp{n_4}^2}
&\leq \sum_{n_3, n_4} \frac{1}{\inormp{n_{34}}^3 \inormp{n_3}^2 \inormp{n_4}^2} \notag\\
&= \sum_{n_{34}, n_4} \frac{1}{\inormp{n_{34}}^3 \inormp{n_{34} - n_4}^2 \inormp{n_4}^2} \notag\\
&\leq \sum_{n_{34}} \frac{1}{\inormp{n_{34}}^4}.
\end{align}

For the change of mass, we again use a telescoping argument.
Let us start with the positive-mass expression and first change $\inormp{n_1}$ to $\inormm{n_1}$,
leading us to estimate
\begin{align}
&\sum_{n_1, \ldots, n_4 \neq 0} \frac{3}{\inormp{n_{134}}^2 \inormp{n_{234}}^2
    \inormp{n_1}^2 \inormm{n_1}^2 \inormp{n_2}^2 \inormp{n_3}^2 \inormp{n_4}^2} \notag\\
\lesssim\; &\sum_{n_1, n_2, n_3, n_{34}}
    \frac{1}{\inormp{n_{134}}^2 \inormp{n_{234}}^2
    \inormp{n_1}^4 \inormp{n_2}^2 \inormp{n_3}^2 \inormp{n_{34} - n_3}^2} \notag\\
\lesssim\; &\sum_{n_1, n_2, n_{34}}
    \frac{1}{\inormp{n_{134}}^2 \inormp{n_{234}}^2
    \inormp{n_1}^4 \inormp{n_2}^2 \inormp{n_{34}}^1} \notag\\
\lesssim\; &\sum_{n_1, n_{34}}
    \frac{1}{\inormp{n_{134}}^2 \inormp{n_1}^4 \inormp{n_{34}}^2} \notag\\
\lesssim\; &\sum_{n_1}
    \frac{1}{\inormp{n_1}^5}.
\end{align}
The change of $\inormc{n_2}$ is symmetric.
The changes of $\inormp{n_3}$ to $\inormm{n_3}$ (resp.\ of $n_4$)
and $\inormp{n_{134}}$ to $\inormm{n_{134}}$ (resp.\ of $n_{234}$)
follow by minor modifications to the order of application of \Cref{thm:discrete convolution}.
In particular, the final sums are always over inverse fifth powers.
Hence we get from $\E [\W_{N,+}^{2} (\W_{N,+}^{[3]})^2]$ to $\E [\W_{N,-}^{2} (\W_{N,-}^{[3]})^2]$
in bounded steps.
\end{proof}

\begin{lemma}\label{thm:3d wick delta w w33}
The $\lambda^4 \W_N (\W_N^{[3]})^3$ term of $\delta_N$
has finite expectation regardless of the mass,
and therefore it vanishes at rate $\beta^{-4}$. 
\end{lemma}
\begin{proof}
Like above, there is essentially only one way to contract the diagram
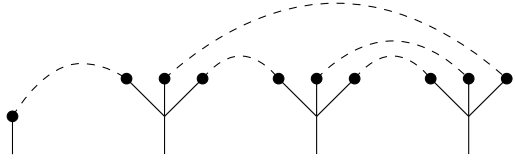
\begin{center}
\begin{tikzpicture}
    \begin{scope}
        \coordinate (A1) at (0,0.5);
        \filldraw (0,0) -- (A1) circle[whitenoise];
    \end{scope}
    \begin{scope}[xshift=2cm, yshift=0.5cm]
        \coordinate (B1) at (-0.5, 0.5);
        \coordinate (B2) at (0, 0.5);
        \coordinate (B3) at (0.5, 0.5);
        \filldraw (0,0) -- (B1) circle[whitenoise];
        \filldraw (0,0) -- (B2) circle[whitenoise];
        \filldraw (0,0) -- (B3) circle[whitenoise];
        \draw (0,-0.5) -- (0,0);
    \end{scope}
    \begin{scope}[xshift=4cm, yshift=0.5cm]
        \coordinate (C1) at (-0.5, 0.5);
        \coordinate (C2) at (0, 0.5);
        \coordinate (C3) at (0.5, 0.5);
        \filldraw (0,0) -- (C1) circle[whitenoise];
        \filldraw (0,0) -- (C2) circle[whitenoise];
        \filldraw (0,0) -- (C3) circle[whitenoise];
        \draw (0,-0.5) -- (0,0);
    \end{scope}
    \begin{scope}[xshift=6cm, yshift=0.5cm]
        \coordinate (D1) at (-0.5, 0.5);
        \coordinate (D2) at (0, 0.5);
        \coordinate (D3) at (0.5, 0.5);
        \filldraw (0,0) -- (D1) circle[whitenoise];
        \filldraw (0,0) -- (D2) circle[whitenoise];
        \filldraw (0,0) -- (D3) circle[whitenoise];
        \draw (0,-0.5) -- (0,0);
    \end{scope}

    \draw[contract] (A1) parabola[bend pos=0.6] bend +(0,0.4) (B1);
    \draw[contract] (B2) parabola[bend pos=0.5] bend +(0,1.0) (D3);
    \draw[contract] (B3) parabola[bend pos=0.5] bend +(0,0.3) (C1);
    \draw[contract] (C2) parabola[bend pos=0.5] bend +(0,0.5) (D2);
    \draw[contract] (C3) parabola[bend pos=0.5] bend +(0,0.3) (D1);
\end{tikzpicture}
\end{center}
This gives rise to the sum
\begin{equation}
\sum_{n_1, \dots, n_5} \frac{1}{
    \inormc{n_1}^2 \inormc{n_2}^{2+\epsilon} \inormc{n_3}^2 \inormc{n_4}^2 \inormc{n_5}^2
    \inormc{n_{123}}^{2-\epsilon} \inormc{n_{245}}^2 \inormc{n_{345}}^2}.
\end{equation}
As in \eqref{eq:3d wick moved decay}, we used the property $\abs{n_{123}} \geq \abs{n_2}$
to move some decay between the two terms.
We define $n_6 = n_4 + n_5$ and sum over $n_1$ and $n_5$ with \Cref{thm:discrete convolution} to get
\begin{equation}
\sum_{n_2, n_3, n_6} \frac{1}{
    \inormc{n_2}^{2+\epsilon} \inormc{n_3}^2 \inormc{n_6}
    \inormc{n_{23}}^{1-\epsilon} \inormc{n_{26}}^2 \inormc{n_{36}}^2}.
\end{equation}
Here the presence of the $n_{23}$ term would complicate further usage of \Cref{thm:discrete convolution},
but we can crudely bound $\inormc{n_{23}}^{-1+\epsilon} \lesssim 1$.
We can then freely sum over $n_2$ and $n_3$ to get the final bound
\begin{equation}
\sum_{n_{45}} \frac{1}{\inormc{n_{45}}^{3+\epsilon}} \lesssim 1.
\end{equation}
As before, in the negative-mass case we sum over non-zero indices only.
\end{proof}

We have now proved that the difference of $\delta_N$ counterterms with different masses vanishes
with sufficiently fast rate.
This was the final remaining difference to estimate.

\bibliographystyle{plainnat}
\bibliography{phi43-metastab-ref}

\end{document}